\documentclass[twoside,reqno]{amsart}
\usepackage{amssymb,amsmath,amsfonts,amsthm}
\usepackage{enumerate}
\usepackage{array}
\usepackage{color}
\usepackage{psfrag,graphicx}
\usepackage[latin1]{inputenc}
\usepackage{tikz}
\usepackage{pgfkeys,pgfplots}

\usepackage[colorlinks=true]{hyperref}
\hypersetup{urlcolor=blue, citecolor=red}

\newtheorem{theorem}{Theorem}
\newtheorem{remark}[theorem]{Remark}

\newtheorem{proposition}[theorem]{Proposition}
\newtheorem{lemma}[theorem]{Lemma}
\newtheorem{corollary}[theorem]{Corollary}

\newcommand{\B}{{\mathbb B}}

\newcommand{\R}{{\mathbb R}}

\newcommand{\HH}{\mathbb H}

\def\dt{\partial_{t}}

\def\div{{\rm div\, }}

\def\u{{\bf u}}
\def\B{{\bf B}}

\def\vphi{\varphi}
\def\eps{\varepsilon}

\def\th{\theta}

\begin{document}

\title[Current-vortex sheets]{Existence of approximate current-vortex sheets near the onset of instability}
\author[A.~Morando] {Alessandro Morando}
\address{DICATAM, Sezione di Matematica, \newline \indent
Universit\`a di Brescia,\newline \indent Via Valotti, 9, 25133 BRESCIA, Italy}
\email{alessandro.morando@unibs.it, paolo.secchi@unibs.it, paola.trebeschi@unibs.it}

\author[P.~Secchi]{Paolo Secchi}

\author[P.~Trebeschi]{Paola Trebeschi}

\date{\today}

\begin{abstract}
The paper is concerned with the free boundary problem for 2D current-vortex sheets in ideal incompressible
magneto-hydrodynamics near the transition point between the linearized stability and instability.
In order to study the dynamics of
the discontinuity near the onset of the instability, Hunter and Thoo \cite{hunter-thoo} have introduced an asymptotic quadratically nonlinear integro-differential equation for the amplitude of small perturbations of the planar discontinuity. The local-in-time existence of smooth solutions to the Cauchy problem for such amplitude equation was already proven in \cite{M-S-T:ONDE1}, under a suitable stability condition. However, the solution found there has a loss of regularity (of order two) from the initial data. In the present paper, we are able to obtain an existence result of solutions with {\it optimal regularity}, in the sense that the regularity of the initial data is preserved in the motion for positive times.
\end{abstract}


\thanks{{The authors are supported by the national research project PRIN 2012 \lq\lq Nonlinear Hyperbolic Partial Differential Equations, Dispersive and Transport Equations: theoretical and applicative aspects\rq\rq.}}

\maketitle

\section{Introduction and main results}\label{int}

Consider the equations of 2-dimensional incompressible magneto-hydrodynamics (MHD)
\begin{equation}
\label{mhd1}
\begin{cases}
\partial_t \u +\nabla \cdot
(\u \otimes \u-\B\otimes \B) +\nabla q =0 \, ,
\\
\partial_t \B -\nabla \times
(\u\times \B) =0 \, ,
\\
\div  \u=0\, ,\;\div \B=0\, \qquad\qquad \text{in }(0,T)\times\R^2,
\end{cases}
\end{equation}
where
$\u=(u_1,u_2)$ denotes the velocity field
and
$\B=(B_1,B_2)$ the magnetic field,
$p$ is the
pressure, $q= p+\frac{1}{2}|\B|^2$ the
total pressure
(for simplicity the density $\rho\equiv1$).

Let us consider {current-vortex sheets} solutions of \eqref{mhd1} (also called \lq\lq tangential discontinuities\rq\rq), that is weak solutions that are smooth on either side of a smooth hypersurface
$$\Gamma(t)=\{y=f(t,x)\}, \qquad  \mbox{where } t\in[0,T], \, (x,y)\in\R^2,\;
$$
and such that at $\Gamma(t)$ satisfy the boundary conditions
\begin{equation}
\label{bc}
\dt f =\u^\pm \cdot N \, ,\quad \B^\pm \cdot N=0 \, ,\quad [q]=0  \, ,
\end{equation}
with $N:=(-\partial_x f, 1)$. In \eqref{bc} $(\u^\pm,\B^\pm,q^\pm)$ denote the values of $(\u,\B,q)$ on the two sides of $\Gamma (t)$, and
$[q]=q^+_{|\Gamma}-q^-_{|\Gamma}$ the jump across
$\Gamma (t)$.

From \eqref{bc} the discontinuity front $\Gamma (t)$
is a tangential discontinuity, namely the plasma does not flow through the discontinuity front and the magnetic field is tangent to $\Gamma (t)$.
The possible jump of the tangential velocity and tangential magnetic field gives a concentration of current and vorticity on the front $\Gamma (t)$.
{Current-vortex sheets are fundamental waves in MHD and play an important role in plasma physics and astrophysics. The existence of current-vortex sheets solutions is known for compressible fluids, see \cite{ChenWang,trakhinin09arma}, see also \cite{cmst,morandotrakhinintrebeschi} for the incompressible case.
%

The necessary and sufficient linear stability condition for planar (constant coefficients) current-vortex
sheets was found a long time ago, see \cite{axford,michael,syrovatskii}.
To introduce it, let us consider a stationary solution of \eqref{mhd1}, \eqref{bc} with interface located at $\{y=0\}$ given by the constant states
\begin{equation}
\begin{array}{ll}\label{constant}
\u^\pm=(U^\pm,0)^T, \qquad \B^\pm=(B^\pm,0)^T
\end{array}
\end{equation}
in the $x$-direction.
The {necessary} and {sufficient stability condition} for the stationary solution is
\begin{equation}
\label{syrovatskii}
|U^+-U^-|^2 < 2 \, \Big( |B^+| ^2+ |B^-|^2 \Big) \, ,
\end{equation}
see \cite{axford,michael,syrovatskii}.
Equality in \eqref{syrovatskii} corresponds to the transition to
{\it violent} instability, i.e. {ill-posedness} of the linearized problem.

Let $U=(U^+,U^-), B=(B^+,B^-)$ and define
\[
\Delta(U,B):= \frac12 \, \Big( |B^+| ^2+ |B^-|^2 \Big) -\frac14|U^+-U^-|^2.
\]
According to \eqref{syrovatskii}, stability/instability occurs when $\Delta(U,B)\gtrless0$.

Hunter and Thoo investigated in \cite{hunter-thoo} the transition to instability when $\Delta(U,B)=0$.
Assume that $U^\pm,B^\pm$ depend on a small positive parameter $\eps$ and
\[
U^\pm=U^\pm_0+\eps U^\pm_1+O(\eps^2), \qquad
B^\pm=B^\pm_0+\eps B^\pm_1+O(\eps^2)
\]
as $\eps\to0^+$, where
\[
\Delta(U_0,B_0)=0\,.
\]
Then
\begin{equation}
\begin{array}{ll}\label{Delta}

\Delta(U,B)=\eps\mu+O(\eps^2)
\end{array}
\end{equation}
as $\eps\to0^+$, where
\[
\mu=B^+_0B^+_1+B^-_0B^-_1 - \frac12\left( U^+_0-U^-_0  \right)\left( U^+_1-U^-_1  \right)
\,.
\]
From \eqref{Delta}, $\mu$ plays the role of a scaled bifurcation parameter:
for small $\eps>0$, if $\mu>0$ the stationary solution \eqref{constant} is linearly stable,
while if $\mu<0$, it is linearly unstable.

It is proved in \cite{hunter-thoo} that the perturbed location of the interface  has the asymptotic expansion
\begin{equation*}
\begin{array}{ll}\label{}
y=f(t,x;\eps)=\eps\vphi( \tau,\th ) + O(\eps^{3/2})  \qquad \mbox{as } \eps\to0^+,
\end{array}
\end{equation*}
where $\tau=\eps^{1/2} t$ is a \lq\lq slow\rq\rq time variable and
$\th=x-\lambda_0 t$  is a new spatial variable in a reference frame moving with the surface wave, $\lambda_0=( U^+_0+U^-_0  )/2$.

As shown in \cite{hunter-thoo}, after a rescaling, and writing again $(t,x)$ for $(\tau,\th)$, the first order term $\vphi$ satisfies the quadratically nonlinear amplitude equation
\begin{equation}\label{onde_integro_diff}
\varphi_{tt}-\mu\varphi_{xx}=\left(\frac12\mathbb H[\phi^2]_{xx}+\phi\varphi_{xx}\right)_{\!\!x}\,,\qquad\phi=\mathbb H[\varphi]\,,
\end{equation}
where the unknown is the scalar function $\varphi=\varphi(t,x)$, whereas $\mathbb H$ denotes the Hilbert transform with respect to $x$.

Equation \eqref{onde_integro_diff} is an integro-differential evolution equation in one space dimension, with quadratic nonlinearity.
This is a nonlocal equation of order two: in fact, it may also be written as
\begin{equation}
\begin{array}{ll}\label{equ1bisbis}
\vphi_{tt}-\mu\vphi_{xx} = \left(  [ \HH;\phi
]\partial_x  \phi
_{x} + \HH[\phi
^2_x]\right)_x \,,
\end{array}
\end{equation}
where $[ \HH;\phi
]\partial_x$ is a pseudo-differential operator of order zero. This alternative form \eqref{equ1bisbis} shows that \eqref{onde_integro_diff} is an equation of second order, due to a cancelation of the third order spatial derivatives appearing in \eqref{onde_integro_diff}.

Equation \eqref{onde_integro_diff} also admits the alternative spatial form
\begin{equation}\label{equ1terter}
\begin{split}
\varphi_{tt}-\left(\mu-2\phi_x\right)\varphi_{xx}+\mathcal Q\left[\varphi\right]=0\,,
\end{split}
\end{equation}
where
\begin{equation}\label{termine_nl0}
\mathcal Q\left[\varphi\right]:=-3\left[\mathbb H\,;\,\phi_x\right]\phi_{xx}-\left[\mathbb H\,;\,\phi\right]\phi_{xxx}\,.
\end{equation}
The alternative form \eqref{equ1terter} puts in evidence the difference $\mu-2\phi_x$ which has a meaningful role. In fact it can be shown that the linearized operator about a given basic state is elliptic and \eqref{onde_integro_diff} is locally linearly ill-posed in points where
\begin{equation}
\label{extstab2}
\mu-2\phi_x<0.
\end{equation}
On the contrary, in points where
\begin{equation}
\begin{array}{ll}\label{extstab}

\mu-2\phi
_{x}>0

\end{array}
\end{equation}
the linearized operator is hyperbolic and \eqref{onde_integro_diff} is locally linearly well-posed, see \cite{hunter-thoo}. In this case we can think of  \eqref{equ1terter} as a nonlinear perturbation of the wave equation.

If $\mu>0$ (linear stability) but condition \eqref{extstab2} holds, we see that the nonlinearity in \eqref{onde_integro_diff} does not regularize the Hadamard instability of the linearized equation, as we could hope. Instead, it has a destabilizing effect. Specifically, under \eqref{extstab2} the equation \eqref{onde_integro_diff} is locally linearly unstable, and an initially linearly stable solution evolves into an unstable regime.

In \cite{hunter-thoo} the reader may also find an heuristic, physical explanation of condition \eqref{extstab2} yielding the linearized ill-posedness, given in terms of a longitudinal strain of the fluid along the discontinuity.
The analysis shows that there is a compression of material particles on the part of the discontinuity where the magnitude of the magnetic field is larger, and expansion of material particles on the opposite part of the discontinuity where the magnitude of the magnetic field is smaller. This typically corresponds to formation of a \lq\lq finger" in the discontinuity that extends into the half-space with the weaker magnetic field. Thus, the loss of stability of the discontinuity is associated with its movement toward the side with the less-stabilizing magnetic field.

It is interesting to observe that the same quadratic operator of \eqref{onde_integro_diff} appears in the first order nonlocal amplitude equation
\begin{equation}\label{amplie}
\varphi_{t}+\frac12\mathbb H[\phi^2]_{xx}+\phi\varphi_{xx}=0\,,\qquad\phi=\mathbb H[\varphi]\,,
\end{equation}
for nonlinear Rayleigh waves  \cite{hamilton-et-al} and surface waves on current-vortex sheets and plasma-vacuum interfaces in incompressible MHD \cite{ali-hunter,ali-hunter-parker,secchi-quart}.
{Equation \eqref{amplie} is considered a canonical model equation for nonlinear surface wave solutions of hyperbolic conservation laws, analogous to the inviscid Burgers equation for bulk waves.}


\bigskip


In this paper we are mainly interested in the nonlinear well-posedness of \eqref{onde_integro_diff} under assumption \eqref{extstab}. More specifically, we will study the local-in-time existence of solutions to the initial value problem for  \eqref{onde_integro_diff}  with sufficiently smooth initial data
\begin{equation}\label{id}
\varphi_{\vert\,t=0}=\varphi^{(0)}\,,\qquad \partial_t\varphi_{\vert\,t=0}=\varphi^{(1)}\,,
\end{equation}
satisfying the following {\it stability condition}
\begin{equation}\label{extstab_id}
\mu-2\phi^{(0)}_{x}>0\,,\qquad\phi^{(0)}:=\mathbb H[\varphi^{(0)}]\,.
\end{equation}

For the sake of convenience, in the paper the unknown $\varphi=\varphi(t,x)$ is a scalar function of  the time $t\in\mathbb R^+$  and the space variable $x$, ranging on the one-dimensional torus $\mathbb T$ (that is $\varphi$ is periodic in $x$). 

Throughout the rest of the paper we will be only interested to seek solutions $\varphi=\varphi(t,x)$ of the equation \eqref{onde_integro_diff} with zero spatial mean, that is $\int_{\mathbb T}\varphi(t,x)\,dx=0$; thus from now on we assume that the initial data $\varphi^{(0)}$, $\varphi^{(1)}$ have zero spatial mean\footnote{It is straightforward to check that every (sufficiently smooth) solution $\varphi=\varphi(t,x)$ to \eqref{onde_integro_diff} keeps spatial mean equal to zero along its time evolution, provided that $\varphi$ and $\partial_t\varphi$ have zero spatial mean at the initial time $t=0$.}. This does not make a restriction of our analysis, since the solvability of the initial value problem for \eqref{onde_integro_diff}, with general initial data, follows at once from the solvability in the case of zero spatial mean initial data.

\begin{remark}\label{rmk:1}
{\rm Let us notice that the periodicity of $\varphi^{(0)}$ implies that $\varphi^{(0)}_x$ has spatial mean equal to zero; since $\varphi^{(0)}$ and $\varphi^{(0)}_x$ are also real-valued then $\phi^{(0)}_x=\mathbb H[\varphi^{(0)}_x]$ is still real-valued with zero spatial mean (see the results collected in the next sections \ref{molt_fourier}, \ref{trasf_hilbert}). Therefore $\phi^{(0)}_x$ (if not identically zero\footnote{Because the spatial mean of $\varphi^{(0)}$ is zero, $\phi^{(0)}_x$ identically zero should imply that $\varphi^{(0)}$ is identically zero too.}) should attain either positive or negative values; consequently inequality \eqref{extstab_id} yields $\mu>0$ (providing {\it linear stability} of \eqref{onde_integro_diff}), while the opposite inequality implies $\mu<0$ (which gives {\it linear instability}). It is therefore somehow natural to regard \eqref{extstab_id} as a {\it stability condition}, under which we investigate the {\it nonlinear well-posedness} of the equation \eqref{onde_integro_diff}.}
\end{remark}

In \cite{M-S-T:ONDE1}, we proved a result of local-in-time existence of the initial value problem for \eqref{onde_integro_diff}, under the following ``uniform counterpart'' of condition \eqref{extstab_id}
\begin{equation}\label{extstab_unif}
\mu-2\phi^{(0)}_{x}\ge\delta\,,\qquad\mbox{on}\,\,\mathbb T\,,
\end{equation}
being $\delta$ some positive constant, see Theorem \ref{nonlin_th} below. For $\phi^{(0)}_x$ has zero spatial mean, arguing as before we first deduce from \eqref{extstab_unif} that $\mu>\delta$. Moreover \eqref{extstab_unif} must be understood as a {\it smallness} assumption on the size of the initial data $\varphi^{(0)}$ in \eqref{id}. One can immediately see that \eqref{extstab_unif} is satisfied whenever the $L^\infty-$norm of $\varphi^{(0)}$ is bounded by
\begin{equation}\label{Linfty}
\Vert \phi^{(0)}_x\Vert_{L^\infty}\le\frac{\mu-\delta}{2}\,.
\end{equation}
On the other hand, from the Sobolev imbedding $H^1(\mathbb T)\hookrightarrow L^\infty(\mathbb T)$) and the Sobolev continuity of the Hilbert transform (cf. Section \ref{trasf_hilbert}), \eqref{Linfty} follows, in its turn, when the $H^1-$norm of $\varphi^{(0)}_x$ is supposed to be sufficiently small, see the next Lemma \ref{lemma_R0} for details.

In \cite{M-S-T:ONDE1} (see Theorem \ref{nonlin_th} in Section \ref{sec_ex_uniq}) we already proved the existence of the solution to the initial value problem for \eqref{onde_integro_diff}. The proof follows from a careful analysis of the linearized problem and application of the Nash-Moser theorem.
However, the result of \cite{M-S-T:ONDE1} is not completely satisfactory because the approach employed there gives a loss of regularity (of order two) from the initial data to the found solution. In the present paper, we are able to provide an existence result for the initial value problem with {\it optimal regularity}, in the sense that the regularity of the initial data is preserved in the motion for positive times; in fact the solution is continuous in time with values in the same function spaces of the initial data.
In a forthcoming paper we will study the continuous dependence in strong norm of solutions on the initial data, so to complete the proof of the well-posedness of \eqref{onde_integro_diff} in the classical sense of Hadamard, after existence and uniqueness.

The main result of the present paper is the following.

\begin{theorem}\label{th_esistenza}
\begin{itemize}
\item [(1)] For $0<\delta<\mu$ given there exist $0<R_0\le1$ and positive constants $C_1$, $C_2$ such that for all $\varphi^{(0)}\in H^{3}(\mathbb T)$, $\varphi^{(1)}\in H^{2}(\mathbb T)$ with zero spatial mean satisfying
\begin{equation}\label{ip_dati_iniziali}
\Vert\varphi^{(0)}_x\Vert^2_{H^2}+\Vert\varphi^{(1)}\Vert^2_{H^2}<R_0^2\,,
\end{equation}
and for
\begin{equation}\label{T0_nuova}
T_0:=C_1\left\{\Vert\varphi^{(0)}_x\Vert^2_{H^2(\mathbb T)}+\Vert\varphi^{(1)}\Vert^2_{H^2(\mathbb T)}\right\}^{-1/2}\,,
\end{equation}
there exists a unique solution $\varphi\in C([0, T_0); H^{3}(\mathbb T))\cap C^1([0, T_0); H^{2}(\mathbb T))$ of the Cauchy problem for \eqref{onde_integro_diff} with initial data $\varphi^{(0)}$, $\varphi^{(1)}$, with zero spatial mean, satisfying
\begin{equation*}
\mu-2\mathbb{H}[\varphi]_x\ge\delta\,,\quad\mbox{on}\,\,[0,T_0)\times\mathbb T
\end{equation*}
and the energy estimate
\begin{equation}\label{stima_energia_2}
\Vert\varphi(t)\Vert_{H^3(\mathbb T)}^2+\Vert\varphi_t(t)\Vert_{H^2(\mathbb T)}^2\le C_2\left\{\Vert\varphi^{(0)}\Vert^2_{H^3(\mathbb T)}+\Vert\varphi^{(1)}\Vert^2_{H^2(\mathbb T)}\right\}\,,\quad\forall\,t\in[0,T_0)\,.
\end{equation}
\item[(2)] If the initial data satisfy $\varphi^{(0)}\in H^{s}(\mathbb T)$, $\varphi^{(1)}\in H^{s-1}(\mathbb T)$, with an arbitrary real $s\ge 3$, and \eqref{ip_dati_iniziali} (with the lower norm $H^2$) then the corresponding solution $\varphi$ belongs to $C([0, T_0); H^{s}(\mathbb T))\cap C^1([0, T_0); H^{s-1}(\mathbb T))$ and satisfies
\begin{equation}\label{stima_energia_s}
\Vert\varphi(t)\Vert_{H^{s}(\mathbb T)}^2+\Vert\varphi_t(t)\Vert_{H^{s-1}(\mathbb T)}^2\le C_s\left\{\Vert\varphi^{(0)}\Vert^2_{H^{s}(\mathbb T)}+\Vert\varphi^{(1)}\Vert^2_{H^{s-1}(\mathbb T)}\right\}\,,\quad\forall\,t\in[0,T_0)\,,
\end{equation}
with a suitable numerical constant $C_s>0$ depending only on $\delta,\mu,s$.
\end{itemize}
\end{theorem}

\begin{remark}
{\rm The smallness assumption \eqref{ip_dati_iniziali} on the initial data $\varphi^{(0)}$, $\varphi^{(1)}$ relates to the stability condition \eqref{extstab_unif} as explained just above: for given $\mu$ and $\delta$, such that $0<\delta<\mu$, $R_0$ in \eqref{ip_dati_iniziali} will be determined in such a way that the latter inequality  implies \eqref{Linfty}, hence the stability condition \eqref{extstab_unif}, see the proof of the next Lemma \ref{lemma_R0}.}
\end{remark}

The proof of Theorem \ref{th_esistenza} relies on a density argument and the existence theorem proved in \cite{M-S-T:ONDE1} (see Theorem \ref{nonlin_th} in Section \ref{sec_ex_uniq}):
given initial data $\varphi^{(0)},\varphi^{(1)}$ we take a sequence of smooth approximating initial data $\varphi^{(0)}_m,\varphi^{(1)}_m$ satisfying the requirements of Theorem \ref{nonlin_th}.  Applying Theorem \ref{nonlin_th} for each $m$
gives the corresponding solutions $\varphi_m$, defined on time intervals $[0,T_m]$, a priori depending on $m$.
We prove a uniform a priori estimate in Sobolev norm for the approximating solutions $\varphi_m$, and show that they can be defined on a {\it common} time interval $[0,T]$.
Then we pass to the limit in $\varphi_m$ for $m\to+\infty$ and obtain the solution $\vphi$ on $[0,T]$.
Finally, we prove the strong continuity in time of the solution $\vphi$.

Of course, the crucial part of the whole proof is the uniform a priori estimate in Sobolev norm for the approximating solutions $\varphi_m$.
For, it is convenient to consider equation \eqref{onde_integro_diff} in the alternative version \eqref{equ1terter}, with the meaningful term $\mu-2\phi_x$ in evidence, in view of the stability condition \eqref{extstab}.
The a priori estimate follows from a careful analysis of the quadratic term $\mathcal Q\left[\varphi\right]$, defined in 	\eqref{termine_nl0}. Studying it in the frequency space, we show that the kernel associated to the quadratic form has symmetry and reality properties that considerably simplify the analysis and, most of all, the quadratic nonlinearity produces a \lq\lq cancelation effect\rq\rq which allow us to prove an estimate of $\mathcal Q\left[\varphi\right]$ by norms of $\vphi$ with the wished for order of regularity.

The paper is organized as follows. After the following Section \ref{prbt} about notations and basic tools,  Section \ref{sec_stimaH3} will be devoted to the derivation of suitable energy estimates for all sufficiently regular solutions, see Theorem \ref{mainHs} below, and Section \ref{sec_uniqueness} will concern the uniqueness of the solution to problem \eqref{onde_integro_diff}, \eqref{id}.
In Section \ref{sec_ex_uniq} we first recall the existence result of \cite{M-S-T:ONDE1}, see Theorem \ref{nonlin_th}; then we show the existence of the solution as in Statement (1) of
Theorem \ref{th_esistenza}, the additional regularity of the solution, notably the strong continuity in time, and at last Statement (2) of
Theorem \ref{th_esistenza}.

In Appendix \ref{stima_commutatore} we prove our commutator estimates involving the Hilbert transform and give other useful estimates.
Appendix \ref{sec_quadratic} is devoted to the analysis of the quadratic term $\mathcal Q[\varphi]$.
In Appendix \ref{Ioffe} we study the lower semi-continuity of the energy, and in Appendix \ref{Lions-Magenes_sect} we recall
some results from abstract functional analysis.


\section{Preliminary results and basic tools}\label{prbt}
\subsection{Notations}\label{not}
Throughout the whole paper, the partial derivative of a function $f(t,x)$ with respect to $t$ or $x$ will be denoted appending to the function the subscript $t$ or $x$ as
\begin{equation*}
f_t:=\frac{\partial f}{\partial t}\,,\qquad f_x:=\frac{\partial f}{\partial x}\,.
\end{equation*}
(The notations $\partial_t f$, $\partial_x f$ will be also used.) Higher order derivatives in $(t,x)$ will be denoted by the repeated indices; for instance $f_{tt}$ and $f_{tx}$ will stand respectively for second order derivatives of $f$ with respect to $t$ twice and $t$, $x$.

Let $\mathbb T$ denote the one-dimensional torus defined as
\begin{equation*}
\mathbb T:=\mathbb R/(2\pi\mathbb Z)\,,
\end{equation*}
that is the set of equivalence classes of real numbers with respect to the equivalence relation $\sim$ defined as
\begin{equation*}
x\sim y\qquad\mbox{if and only if}\qquad x-y\in 2\pi\mathbb Z\,.
\end{equation*}
It is customary to identify functions that are defined on $\mathbb T$ with $2\pi-$periodic functions on $\mathbb R$. According to this convention, it will be usual referring to $f:\mathbb T\rightarrow\mathbb C$ as a ``periodic function''.

All periodic functions $f:\mathbb T\rightarrow\mathbb C$ can be expanded in terms of Fourier series as
\begin{equation*}
f(x)=\sum\limits_{k\in\mathbb Z}\widehat{f}(k)e^{ikx}\,,
\end{equation*}
where $\left\{\widehat{f}(k)\right\}_{k\in\mathbb Z}$ are the Fourier coefficients defined by
\begin{equation}\label{coeff_fourier}
\widehat{f}(k):=\frac1{2\pi}\int_{\mathbb T}f(x)e^{-ikx}\,dx\,,\qquad k\in\mathbb Z\,.
\end{equation}
For $1\le p\le +\infty$, we denote by $L^p(\mathbb T)$ the usual Lebesgue space of exponent $p$ on $\mathbb T$, defined as the set of (equivalence classes of) measurable functions $f:\mathbb T\rightarrow\mathbb C$ such that the norm
\begin{equation*}
\Vert f\Vert_{L^p(\mathbb T)}:=
\begin{cases}\left(\int_{\mathbb T}\vert f(x)\vert^p\,dx\right)^{1/p}\,,&\quad\mbox{if}\,\,p<+\infty\,\\
\mbox{ess sup}_{x\in\mathbb T}\vert f(x)\vert\,,&\quad\mbox{if}\,\,p=+\infty
\end{cases}
\end{equation*}
is finite. We denote
\begin{equation*}
(f,g)_{L^2(\mathbb T)}:=\int_{\mathbb T}f(x)\overline{g(x)}\,dx
\end{equation*}
the inner product of two functions $f, g\in L^2(\mathbb T)$ ($\overline z$ denotes the conjugate of $z\in\mathbb C$).

Let us recall that Parseval's identity
\begin{equation}\label{parseval_1}
\Vert f\Vert_{L^2(\mathbb T)}^2=2\pi\sum\limits_{k\in\mathbb Z}\vert\widehat{f}(k)\vert^2
\end{equation}
holds true for every $f\in L^2(\mathbb T)$.

Let us also recall that, whenever $f$ and $g$ are sufficiently smooth periodic functions on $\mathbb T$, there holds
\begin{equation}\label{convoluzione}
\widehat{f\cdot g}=\widehat{f}\ast\widehat{g}\,,
\end{equation}
where $\widehat f\ast\widehat g$ is the {\it discrete convolution} of the sequences $\widehat f:=\left\{\widehat f(k)\right\}_{k\in\mathbb Z}$ and $\widehat g:=\left\{\widehat g(k)\right\}_{k\in\mathbb Z}$ defined by
\begin{equation}\label{conv}
\widehat f\ast\widehat g(k):=\sum\limits_{\ell}\widehat{f}(k-\ell)\widehat{g}(\ell)\,,\qquad\forall\,k\in\mathbb Z\,.
\end{equation}

For all $s\in\mathbb R$, $H^s(\mathbb T)$ will denote the Sobolev space of order $s$ on $\mathbb T$, defined to be the set of periodic functions\footnote{The word ``function'' is used here, and in the rest of the paper, in a wide sense. To be more precise, one should speak about ``periodic distributions'' on the torus, instead of ``periodic functions'', when dealing with real order Sobolev spaces. However, for the sake of simplicity, here we prefer to avoid the precise framework of distributions. We refer the reader to the monograph \cite{ruzhansky-turunen} for a thorough presentation of the periodic setting.} $f:\mathbb T\rightarrow\mathbb C$ such that
\begin{equation}\label{normaHs}
\Vert f\Vert_{H^s(\mathbb T)}^2:=2\pi\sum\limits_{k\in\mathbb Z}\langle k\rangle^{2s}\vert\widehat{f}(k)\vert^2<+\infty\,,
\end{equation}
where it is set
\begin{equation}\label{bracket}
\langle k\rangle:=(1+\vert k\vert^2)^{1/2}\,.
\end{equation}
The function $\Vert\cdot\Vert_{H^s(\mathbb T)}$ defines a norm on $H^s(\mathbb T)$, associated to the inner product
\begin{equation*}
(f,g)_{H^s(\mathbb T)}:=2\pi\sum\limits_{k\in\mathbb Z}\langle k\rangle^{2s}\widehat{f}(k)\overline{\widehat{g}(k)}\,,
\end{equation*}
which turns $H^s(\mathbb T)$ into a Hilbert space\footnote{Normalizing by $2\pi$ the $H^s-$norm in \eqref{normaHs} is in agreement with Parseval's identity \eqref{parseval_1}; indeed for $s=0$ the norm \eqref{normaHs} in $H^0(\mathbb T)\equiv L^2(\mathbb T)$ reduces exactly to the $L^2-$norm.}.

Because of the relation between differentiation and Fourier coefficients, it is obvious that when $s$ is a positive integer $H^s(\mathbb T)$ reduces to the space of periodic functions $f:\mathbb T\rightarrow\mathbb C$ such that
\begin{equation*}
\partial_x^k f\in L^2(\mathbb T)\,,\quad\mbox{for}\,\,0\le k\le s
\end{equation*}
and
\begin{equation*}
\sum\limits_{k=0}^s\left\Vert \partial_x^k f\right\Vert_{L^2(\mathbb T)}
\end{equation*}
defines a norm in $H^s(\mathbb T)$ equivalent to \eqref{normaHs}\footnote{Even though the functions $f$ involved here depend on $x\in \mathbb T$ alone, the partial derivative notation $\partial_x^k:=\partial_x\dots\partial_x$ ($k$ times) is used just in order to be consistent with the notations adopted in the subsequent sections, where functions will also depend on time.}.

We recall that for every $p\in[1,+\infty]$, $\ell^p$ denotes the space of all complex sequences $\{c(k)\}_{k\in\mathbb Z}$ such that
\begin{equation}\label{normalp}
\begin{split}
&\Vert\{c(k)\}\Vert_{\ell^p}^p:=\sum\limits_{k\in\mathbb Z}\vert c(k)\vert^p<+\infty\,,\,\,\mbox{if}\,\,p<+\infty\,,\\ &\Vert\{c(k)\}\Vert_{\ell^\infty}:=\sup\limits_{k\in\mathbb Z}\vert c(k)\vert<+\infty\,,\,\,\mbox{if}\,\,p=+\infty\,,
\end{split}
\end{equation}
provided with the norm $\Vert\cdot\Vert_{\ell^p}$ defined above.

It is useful to recall that for all functions $f\in H^\tau(\mathbb T)$ the following estimate
\begin{equation}\label{imm_sobolev}
\Vert\{\widehat f(k)\}\Vert_{\ell^1}\le c_\tau\Vert f\Vert_{H^\tau(\mathbb T)}
\end{equation}
holds true  as long as $\tau>1/2$, with some positive constant $c_\tau$ depending only on $\tau$.

In the following, we are mainly concerned with real-valued periodic functions $f:\mathbb T\rightarrow\mathbb R$ with zero spatial mean, that is such that
\begin{equation*}
\int_{\mathbb T}f(x)\,dx=0\,.
\end{equation*}
For such functions the Fourier coefficients \eqref{coeff_fourier} obey the additional constraints
\begin{equation}\label{coeff_fourier_condizioni}
\widehat{f}(0)=0\,,\qquad \overline{\widehat{f}(k)}=\widehat{f}(-k)\,,\,\,\forall\,k\in\mathbb Z\,.
\end{equation}

\vspace{.5cm}
Hereafter, we will deal with spaces of functions that depend even on time $t$. It will be convenient to regard real-valued functions $f=f(t,x)$, depending on time and space, as vector-valued functions of $t$ alone taking values in some Banach space $\mathcal X$ of functions depending on $x\in\mathbb T$. 

For $T>0$ and $j\in\mathbb N$, we denote by $C^j([0,T]; \mathcal X)$ the space of $j$ times continuously differentiable functions $f:[0,T]\rightarrow \mathcal X$.
\subsection{Some reminds on periodic Fourier multipliers}\label{molt_fourier}
For a given sequence of real (or complex) numbers $\{A(k)\}_{k\in\mathbb Z}$, we denote by $A$ the linear operator defined on periodic functions $f:\mathbb T\rightarrow\mathbb C$ by setting
\begin{equation}\label{operatore1}
Af(x):=\sum\limits_{k\in\mathbb Z}A(k)\widehat{f}(k)e^{ikx}\,,\qquad x\in\mathbb T\,,
\end{equation}
or equivalently, on the Fourier side, by its Fourier coefficients
\begin{equation}\label{operatore2}
\widehat{Af}(k)=A(k)\widehat{f}(k)\,,\qquad\forall\,k\in\mathbb Z\,.
\end{equation}
We refer to the sequence $\{A(k)\}_{k\in\mathbb Z}$ as the {\it symbol} of the operator $A$.

The following continuity result (cf. \cite{M-S-T:ONDE1}) will be useful in the sequel.
\begin{proposition}\label{prop_molt_sobolev}
Let the sequence  $\{A(k)\}_{k\in\mathbb Z}$ satisfy the following assumption
\begin{equation}\label{stima_simbolo}
\langle k\rangle^{-m}\vert A(k)\vert\le C\,,\qquad\forall\,k\in\mathbb Z\,,
\end{equation}
with suitable constants $m\in\mathbb R$, $C>0$; then the operator $A$ with symbol $\{A(k)\}_{k\in\mathbb Z}$, defined by \eqref{operatore1}, extends as a linear bounded operator
\begin{equation*}
A:H^{s}(\mathbb T)\rightarrow H^{s-m}(\mathbb T)\,,
\end{equation*}
for all $s\in\mathbb R$; more precisely
\begin{equation*}
\Vert Af\Vert_{H^{s-m}(\mathbb T)}\le C\Vert f\Vert_{H^{s}(\mathbb T)}\,,\qquad\forall\,f\in H^{s}(\mathbb T)\,,
\end{equation*}
where $C$ is the same constant involved in \eqref{stima_simbolo}.
\end{proposition}
We will refer to an operator $A$, under the assumptions of Proposition \ref{prop_molt_sobolev}, as a {\it Fourier multiplier} of order $m$. Such an operator transforms periodic functions with mean zero into functions of the same type, as it is easily seen by observing that
$\widehat{Af}(0)=A(0)\widehat{f}(0)=0,
$ as long as $\widehat f(0)=0$.
\begin{remark}\label{remark_prodotto}
As a straightforward consequence of formulas \eqref{operatore1}, \eqref{operatore2}, it even follows that the composition $AB$ of two Fourier multipliers $A$ and $B$, whose symbols are respectively  $\{A(k)\}_{k\in\mathbb Z}$ and  $\{B(k)\}_{k\in\mathbb Z}$, is again a Fourier multiplier whose symbol is given by  $\{A(k)B(k)\}_{k\in\mathbb Z}$ (the order of $AB$ being the sum of the orders of $A$ and $B$ separately, because of \eqref{stima_simbolo}). We have in particular that $AB=BA$.
\end{remark}

An example of a Fourier multiplier of order one is provided by the $x-$derivative, i.e. $Af=f_x$, since indeed
\begin{equation*}
\widehat{Af}(k)=\widehat{f_x}(k)=ik\widehat{f}(k)\,,\qquad\forall\,k\in\mathbb Z\,.
\end{equation*}
Another relevant example of a Fourier multiplier is considered in the next section.
\subsection{Discrete Hilbert transform}\label{trasf_hilbert}
The discrete Hilbert transform of a periodic function $f:\mathbb T\rightarrow\mathbb C$, denoted by $\mathbb H[f]$, is defined on the Fourier side by setting
\begin{equation}\label{hilbert1}
\widehat{\mathbb H[f]}(k)=-i\,{\rm sgn}\,k\widehat{f}(k)\,,\qquad\forall\,k\in\mathbb Z\,,
\end{equation}
where
\begin{equation}\label{segno}
{\rm sgn}\,k:=
\begin{cases}
1\,,\quad\mbox{if}\,\,k>0\,,\\ 0\,,\quad\mbox{if}\,\,k=0\,,\\ -1\,,\quad\mbox{if}\,\,k<0\,.
\end{cases}
\end{equation}
It is clear that, in view of Proposition \ref{prop_molt_sobolev}, the Hilbert transform provides a Fourier multiplier of order zero, the condition \eqref{stima_simbolo} being satisfied by $A(k)=-i\,{\rm sgn}\,k$ with $m=0$ and $C=1$; then after Proposition \ref{prop_molt_sobolev} we conclude that
\begin{equation*}
\mathbb H:H^s(\mathbb T)\rightarrow H^s(\mathbb T)
\end{equation*}
is a linear bounded operator and
\begin{equation}\label{stima_hilbert}
\Vert\mathbb H[f]\Vert_{H^s(\mathbb T)}\le\Vert f\Vert_{H^s(\mathbb T)}\,,\qquad\forall\,f\in H^s(\mathbb T)
\end{equation}
for all $s\in\mathbb R$.

Here below we collect a few elementary properties of the Hilbert transform that will be useful in the sequel.
\begin{itemize}
\item[1.] The Hilbert transform commutes with the $x-$derivative. It is a particular case of the property recalled in Remark \ref{remark_prodotto};
\item[2.] For all periodic functions $f,g:\mathbb T\rightarrow\mathbb C$ there holds
\begin{equation}\label{prodotto_hilbert}
\mathbb H\left[fg-\mathbb H[f]\mathbb H[g]\right]=f\mathbb H[g]+\mathbb H[f]g\,.
\end{equation}
\item[3.] For every periodic function $f:\mathbb T\rightarrow\mathbb C$, with zero mean, and $k\in\mathbb Z$, the following formulas of calculus hold true \footnote{Notice that, according to the convention ${\rm sgn}\,0=0$ (see \eqref{segno}), the Hilbert transform $\mathbb H[f]$ of any periodic function $f$ on $\mathbb T$ has zero mean. Hence, the first formula in \eqref{calcolo} is not true when the mean of  $f$ is different from zero. In the latter case, that formula should be replaced by $\mathbb H^2[f]=-f +\widehat{f}(0)$.}
\begin{equation}\label{calcolo}
\mathbb H^2[f]=-f\,,\qquad \mathbb H\left[e^{ik\cdot}\right](x)=-i\,{\rm sgn}\,k\,e^{ikx}\,.
\end{equation}
\item[4.] For all periodic functions $f,g\in L^2(\mathbb T)$ there holds
\begin{equation}\label{integraleH}
\left(\mathbb H[f], g\right)_{L^2(\mathbb T)}=\left(f, -\mathbb H[g]\right)_{L^2(\mathbb T)}.
\end{equation}
\item[5.] For all periodic functions $f,g\in L^2(\mathbb T)$ and $h\in L^\infty(\mathbb T)$ there holds
\begin{equation}\label{aggiunto}
\left(\left[h;\mathbb H\right]f , g\right)_{L^2(\mathbb T)}=\left(f ,\left[h;\mathbb H\right]g\right)_{L^2(\mathbb T)},
\end{equation}
where $\left[h;\mathbb H\right]$ denotes the commutator between the multiplication by the function $h$ and the Hilbert transform $\mathbb H$.
\end{itemize}

\section{A priori estimates}\label{sec_stimaH3}

The first step needed to prove Theorem \ref{th_esistenza} from the existence result proved in \cite{M-S-T:ONDE1} (see Theorem \ref{nonlin_th} in Section \ref{sec_ex_uniq}) is the derivation of suitable energy estimates for all sufficiently regular solutions.

\begin{theorem}\label{mainHs}
Let $0<R\le 1$ and $0<\delta<\mu$ be arbitrarily given and $\rho^2$ be defined as
\begin{equation}\label{rho0}
\rho^2:=(1+\mu+C)R^2\,,
\end{equation}
being $C$ a numerical positive constant. Let $T>0$ be chosen such that
\begin{equation}\label{ip:1}
1-C_\delta T\rho\ge\frac12\,,
\end{equation}
where $C_\delta$ is the positive constant depending on $\delta$ defined in \eqref{Cdelta}.
\begin{itemize}
\item[(1)] If $\varphi=\varphi(t,x)$ is a solution on $[0,T]\times\mathbb T$ of equation \eqref{onde_integro_diff}, with zero spatial mean, satisfying
    \begin{equation}\label{regH2}
    \varphi\in C([0,T]; H^{4}(\mathbb T))\cap C^1([0,T]; H^{3}(\mathbb T))\,,
    \end{equation}
    \begin{equation}\label{sign-cond-sol}
    \mu-2\mathbb H[\varphi]_x\ge\frac{\delta}{2}\,,\quad\mbox{on}\,\,[0,T]\times\mathbb T
    \end{equation}
    and
    \begin{equation}\label{ip:3}
    \Vert\varphi_x(0)\Vert_{H^2}^2+\Vert\varphi_t(0)\Vert^2_{H^2}\le R^2\,,
    \end{equation}
    then
    \begin{equation}\label{stima_H2}
    \Vert\varphi_x(t)\Vert^2_{H^2}+\Vert\varphi_t(t)\Vert^2_{H^2}\le\frac{4\rho^2}{\min\left\{1,\frac{\delta}{2}\right\}}\,,\quad\forall\,t\in[0,T]\,.
    \end{equation}
\item[(2)] If $\varphi=\varphi(t,x)$ is a solution on $[0,T]\times\mathbb T$ of equation \eqref{onde_integro_diff}, with zero spatial mean, satisfying
    \begin{equation}\label{regHs}
    \varphi\in C([0,T]; H^{r+2}(\mathbb T))\cap C^1([0,T]; H^{r+1}(\mathbb T))\,,
    \end{equation}
    for an arbitrary $r\ge 2$, and the assumptions \eqref{sign-cond-sol}, \eqref{ip:3} (with the lower norm $H^2$), then for all $t\in[0,T]$
    \begin{equation}\label{stima_Hs}
    \Vert\varphi_x(t)\Vert^2_{H^r}+\Vert\varphi_t(t)\Vert^2_{H^r}\le\frac{1+\mu+C}{\min\left\{1,\frac{\delta}{2}\right\}}e^{\widetilde C_r}\left\{\Vert\varphi_x(0)\Vert^2_{H^r}+\Vert\varphi_t(0)\Vert^2_{H^r}\right\}\,,
    \end{equation}
    where $C, \widetilde{C}_r$ are suitable positive numerical constants.
\end{itemize}
\end{theorem}
\begin{remark}\label{rmk:Hs1}
Let us point out that in fact for every solution to the problem \eqref{onde_integro_diff}, \eqref{id} satisfying the regularity condition in \eqref{regHs}, for an arbitrary $r\ge 2$,
the additional regularity $\varphi\in C^2([0,T]; H^r(\mathbb T))$ comes for free from the preceding by using the equation and the commutator estimates in Section \ref{stima_commutatore}, see Corollary \ref{add_regularity}.
\end{remark}


\begin{proof}
Let $\varphi=\varphi(t,x)$ be a solution of \eqref{onde_integro_diff}, defined on $[0,T]\times\mathbb T$ with $T$ given in \eqref{ip:1}, fulfilling assumptions \eqref{regHs} and \eqref{sign-cond-sol} on $[0,T]$.
In order to exploit the sign condition \eqref{sign-cond-sol}, it is first convenient to rewrite the equation \eqref{onde_integro_diff} (in the equivalent form \eqref{equ1bisbis}) as follows. Starting from \eqref{equ1bisbis}, differentiating in $x$ and using the first identity in \eqref{calcolo} (recall that $\varphi$ has zero spatial mean), we get
\begin{equation*}
\begin{split}
\varphi_{tt}&-\mu\varphi_{xx}-\left(\left[\mathbb H\,;\,\phi\right]\phi_{xx}+\mathbb H\left[\phi^2_x\right]\right)_x\\
&=\varphi_{tt}-\mu\varphi_{xx}-\left[\mathbb H\,;\,\phi_x\right]\phi_{xx}-\left[\mathbb H\,;\,\phi\right]\phi_{xxx}-\mathbb H\left[2\phi_x\phi_{xx}\right]\\
&=\varphi_{tt}-\mu\varphi_{xx}-\left[\mathbb H\,;\,\phi_x\right]\phi_{xx}-\left[\mathbb H\,;\,\phi\right]\phi_{xxx}-\left(2\phi_x\mathbb H\left[\phi_{xx}\right]+2\left[\mathbb H\,;\,\phi_x\right]\phi_{xx}\right)\\
&=\varphi_{tt}-\mu\varphi_{xx}-\left[\mathbb H\,;\,\phi_x\right]\phi_{xx}-\left[\mathbb H\,;\,\phi\right]\phi_{xxx}+2\phi_x\varphi_{xx}-2\left[\mathbb H\,;\,\phi_x\right]\phi_{xx}\\
&=\varphi_{tt}-\left(\mu-2\phi_x\right)\varphi_{xx}-3\left[\mathbb H\,;\,\phi_x\right]\phi_{xx}-\left[\mathbb H\,;\,\phi\right]\phi_{xxx}\\
&=\varphi_{tt}-\left(\mu-2\phi_x\right)\varphi_{xx}+\mathcal Q\left[\varphi\right]\,,
\end{split}
\end{equation*}
where it is set
\begin{equation}\label{termine_nl}
\mathcal Q\left[\varphi\right]:=-3\left[\mathbb H\,;\,\phi_x\right]\phi_{xx}-\left[\mathbb H\,;\,\phi\right]\phi_{xxx}\,,
\end{equation}
and we recall that $\phi=\mathbb H[\varphi]$.
Hence the equation \eqref{onde_integro_diff} takes the form
\begin{equation}\label{equ1ter}
\begin{split}
\varphi_{tt}-\left(\mu-2\phi_x\right)\varphi_{xx}+\mathcal Q\left[\varphi\right]=0\,.
\end{split}
\end{equation}
Finding an estimate for the $H^r-$norm of $\varphi_x$, $\varphi_t$ amounts to provide an estimate of the $L^2-$norm of $\langle\partial_x\rangle^r\varphi_x$, $\langle\partial_x\rangle^r\varphi_t$, where hereafter $\langle\partial_x\rangle^r$ denotes the Fourier multiplier of symbol $\langle k\rangle^r$ that is
\begin{equation}\label{ds}
\widehat{\langle\partial_x\rangle^r u}(k)=\langle k\rangle^r \widehat{u}(k)\,,\quad\forall\,k\in\mathbb Z\,.
\end{equation}
Indeed from \eqref{parseval_1} and \eqref{ds} the identity
\begin{equation}\label{norme_uguali}
\Vert u\Vert_{H^r(\mathbb T)}=\Vert\langle\partial_x\rangle^r u\Vert_{L^2(\mathbb T)}
\end{equation}
follows at once.

In order to get the desired estimate, let us apply the operator $\langle\partial_x\rangle^r$ to the equation \eqref{equ1ter}; then we find
\begin{equation*}
\left(\langle\partial_x\rangle^r\varphi\right)_{tt}-\left(\mu-2\phi_x\right)\left(\langle\partial_x\rangle^r\varphi\right)_{xx}+2\left[\langle\partial_x\rangle^r\,;\,\phi_x\right]\varphi_{xx}
+\langle\partial_x\rangle^r\mathcal Q\left[\varphi\right]=0\,,
\end{equation*}
where $\left[A\,;\,B\right]:=AB-BA$ is the commutator between the operators $A$ and $B$.
In order to shorten the involved expressions, later on we set
\begin{equation}\label{derivatas}
\psi:=\langle\partial_x\rangle^r\varphi\,.
\end{equation}
Then the last equation becomes
\begin{equation}\label{equpsi}
\psi_{tt}-\left(\mu-2\phi_x\right)\psi_{xx}+2\left[\langle\partial_x\rangle^r\,;\,\phi_x\right]\varphi_{xx}
+\langle\partial_x\rangle^r\mathcal Q\left[\varphi\right]=0\,.
\end{equation}
Let us multiply \eqref{equpsi} by $\psi_t$ and integrate in space over $\mathbb T$ to get
\begin{equation}\label{equpsi_int}
\int_{\mathbb T}\psi_{tt}\psi_t\,dx-\int_{\mathbb T}\left(\mu-2\phi_x\right)\psi_{xx}\psi_t\,dx+2\int_{\mathbb T} \left[\langle\partial_x\rangle^r\,;\,\phi_x\right]\varphi_{xx} \psi_t\,dx+\int_{\mathbb T}\langle\partial_x\rangle^r\mathcal Q\left[\varphi\right]\psi_t\,dx=0\,.
\end{equation}
Differentiation in $t$ under the integral sign, Leibniz's formula, integration by parts in $x$ give for the second integral in the left-hand side above
\begin{equation}\label{int1e2}
\begin{split}
&-\int_{\mathbb T}\left(\mu-2\phi_x\right)\psi_{xx}\psi_t\,dx=\int_{\mathbb T}\partial_{x}\left(\left(\mu-2\phi_x\right)\psi_{t}\right)\psi_x\,dx\\
&\quad=-2\int_{\mathbb T}\phi_{xx}\psi_{t}\psi_x\,dx+\int_{\mathbb T}\left(\mu-2\phi_x\right)\psi_{xt}\psi_x\,dx\\
&\quad=-2\int_{\mathbb T}\mathbb H\left[\varphi_{xx}\right]\psi_{t}\psi_x\,dx+\frac12\frac{d}{dt}\int_{\mathbb T}\left(\mu-2\phi_x\right)\vert\psi_{x}\vert^2\,dx+\int_{\mathbb T}\phi_{xt}\vert\psi_{x}\vert^2\,dx\,.
\end{split}
\end{equation}
Substituting \eqref{int1e2} into \eqref{equpsi_int} we get
\begin{equation}\label{equpsi_int_1}
\begin{split}
\frac12\frac{d}{dt}&\left(\Vert\psi_t(t)\Vert^2_{L^2(\mathbb T)}+\int_{\mathbb T}\left(\mu-2\phi_x\right)\vert\psi_{x}\vert^2\,dx\right)\\
&=2\int_{\mathbb T}\mathbb H\left[\varphi_{xx}\right]\psi_{t}\psi_x\,dx-\int_{\mathbb T}\phi_{xt}\vert\psi_{x}\vert^2\,dx-2\int_{\mathbb T} \left[\langle\partial_x\rangle^r\,;\,\phi_x\right]\varphi_{xx} \psi_t\,dx-\int_{\mathbb T}\langle\partial_x\rangle^r\mathcal Q\left[\varphi\right]\psi_t\,dx\\
&=\sum\limits_{k=1}^4 J_k\,.
\end{split}
\end{equation}
Now we estimate separately each of the integrals $J_1$,..., $J_4$ involved in the right-hand side above. In the following, we  denote by the letter $C$ a positive numerical constant, independent of $r$, that may be different from line to line. With $C_r$, we denote a positive numerical constant depending on $r$.

{\it Estimate of $J_1$:} H\"older  inequality (recall that $\psi$ has zero spatial mean), Sobolev's imbedding and the Sobolev continuity of $\mathbb H$ yield
\begin{equation}\label{stima_J1}
\begin{split}
\vert J_1\vert &=\left\vert 2\int_{\mathbb T}\mathbb H\left[\varphi_{xx}\right]\psi_{t}\psi_x\,dx\right\vert\le 2\Vert\mathbb H\left[\varphi_{xx}\right]\Vert_{L^\infty(\mathbb T)}\Vert\psi_x\Vert_{L^2(\mathbb T)}\Vert\psi_t\Vert_{L^2(\mathbb T)}\\
&\le C\Vert\varphi_{xx}\Vert_{H^1(\mathbb T)}\Vert\psi_x\Vert_{L^2(\mathbb T)}\Vert\psi_t\Vert_{L^2(\mathbb T)}\le C\Vert\varphi_{x}\Vert_{H^2(\mathbb T)}\Vert\psi_x\Vert_{L^2(\mathbb T)}\Vert\psi_t\Vert_{L^2(\mathbb T)}\,.
\end{split}
\end{equation}

{\it Estimate of $J_2$:} arguing as for $J_1$ and using Poincar\'e inequality we get
\begin{equation}\label{stima_J2}
\begin{split}
\vert J_2\vert &=\left\vert -\int_{\mathbb T}\phi_{xt}\vert\psi_{x}\vert^2\,dx\right\vert\le \Vert\phi_{xt}\Vert_{L^\infty(\mathbb T)}\Vert\psi_x\Vert^2_{L^2(\mathbb T)}\\
&\le C\Vert\phi_{xt}\Vert_{H^1(\mathbb T)}\Vert\psi_x\Vert^2_{L^2(\mathbb T)}\le C\Vert\phi_{xxt}\Vert_{L^2(\mathbb T)}\Vert\psi_x\Vert^2_{L^2(\mathbb T)}= C\Vert\mathbb H\left[\varphi_{xxt}\right]\Vert_{L^2(\mathbb T)}\Vert\psi_x\Vert^2_{L^2(\mathbb T)}\\
&\le C\Vert\varphi_{xxt}\Vert_{L^2(\mathbb T)}\Vert\psi_x\Vert^2_{L^2(\mathbb T)}\le C\Vert\varphi_{t}\Vert_{H^2(\mathbb T)}\Vert\psi_x\Vert^2_{L^2(\mathbb T)}\,.
\end{split}
\end{equation}

{\it Estimate of $J_3$:} first from Cauchy-Schwarz's inequality we compute
\begin{equation}\label{stima_J3_1}
\vert J_3\vert\le 2\Vert\left[\langle\partial_x\rangle^r\,;\,\phi_x\right]\varphi_{xx}\Vert_{L^2(\mathbb T)} \Vert\psi_t\Vert_{L^2(\mathbb T)}\,.
\end{equation}
In order to estimate the $L^2-$norm of $\left[\langle\partial_x\rangle^r\,;\,\phi_x\right]\varphi_{xx}$, we use \eqref{stima_comm_ds} of Lemma \ref{commutatore_ds} with $v=\phi_x=\mathbb H[\varphi_x]$ and $f=\varphi_{xx}$ together with the $H^r-$continuity of $\mathbb H$; then we find
\begin{equation}\label{stima_J3_2}
\begin{split}
\Vert\left[\langle\partial_x\rangle^r\,;\,\phi_x\right]\varphi_{xx}\Vert_{L^2(\mathbb T)}&\le C_r\left\{\Vert\phi_x\Vert_{H^r(\mathbb T)}\Vert\varphi_{xx}\Vert_{H^1(\mathbb T)}+\Vert\phi_{xx}\Vert_{H^1(\mathbb T)}\Vert\varphi_{xx}\Vert_{H^{r-1}(\mathbb T)}\right\}\\
&\le C_r\Vert\varphi_x\Vert_{H^r(\mathbb T)}\Vert\varphi_x\Vert_{H^2(\mathbb T)}=C_r\Vert\psi_x\Vert_{L^2(\mathbb T)}\Vert\varphi_x\Vert_{H^2(\mathbb T)}\,,
\end{split}
\end{equation}
recall that $\Vert\varphi\Vert_{H^r(\mathbb T)}=\Vert\langle\partial_x\rangle^r\varphi\Vert_{L^2(\mathbb T)}=\Vert\psi\Vert_{L^2(\mathbb T)}$ (cf. \eqref{normaHs}, \eqref{derivatas}). Combining \eqref{stima_J3_1}, \eqref{stima_J3_2} yields
\begin{equation}\label{stima_J3}
\vert J_3\vert\le C_r\Vert\varphi_x\Vert_{H^2(\mathbb T)}\Vert\psi_x\Vert_{L^2(\mathbb T)}\Vert\psi_t\Vert_{L^2(\mathbb T)}\,.
\end{equation}

{\it Estimate of $J_4$:} using H\"older's inequality and Lemma \ref{lemma_stima_quadr} to estimate the $H^r-$norm of $\mathcal Q\left[\varphi\right]$ we get
\begin{equation}\label{stima_J4}
\begin{split}
\vert J_4\vert\le \int_{\mathbb T}\vert\langle\partial_x\rangle^r\mathcal Q\left[\varphi\right]\vert\vert\psi_t\vert\,dx\le \Vert \langle\partial_x\rangle^r\mathcal Q\left[\varphi\right]\Vert_{L^2(\mathbb T)}\Vert\psi_t\Vert_{L^2(\mathbb T)}&=\Vert \mathcal Q\left[\varphi\right]\Vert_{H^r(\mathbb T)}\Vert\psi_t\Vert_{L^2(\mathbb T)}\\
&\le C\Vert\varphi_x\Vert_{H^2(\mathbb T)}\Vert\psi_x\Vert_{L^2(\mathbb T)}\Vert\psi_t\Vert_{L^2(\mathbb T)}\,.
\end{split}
\end{equation}

Gathering estimates \eqref{stima_J1}--\eqref{stima_J4} and using Cauchy--Schwarz inequality, the right-hand side of \eqref{equpsi_int} is estimated as
\begin{equation}\label{stima_J1-4}
\begin{split}
\left\vert\sum\limits_{k=1}^4 J_k\right\vert&\le C_r\{\Vert\varphi_x\Vert_{H^2(\mathbb T)}\Vert\psi_x\Vert_{L^2(\mathbb T)}\Vert\psi_t\Vert_{L^2(\mathbb T)}+\Vert\varphi_{t}\Vert_{H^2(\mathbb T)}\Vert\psi_x\Vert^2_{L^2(\mathbb T)}\}\\
&\le C_r\{\Vert\varphi_x\Vert_{H^2(\mathbb T)}+\Vert\varphi_{t}\Vert_{H^2(\mathbb T)}\}\{\Vert\psi_x\Vert^2_{L^2(\mathbb T)}+\Vert\psi_t\Vert^2_{L^2(\mathbb T)}\}\\
&\le C_r\left\{\Vert\varphi_x\Vert^2_{H^2(\mathbb T)}+\Vert\varphi_{t}\Vert^2_{H^2(\mathbb T)}\right\}^{1/2}\{\Vert\psi_x\Vert^2_{L^2(\mathbb T)}+\Vert\psi_t\Vert^2_{L^2(\mathbb T)}\}\,.
\end{split}
\end{equation}
Hence from \eqref{equpsi_int_1} and \eqref{stima_J1-4} we get
\begin{equation}\label{stima_equpsi}
\frac{d}{dt}\left(\Vert\psi_t(t)\Vert^2_{L^2(\mathbb T)}+\int_{\mathbb T}\left(\mu-2\phi_x\right)\vert\psi_{x}\vert^2\,dx\right)\le C_r\left\{\Vert\varphi_x\Vert^2_{H^2(\mathbb T)}+\Vert\varphi_{t}\Vert^2_{H^2(\mathbb T)}\right\}^{1/2}\{\Vert\psi_x\Vert^2_{L^2(\mathbb T)}+\Vert\psi_t\Vert^2_{L^2(\mathbb T)}\}\,.
\end{equation}

Now we have to treat differently two cases, corresponding to statement (1) ($r=2$) and (2) (arbitrary $r\geq 2$).

\vspace{.5cm}

{\it Case $r=2$.} From \eqref{derivatas} and \eqref{normaHs} with $r=2$, we first observe that
\begin{equation}\label{identita_norme}
\Vert\psi_x\Vert_{L^2(\mathbb T)}=\Vert\langle\partial_x\rangle^2\varphi_x\Vert_{L^2(\mathbb T)}=\Vert\varphi_x\Vert_{H^2(\mathbb T)}\qquad \Vert\psi_t\Vert_{L^2(\mathbb T)}=\Vert\langle\partial_x\rangle^2\varphi_t\Vert_{L^2(\mathbb T)}=\Vert\varphi_t\Vert_{H^2(\mathbb T)}\,,
\end{equation}
hence \eqref{stima_equpsi} becomes
\begin{equation}\label{stima_equpsi_s2}
\frac{d}{dt}\left(\Vert\varphi_t(t)\Vert^2_{H^2(\mathbb T)}+\int_{\mathbb T}\left(\mu-2\phi_x\right)\vert\langle\partial_x\rangle^2\varphi_{x}\vert^2\,dx\right)\le C_2\left\{\Vert\varphi_x\Vert^2_{H^2(\mathbb T)}+\Vert\varphi_{t}\Vert^2_{H^2(\mathbb T)}\right\}^{3/2}\,.
\end{equation}

Using the sign condition \eqref{sign-cond-sol}, we  further estimate  the $H^2-$norm of $\varphi_x$ in the right-hand side above by
\begin{equation}\label{stima_phix}
\begin{split}
\Vert \varphi_x(t)\Vert^2_{H^2(\mathbb T)}& =\Vert\langle\partial_x\rangle^2\varphi_x(t)\Vert^2_{L^2(\mathbb T)}=
\int_{\mathbb T}\vert\langle\partial_x\rangle^2\varphi_x(t)\vert^2\,dx
=\frac2{\delta}\int_{\mathbb T}\frac{\delta}{2}\vert\langle\partial_x\rangle^2\varphi_x(t)\vert^2\,dx\\
&\le \frac2{\delta}\int_{\mathbb T}\left(\mu-2\phi_x(t)\right)\vert\langle\partial_x\rangle^2\varphi_x(t)\vert^2\,dx\,,\quad\mbox{for}\,\,0\le t\le T\,.
\end{split}
\end{equation}
From \eqref{stima_equpsi_s2}, \eqref{stima_phix} and setting
\begin{equation}\label{defv}
y^2(t):=\Vert\varphi_t(t)\Vert^2_{H^2(\mathbb T)}+\int_{\mathbb T}\left(\mu-2\phi_x(t)\right) \vert\langle\partial_x\rangle^2\varphi_x(t)\vert^2\,dx\,,
\end{equation}
we get
\begin{equation}\label{stima_v}
\frac{dy^2(t)}{dt}=2y(t)\frac{dy(t)}{dt}\le C_2\max\left\{1,\frac{2}{\delta}\right\}^{\frac32}\, y^3(t)\,,\quad\mbox{for}\,\,0\le t\le T\,.
\end{equation}
If we assume, without loss of generality,  that $y(t)>0$ on $[0,T]$, from \eqref{stima_v} we get
\begin{equation}\label{dis_y}
\frac{d}{dt}{y(t)}\le C_\delta y^2(t)\,,\quad\mbox{for}\,\,0\le t\le T\,,
\end{equation}
where it is set
\begin{equation}\label{Cdelta}
C_\delta:=\frac{C_2\max\left\{1,\frac{2}{\delta}\right\}^{3/2}}{2}\,,
\end{equation}
and $C_2$ is the purely numerical constant involved in \eqref{stima_equpsi_s2}.

From \eqref{dis_y}, we derive that $y(t)$ satisfies
\begin{equation}\label{stima_y}
y(t)\leq \frac{y(0)}{1-C_\delta ty(0)}\,, \quad \mbox{as long as}\,\, t<\frac{1}{C_\delta y(0)}\,.
\end{equation}
From Sobolev imbedding $H^1(\mathbb T)\hookrightarrow L^\infty(\mathbb T)$, we have (recall that $\mu>0$)
\begin{equation*}
\mu-2\phi_x(0)\leq \mu+ 2\Vert\phi_x(0)\Vert_{L^\infty(\mathbb T)}\leq \mu+ C\Vert\phi_x(0)\Vert_{H^1(\mathbb T)}\,,
\end{equation*}
hence for $0<R\le 1$ satisfying \eqref{ip:3} and in view of \eqref{rho0}
\begin{equation}\label{stimaH3_2}
\begin{split}
y^2(0)&=\int_{\mathbb T}\left(\mu-2\phi_x(0)\right)\vert\langle\partial_x\rangle^2 \varphi_x(0)\vert^2\,dx+\Vert \varphi_t(0)\Vert^2_{H^2(\mathbb T)}\\
&\leq \left(\mu+ C\Vert\phi_x(0)\Vert_{H^1(\mathbb T)}\right)\Vert\varphi_x(0)\Vert^2_{H^2(\mathbb T)} + \Vert\varphi_t(0)\Vert^2_{H^2(\mathbb T)}\\
&\le \left(\mu+ C\Vert\varphi_x(0)\Vert_{H^2(\mathbb T)}\right)\Vert\varphi_x(0)\Vert^2_{H^2(\mathbb T)}+\Vert\varphi_t(0)\Vert^2_{H^2(\mathbb T)}\\
&\le (\mu+ C R)R^2+R^2\le (1+\mu+C)R^2=\rho^2\,.
\end{split}
\end{equation}
Because of \eqref{ip:1}, from \eqref{stima_y} we have
\begin{equation}\label{stima_y1}
y(t)\leq \frac{y(0)}{1-C_\delta T\rho}\leq 2y(0)\le 2\rho\,,\quad\forall\, t\in[0,T]\,.
\end{equation}
From \eqref{stima_phix} we find
\begin{equation}\label{stimaH3_1}
\min\left\{\frac{\delta}{2}, 1\right\} \left\{\Vert\varphi_x(t)\Vert^2_{H^2(\mathbb T)}+ \Vert\varphi_t(t)\Vert^2_{H^2(\mathbb T)}\right\}\leq y^2(t)\,.
\end{equation}
From \eqref{stimaH3_1} and \eqref{stima_y1} we finally find
\begin{equation}\label{stimaH3_4}
\Vert\varphi_x(t)\Vert^2_{H^2(\mathbb T)}+ \Vert\varphi_t(t)\Vert^2_{H^2(\mathbb T)}
\leq\frac{4\rho^2}{\min\{\frac{\delta}{2}, 1\}}\,,
\end{equation}
which provides the estimate \eqref{stima_H2}.
\vspace{.5cm}

{\it Case of general $r\geq 2$.} We come back to estimate \eqref{stima_equpsi}. From the sign condition \eqref{sign-cond-sol}, exactly as in \eqref{stima_phix} we get
\begin{equation}\label{stima_psix}
\Vert \psi_x(t)\Vert^2_{L^2(\mathbb T)}\le \frac2{\delta}\int_{\mathbb T}\left(\mu-2\phi_x(t)\right)\vert\psi_x(t)\vert^2\,dx\,,\quad \forall\,  t\in[0, T]\,.
\end{equation}
From \eqref{stima_equpsi} and \eqref{stima_psix} we obtain that
\begin{equation*}
z(t):=\Vert\psi_t(t)\Vert^2_{L^2(\mathbb T)}+\int_{\mathbb T}\left(\mu-2\phi_x(t)\right)\vert\psi_{x}(t)\vert^2\,dx
\end{equation*}
satisfies
\begin{equation}\label{stima_s2}
\frac{d}{dt} z(t) \leq C_r\max\left\{\frac{2}{\delta},1\right\} \left( \Vert\varphi_x(t)\Vert^2_{H^2(\mathbb T)}+ \Vert\varphi_t(t)\Vert^2_{H^2(\mathbb T)}\right)^{1/2}z(t)\,, \quad \forall\, t\in[0,T]\,.
\end{equation}
By Gr\"onwall's lemma we find
\begin{equation}\label{gronwall1}
\begin{split}
z(t)\leq z(0)e^{C_r\max\{\frac{2}{\delta},1\}\int_0^{T} \left( \Vert\varphi_x(\tau)\Vert^2_{H^2(\mathbb T)}+ \Vert\varphi_t(\tau)\Vert^2_{H^2(\mathbb T)}\right)^{1/2}\,d\tau}\,.
\end{split}
\end{equation}
Then from \eqref{ip:1} and \eqref{stimaH3_4} we have
\begin{equation}\label{stima_int}
\int_0^{T} \left( \Vert\varphi_x(\tau)\Vert^2_{H^2(\mathbb T)}+ \Vert\varphi_t(\tau)\Vert^2_{H^2(\mathbb T)}\right)^{1/2}\,d\tau\le \frac{2\rho}{\left(\min\left\{\frac{\delta}{2}, 1\right\}\right)^{1/2}}T\le \frac{1}{C_\delta\left(\min\left\{\frac{\delta}{2}, 1\right\}\right)^{1/2}}\,,
\end{equation}
which yields, using also \eqref{Cdelta} and the trivial identity $\max\left\{\frac{2}{\delta},1\right\}=\frac1{\min\left\{\frac{\delta}{2},1\right\}}$,
\begin{equation}\label{esponente}
\begin{split}
\max\{\frac{2}{\delta},1\}&\int_0^{T} \left( \Vert\varphi_x(\tau)\Vert^2_{H^2(\mathbb T)}+ \Vert\varphi_t(\tau)\Vert^2_{H^2(\mathbb T)}\right)^{1/2}\,d\tau\le \frac{\max\{\frac{2}{\delta},1\}}{C_\delta\left(\min\left\{\frac{\delta}{2}, 1\right\}\right)^{1/2}}\\
&=\frac{2\max\{\frac{2}{\delta},1\}}{C_2\max\left\{\frac{2}{\delta},1\right\}^{3/2}\left(\min\left\{\frac{\delta}{2}, 1\right\}\right)^{1/2}}=\frac{2}{C_2}\,.
\end{split}
\end{equation}

Noticing also that, being $0<R\le 1$,
\begin{equation}\label{z0}
\begin{split}
z(0)&:=\Vert\psi_t(0)\Vert^2_{L^2(\mathbb T)}+\int_{\mathbb T}\left(\mu-2\phi_x(0)\right)\vert\psi_{x}(0)\vert^2\,dx\leq \Vert\varphi_t(0)\Vert^2_{H^r(\mathbb T)} +\left(\mu+ C \Vert\varphi_x(0)\Vert_{H^1(\mathbb T)}\right)\Vert\varphi_x(0)\Vert^2_{H^r(\mathbb T)}\\
&\le \Vert\varphi_t(0)\Vert^2_{H^r(\mathbb T)} +\left(\mu+C R\right)\Vert\varphi_x(0)\Vert^2_{H^r(\mathbb T)}\\
&\le(1+\mu+C)\left\{\Vert\varphi_x(0)\Vert^2_{H^r(\mathbb T)}+ \Vert\varphi_t(0)\Vert^2_{H^r(\mathbb T)}\right\}\,.
\end{split}
\end{equation}
Using \eqref{esponente} and \eqref{z0} to bound the right-hand side of \eqref{gronwall1} we find
\begin{equation}\label{gronwall2}
z(t)\leq (1+\mu+C)e^{\widetilde C_r}\left\{\Vert\varphi_x(0)\Vert^2_{H^r(\mathbb T)}+ \Vert\varphi_t(0)\Vert^2_{H^r(\mathbb T)}\right\}\,,
\end{equation}
being $\widetilde C_r:=\frac{2C_r}{C_2}$.

We conclude as in the case $r=2$, by noticing that
\begin{equation*}
\min\left\{\frac{\delta}{2}, 1\right\} \left\{\Vert\varphi_x(t)\Vert^2_{H^r(\mathbb T)}+ \Vert\varphi_t(t)\Vert^2_{H^r(\mathbb T)}\right\}\leq z(t)\,,
\end{equation*}
which yields the estimate
\begin{equation}\label{stima_s>2}
\Vert\varphi_x(t)\Vert^2_{H^r(\mathbb T)}+ \Vert\varphi_t(t)\Vert^2_{H^r(\mathbb T)}
\leq \frac{(1+\mu+C)e^{\widetilde C_r}}{\min\left\{\frac{\delta}{2}, 1\right\}}\left\{\Vert\varphi_x(0)\Vert^2_{H^r(\mathbb T)}+ \Vert\varphi_t(0)\Vert^2_{H^r(\mathbb T)}\right\}\,.
\end{equation}
The latter is exactly the estimate \eqref{stima_Hs}.
\end{proof}

\section{Uniqueness of the solution}\label{sec_uniqueness}
We first manage to prove the uniqueness of the solution to the problem \eqref{onde_integro_diff}, \eqref{id}. Let  $\left(\varphi^{(0)},\varphi^{(1)}\right)\in H^3(\mathbb T)\times H^2(\mathbb T)$ satisfy \eqref{ip_dati_iniziali}. Assume that $\varphi$ and $\theta$ are two solutions of \eqref{onde_integro_diff}, \eqref{id} with the same initial data $\left(\varphi^{(0)},\varphi^{(1)}\right)$, defined on the same interval $[0,T]$ such that
 \begin{equation*}
 1-C_\delta T\rho_0\geq \frac{1}{2}
 \end{equation*}
 where
 \begin{equation}\label{rho0_1}
 \rho_0^2=(1+\mu+C)R_0^2,
 \end{equation}
  with zero spatial mean and  satisfying \eqref{sign-cond-sol}. The equation satisfied by the difference $v:=\varphi-\theta$ is
\begin{equation}\label{eq_differenza1}
v_{tt}-(\mu-2\phi_x)v_{xx} +2(\phi_x-\Theta_x)\theta_{xx} +\mathcal Q[\varphi]-\mathcal Q[\theta]=0\,.
\end{equation}
where
\begin{equation*}
\phi:=\mathbb H[\varphi], \quad \Theta:=\mathbb H[\theta]\,
\end{equation*}
and the nonlinear operator $\mathcal Q$ is defined in \eqref{termine_nl}.
Moreover by easy calculations we may restate the nonlinear term $\mathcal Q[\varphi]-\mathcal Q[\theta]$ in terms of the difference $v=\varphi-\theta$ as
\begin{equation*}
\begin{split}
\mathcal Q[\varphi]-\mathcal Q[\theta]&=-3\left[\mathbb H;\phi_x-\Theta_x\right]\phi_{xx}-\left[\mathbb H; \phi-\Theta\right]\phi_{xxx}-3\left[\mathbb H;\Theta_x\right]\left(\phi_{xx}-\Theta_{xx}\right)-\left[\mathbb H; \Theta\right]\left(\phi_{xxx}- \Theta_{xxx}\right)\\
&=-3\left[\mathbb H;V_x\right]\phi_{xx}-\left[\mathbb H;V\right]\phi_{xxx}-3\left[\mathbb H;\Theta_x\right]V_{xx}-\left[\mathbb H; \Theta\right]V_{xxx}\,,
\end{split}
\end{equation*}
where it is set
\begin{equation*}
V:=\mathbb H[v]\,.
\end{equation*}
Hence the equation \eqref{eq_differenza1} becomes
\begin{equation}\label{eq_differenza2}
v_{tt}-(\mu-2\phi_x)v_{xx} +2V_x\theta_{xx}-3\left[\mathbb H;V_x\right]\phi_{xx}-\left[\mathbb H;V\right]\phi_{xxx}-3\left[\mathbb H;\Theta_x\right]V_{xx}-\left[\mathbb H; \Theta\right]V_{xxx}=0\,.
\end{equation}
Following the same calculations in the proof of Theorem \ref{mainHs} (see \eqref{int1e2}), we multiply by $v$ the equation \eqref{eq_differenza2} and integrate by parts to get
\begin{equation}\label{eq_differenza3}
\begin{split}
\frac12\frac{d}{dt}&\left(\Vert v_t\Vert^2_{L^2(\mathbb T)}+\int_{\mathbb T}(\mu-2\phi_x)\vert v_x\vert^2\,dx\right)\\
&=-\int_{\mathbb T}\phi_{xt}\vert v_x\vert^2\,dx+2\int_{\mathbb T}\phi_{xx}v_tv_x\,dx-2\int_{\mathbb T}\theta_{xx}V_xv_t\,dx+3\int_{\mathbb T}\left[\mathbb H;V_x\right]\phi_{xx} v_t\,dx\\
&\quad+\int_{\mathbb T}\left[\mathbb H;V\right]\phi_{xxx}v_t\,dx+3\int_{\mathbb T}\left[\mathbb H;\Theta_x\right]V_{xx}v_t\,dx+\int_{\mathbb T}\left[\mathbb H; \Theta\right]V_{xxx}v_t\,dx=\sum\limits_{k=1}^7\mathcal J_k\,.
\end{split}
\end{equation}
Let us now provide a suitable estimate of each of the integral terms $\mathcal J_k$, $k=1,\dots,7$. We recall that from Theorem \ref{mainHs} we know that
\begin{equation}\label{stima_differenza}
\Vert\varphi_x(t)\Vert^2_{H^{2}(\mathbb T)}+\Vert\varphi_t(t)\Vert^2_{H^2(\mathbb T)}\le M\,,\quad \Vert\theta_x(t)\Vert^2_{H^{2}(\mathbb T)}+\Vert\theta_t(t)\Vert^2_{H^2(\mathbb T)}\le M\,,\quad 0\le t\le T\,,
\end{equation}
where for simplicity we have set
\begin{equation}\label{M0}
M:=\frac{4\rho_0^2}{\min\{\frac{\delta}{2}, 1\}}\,,
\end{equation}
with $\rho_0$ defined from $R_0$ in \eqref{rho0_1}.

The estimates of $\mathcal J_1,\dots,\mathcal J_4$ below follow from a straightforward application of H\"older's inequality, together with the Sobolev imbedding $H^1(\mathbb T)\hookrightarrow L^\infty(\mathbb T)$, the Sobolev continuity of $\mathbb H$ and estimates \eqref{stima_differenza}

{\it Estimate of $\mathcal J_1$:}
\begin{equation}\label{stima_J1_unicita}
\begin{split}
\vert\mathcal J_1\vert &\le \Vert\phi_{xt}\Vert_{L^\infty(\mathbb T)}\Vert v_x\Vert_{L^2(\mathbb T)}^2\le C\Vert\varphi_{xt}\Vert_{H^1(\mathbb T)}\Vert v_x\Vert_{L^2(\mathbb T)}^2\\
&\le C\Vert\varphi_{t}\Vert_{H^2(\mathbb T)}\Vert v_x\Vert_{L^2(\mathbb T)}^2\le C\sqrt{M}\Vert v_x\Vert_{L^2(\mathbb T)}^2\,.
\end{split}
\end{equation}

{\it Estimate of $\mathcal J_2$:}
\begin{equation}\label{stima_J2_unicita}
\begin{split}
\vert\mathcal J_2\vert &\le 2\Vert\phi_{xx}\Vert_{L^\infty(\mathbb T)}\Vert v_x\Vert_{L^2(\mathbb T)}\Vert v_t\Vert_{L^2(\mathbb T)}\le C\Vert\varphi_{x}\Vert_{H^2(\mathbb T)}\Vert v_x\Vert_{L^2(\mathbb T)}\Vert v_t\Vert_{L^2(\mathbb T)}\\
&\le C\sqrt{M}\Vert v_x\Vert_{L^2(\mathbb T)}\Vert v_t\Vert_{L^2(\mathbb T)}\,.
\end{split}
\end{equation}

{\it Estimate of $\mathcal J_3$:}
\begin{equation}\label{stima_J3_unicita}
\begin{split}
\vert\mathcal J_3\vert &\le 2\Vert\theta_{xx}\Vert_{L^\infty(\mathbb T)}\Vert V_x\Vert_{L^2(\mathbb T)}\Vert v_t\Vert_{L^2(\mathbb T)}\le C\Vert\theta_{x}\Vert_{H^2(\mathbb T)}\Vert v_x\Vert_{L^2(\mathbb T)}\Vert v_t\Vert_{L^2(\mathbb T)}\\
&\le C\sqrt{M}\Vert v_x\Vert_{L^2(\mathbb T)}\Vert v_t\Vert_{L^2(\mathbb T)}\,.
\end{split}
\end{equation}

{\it Estimate of $\mathcal J_4$:} writing explicitly the involved commutator and using formulas \eqref{calcolo}, \eqref{integraleH} the integral $\mathcal J_4$ can be restated as
\begin{equation*}
\begin{split}
\mathcal J_4 &=3\int_{\mathbb T}\left[\mathbb H;V_x\right]\phi_{xx} v_t\,dx=3\int_{\mathbb T}\mathbb H[V_x\phi_{xx}]v_t\,dx-3\int_{\mathbb T}V_x\mathbb H[\phi_{xx}]v_t\,dx\\
&=-3\int_{\mathbb T}V_x\phi_{xx}V_t\,dx+3\int_{\mathbb T}V_x\varphi_{xx}v_t\,dx\,.
\end{split}
\end{equation*}
Then we get
\begin{equation}\label{stima_J4_unicita}
\begin{split}
\vert\mathcal J_4\vert &\le 3\Vert\phi_{xx}\Vert_{L^\infty(\mathbb T)}\Vert V_x\Vert_{L^2(\mathbb T)}\Vert V_t\Vert_{L^2(\mathbb T)}+3\Vert\varphi_{xx}\Vert_{L^\infty(\mathbb T)}\Vert V_x\Vert_{L^2(\mathbb T)}\Vert v_t\Vert_{L^2(\mathbb T)}\\
&\le C\Vert\varphi_{x}\Vert_{H^2(\mathbb T)}\Vert v_x\Vert_{L^2(\mathbb T)}\Vert v_t\Vert_{L^2(\mathbb T)}\le C\sqrt{M}\Vert v_x\Vert_{L^2(\mathbb T)}\Vert v_t\Vert_{L^2(\mathbb T)}\,.
\end{split}
\end{equation}

As for the integrals $\mathcal J_5$, $\mathcal J_6$, $\mathcal J_7$ their estimates follow from the commutator estimates collected in lemmas \ref{lemma_comm}, \ref{lemma_comm_ale_2}.

{\it Estimate of $\mathcal J_5$:} applying the first estimate of Lemma \ref{lemma_comm} (with $\sigma=1$), using Poincar\'e's inequality on $v$ (because $v$ has zero spatial mean) and, for the rest, the same arguments employed to prove the estimates of $\mathcal J_1,\dots,\mathcal J_4$ above we get
\begin{equation}\label{stima_J5_unicita}
\begin{split}
\vert\mathcal J_5\vert &\le \Vert\left[\mathbb H\,;\,V\right]\phi_{xxx}\Vert_{L^2(\mathbb T)}\Vert v_t\Vert_{L^2(\mathbb T)}\le C\Vert V\Vert_{H^1(\mathbb T)}\Vert \phi_{xxx}\Vert_{L^2(\mathbb T)}\Vert v_t\Vert_{L^2(\mathbb T)}\\
&\le C\Vert \varphi_{x}\Vert_{H^2(\mathbb T)}\Vert v\Vert_{H^1(\mathbb T)}\Vert v_t\Vert_{L^2(\mathbb T)}\le C\Vert \varphi_{x}\Vert_{H^2(\mathbb T)}\Vert v_x\Vert_{L^2(\mathbb T)}\Vert v_t\Vert_{L^2(\mathbb T)}\\
&\le C\sqrt{M}\Vert v_x\Vert_{L^2(\mathbb T)}\Vert v_t\Vert_{L^2(\mathbb T)}\,.
\end{split}
\end{equation}

{\it Estimate of $\mathcal J_6$:} applying the second estimate of Lemma \ref{lemma_comm} (with $\sigma=1$), as in the case of $\mathcal J_5$ we get
\begin{equation}\label{stima_J6_unicita}
\begin{split}
\vert\mathcal J_6\vert &\le 3\Vert\left[\mathbb H\,;\,\Theta_x\right]V_{xx}\Vert_{L^2(\mathbb T)}\Vert v_t\Vert_{L^2(\mathbb T)}\le C\Vert\Theta_{xx}\Vert_{H^1(\mathbb T)}\Vert V_{x}\Vert_{L^2(\mathbb T)}\Vert v_t\Vert_{L^2(\mathbb T)}\\
&\le C\Vert \theta_{x}\Vert_{H^2(\mathbb T)}\Vert v_x\Vert_{L^2(\mathbb T)}\Vert v_t\Vert_{L^2(\mathbb T)}\le C\sqrt{M}\Vert v_x\Vert_{L^2(\mathbb T)}\Vert v_t\Vert_{L^2(\mathbb T)}\,.
\end{split}
\end{equation}

{\it Estimate of $\mathcal J_7$:} applying the first estimate of Lemma \ref{lemma_comm_ale_2} (with $\sigma=0$ and $p=3$), Poincar\'e's inequality on $v$ and then arguing as in the case of $\mathcal J_5$, $\mathcal J_6$ we get
\begin{equation}\label{stima_J7_unicita}
\begin{split}
\vert\mathcal J_7\vert &\le \Vert\left[\mathbb H\,;\,\Theta\right]V_{xxx}\Vert_{L^2(\mathbb T)}\Vert v_t\Vert_{L^2(\mathbb T)}\le C\Vert\Theta_{xxx}\Vert_{L^2(\mathbb T)}\Vert V\Vert_{H^1(\mathbb T)}\Vert v_t\Vert_{L^2(\mathbb T)}\\
&\le C\Vert \theta_{x}\Vert_{H^2(\mathbb T)}\Vert v\Vert_{H^1(\mathbb T)}\Vert v_t\Vert_{L^2(\mathbb T)}\le C\Vert \theta_{x}\Vert_{H^2(\mathbb T)}\Vert v_x\Vert_{L^2(\mathbb T)}\Vert v_t\Vert_{L^2(\mathbb T)}\\
&\le C\sqrt{M}\Vert v_x\Vert_{L^2(\mathbb T)}\Vert v_t\Vert_{L^2(\mathbb T)}\,.
\end{split}
\end{equation}

Gathering the above estimates \eqref{stima_J1_unicita}-- \eqref{stima_J7_unicita}, from \eqref{eq_differenza3} we derive immediately
\begin{equation}\label{eq_differenza4}
\begin{split}
\frac12\frac{d}{dt}&\left(\Vert v_t\Vert^2_{L^2(\mathbb T)}+\int_{\mathbb T}(\mu-2\phi_x)\vert v_x\vert^2\,dx\right)
\le C\sqrt{M}\left\{\Vert v_x\Vert_{L^2(\mathbb T)}^2+\Vert v_t\Vert_{L^2(\mathbb T)}^2\right\}\,,\quad\mbox{on}\,\,[0,T]\,.
\end{split}
\end{equation}
If we exploit once again the sign condition \eqref{sign-cond-sol} for $\varphi$ to get the estimate
\begin{equation}\label{stima_vx}
\Vert v_x(t)\Vert^2_{L^2(\mathbb T)}\le \frac2{\delta}\int_{\mathbb T}\left(\mu-2\phi_x(t)\right)\vert v_x(t)\vert^2\,dx\,,\quad \forall\,  t\in[0, T]\,,
\end{equation}
from \eqref{eq_differenza4}, \eqref{stima_vx} we find for
\begin{equation*}
w(t):=\Vert v_t\Vert^2_{L^2(\mathbb T)}+\int_{\mathbb T}(\mu-2\phi_x)\vert v_x\vert^2\,dx
\end{equation*}
the following estimate
\begin{equation*}\label{eq_differenza5}
\frac{d}{dt}w(t)\le C\sqrt M w(t)\,,\quad\mbox{on}\,\,[0,T]\,,
\end{equation*}
from which a straightforward application of Gr\"onwall's lemma implies the bound
\begin{equation*}\label{eq_differenza6}
w(t)\le w(0)e^{C\sqrt M t}\le w(0)e^{C\sqrt M T} \,,\quad\mbox{on}\,\,[0,T]\,.
\end{equation*}
Since $v_x(0)=v_t(0)=0$ yields that $w(0)=0$, then
\begin{equation}\label{identita_w}
w(t)\equiv 0\,,\quad\mbox{on}\,\,[0,T]
\end{equation}
follows. Because of the definition of $w$ and the sign condition on $\varphi$ (cf. \eqref{sign-cond-sol}), from \eqref{identita_w} we derive that
\begin{equation*}
v_x(t)=v_t(t)=0\,,\quad\mbox{on}\,\,[0,T]
\end{equation*}
which also gives
\begin{equation*}
v(t)=0\,,\quad\mbox{on}\,\,[0,T]\,,
\end{equation*}
in view of Poincar\'e lemma. This ends the proof of the uniqueness.


\section{Proof of Theorem \ref{th_esistenza}}\label{sec_ex_uniq}
The proof of our main Theorem \ref{th_esistenza} relies on a previous existence result for the same problem \eqref{onde_integro_diff}, \eqref{id}, proved in \cite{M-S-T:ONDE1}. For reader's convenience, we recall below the statement of this result.
\begin{theorem}\label{nonlin_th}
(cf. \cite{M-S-T:ONDE1}) Let $\mu>\delta>0$.
\begin{itemize}
\item[(1)] Assume that $\varphi^{(0)}\in H^{11}(\mathbb T)$, $\varphi^{(1)}\in H^{10}(\mathbb T)$ have zero spatial mean and satisfy
    \begin{equation}\label{sign-cond-id}
    \mu-2\mathbb H[\varphi^{(0)}]_x\ge\delta\,,\quad\mbox{in}\,\,\mathbb T\,.
    \end{equation}
    Then there exists $T>0$, depending only on $\Vert\varphi^{(0)}\Vert_{H^{11}(\mathbb T)}$, $\Vert\varphi^{(1)}\Vert_{H^{10}(\mathbb T)}$, $\mu$, $\delta$, such that the initial value problem \eqref{onde_integro_diff}, \eqref{id} with initial data $\varphi^{(0)}$, $\varphi^{(1)}$ admits a unique solution $\varphi$ on $[0,T]$, with zero spatial mean, satisfying
    \begin{equation*}
    \varphi\in L^2(0, T; H^{9}(\mathbb T))\cap H^1(0, T; H^{8}(\mathbb T))\cap H^2(0, T; H^{7}(\mathbb T))\,,
    \end{equation*}
\begin{equation}\label{sign_cond_nonlin}
\mu- 2\mathbb{H}[\varphi]_x\geq \delta/2 \quad {\rm in}\,\,[0,T]\times\mathbb{T}\, .
    \end{equation}
\item[(2)] If $\nu>10$ and $\varphi^{(0)}\in H^{\nu+1}(\mathbb T)$, $\varphi^{(1)}\in H^{\nu}(\mathbb T)$, with zero spatial mean, satisfy condition \eqref{sign-cond-id} then the solution $\varphi$ of \eqref{onde_integro_diff}, \eqref{id} with initial data $\varphi^{(0)}$, $\varphi^{(1)}$, considered in the statement (1), satisfies
    \begin{equation*}
    \varphi\in L^2(0, T; H^{\nu-1}(\mathbb T))\cap H^1(0, T; H^{\nu-2}(\mathbb T))\cap H^2(0, T; H^{\nu-3}(\mathbb T))\,.
    \end{equation*}
\end{itemize}
\end{theorem}
\begin{remark}\label{differenze_th}
Let us point out that there are significant differences between the well posedness results provided in Theorem \ref{th_esistenza}  and Theorem \ref{nonlin_th} above. Theorem \ref{th_esistenza} decreases  the regularity  required to the solution of the problem \eqref{onde_integro_diff}, \eqref{id},  leading to  space $H^3(\mathbb T)$. Here, we have ``optimal regularity'' of the data, in the sense that there is no loss of  regularity of the solution with respect to the initial data. On the contrary, in the result given by Theorem \ref{nonlin_th} there is a loss of regularity from the initial data to the solution, which comes from the use of  Nash-Moser's iteration method in the proof (see  \cite{M-S-T:ONDE1} for details).
\end{remark}


\subsection{Existence of the solution. Proof of Statement (1).}\label{sec_existence}
Let us prove the statement (1) of Theorem \ref{th_esistenza}. In order to prove the existence of the solution in $C([0,T_0);H^3(\mathbb T))\cap C^1([0,T_0);H^2(\mathbb T)) $, for a suitable $T_0>0$  we will go on with the following scheme:
\begin{itemize}
\item[(1)] we approximate the initial data $(\varphi^{(0)}, \varphi^{(1)})\in H^3(\mathbb T)\times H^2(\mathbb T)$  by sequences of sufficiently smooth data $(\varphi_m^{(0)}, \varphi_m^{(1)})$  so that the local existence result of Theorem \ref{nonlin_th} applies, giving a sequence of solutions $\varphi_m$ defined on suitable intervals $[0,T_m]$.
 \item[(2)] We show that, for each $m$, the solution $\varphi_m$ can be prolonged in time on $[0,T]$, for every $T<T_0$ independent of $m$. Moreover, on $[0,T]$ the solution $\varphi_m$ satisfies an $m$-uniform estimate in $H^3$.
 \item[(3)] Up to a subsequence we pass to the weak-limit in the sequence $\{\varphi_m\}$ and show that such limit $\varphi$ is the solution corresponding to the original data $(\varphi^{(0)}, \varphi^{(1)})\in H^3(\mathbb T)\times H^2(\mathbb T)$ on $[0,T]$. Moreover the weak limit $\varphi$  is strongly continuous on $[0,T]$ as an $H^3-$valued function.
 \item[(4)] For the arbitrariness of $T<T_0$, we get that  $\varphi \in C([0,T_0);H^3(\mathbb T))\cap C^1([0,T_0);H^2(\mathbb T)) $.
 \end{itemize}
 In order to prove Theorem \ref{th_esistenza} we need applying the result of the following technical lemma.
 \begin{lemma}\label{lemma_R0}
 For $0<\delta<\mu$ given there exists $0<R_0\le 1$ such that for every $\psi=\psi(x)\in H^2(\mathbb T)$ satisfying
 \begin{equation}\label{ip_lemma_R0}
 \Vert\psi_x\Vert^2_{H^1}\le\frac{4(1+\mu+C)}{\min\left\{\frac{\delta}{2},1\right\}}R_0^2
 \end{equation}
 then
 \begin{equation}\label{sign-cond-lemma}
 \mu-2\mathbb H[\psi]_x\ge\delta\,,\quad\mbox{on}\,\,\mathbb T\,.
 \end{equation}
 \end{lemma}
\begin{proof}
We first observe that condition \eqref{sign-cond-lemma} is verified if
\begin{equation}\label{norma_Linfinito}
\Vert\mathbb H[\psi]_x\Vert_{L^\infty}\le\frac{\mu-\delta}{2}\,.
\end{equation}
On the other hand, from Sobolev Imbedding Theorem and the Sobolev continuity of $\mathbb H$ we have
\begin{equation*}
\Vert\mathbb H[\psi]_x\Vert_{L^\infty}\le C\Vert H[\psi]_x\Vert_{H^1}\le C\Vert\psi_x\Vert_{H^1}\,.
\end{equation*}
Then \eqref{norma_Linfinito} follows if
\begin{equation*}
\Vert\psi_x\Vert_{H^1}\le\frac{\mu-\delta}{2C}\,.
\end{equation*}
Thus the result follows by choosing $0<R_0\le 1$ such that
\begin{equation*}
\frac{4(1+\mu+C)}{\min\left\{\frac{\delta}{2},1\right\}}R_0^2\le \left(\frac{\mu-\delta}{2C}\right)^2\,.
\end{equation*}
\end{proof}

\vspace{.5cm}

{\it Proof of Theorem \ref{th_esistenza}, Statement (1).}
We proceed to prove the above steps.
For fixed $0<\delta<\mu$, let $0<R_0\le 1$ be defined as in Lemma \ref{lemma_R0} and let $(\varphi^{(0)}, \varphi^{(1)})\in H^3(\mathbb T)\times H^2(\mathbb T)$ satisfy the assumption \eqref{ip_dati_iniziali} with the previously given $R_0>0$. Since we have $\frac{4(1+\mu+C)}{\min\left\{\frac{\delta}{2},1\right\}}\ge 1$, from \eqref{ip_dati_iniziali} and Lemma \ref{lemma_R0} it follows that
\begin{equation*}
\mu-2\mathbb H[\varphi^{(0)}]_x\ge\delta\,,\quad\mbox{in}\,\,\mathbb T\,.
\end{equation*}

Our final goal is to prove the existence of a solution to the problem \eqref{onde_integro_diff}, \eqref{id} on the time interval $[0,T_0)$, with
\begin{equation*}
T_0:=\frac1{2C_\delta(1+\mu+C)^{1/2}\{\Vert\varphi^{(0)}_x\Vert^2_{H^2(\mathbb T)}+\Vert\varphi^{(1)}\Vert^2_{H^2(\mathbb T)}\}^{1/2}}\,,
\end{equation*}
which is of the form \eqref{T0_nuova} with $C_1^{-1}=2C_\delta(1+\mu+C)^{1/2}$.

Let ${R}>0$ be arbitrarily chosen such that
\begin{equation}\label{Rtilde}
\Vert\varphi^{(0)}_x\Vert^2_{H^2(\mathbb T)}+\Vert\varphi^{(1)}\Vert^2_{H^2(\mathbb T)}< R^2<R_0^2\,.
\end{equation}
Let $ T$ be defined by
\begin{equation}\label{T0_bis}
 T:=\frac1{2C_\delta\rho}\,,
\end{equation}
with $\rho$ and $C_\delta$ given by \eqref{rho0} and \eqref{Cdelta}, respectively. Note that $T\nearrow T_0$ as $R\searrow \{\Vert\varphi^{(0)}_x\Vert^2_{H^2(\mathbb T)}+\Vert\varphi^{(1)}\Vert^2_{H^2(\mathbb T)}\}^{1/2}$.

Later on we use the following notation
  \begin{equation}\label{Hinfty}
  H^\infty(\mathbb T):= \bigcap\limits_{m=0}^{+\infty}H^m(\mathbb T).
  \end{equation}
 {\it Step (1): Approximation of the initial data.} By density there exist two sequences $\varphi^{(0)}_m,\, \varphi^{(1)}_m\in H^{\infty}(\mathbb T)$, with zero spatial mean, such that
 \begin{equation}\label{convergenze}
 \begin{split}
 &\varphi^{(0)}_m \rightarrow \varphi^{(0)}\,, \quad \mbox{in}\,\, H^3(\mathbb T)\,,\\
 & \varphi^{(1)}_m \rightarrow \varphi^{(1)}\,\quad \mbox{in}\,\, H^2(\mathbb T)\,;
 \end{split}
 \end{equation}
 without loss of generality, we may even assume that
 \begin{equation}\label{ipotesi_dati_appr}
 \Vert\varphi^{(0)}_{m,x}\Vert^2_{H^2}+\Vert\varphi^{(1)}_{m}\Vert^2_{H^2}<R^2\,,\quad\mbox{for}\,\, m=1,2,\dots\,,
 \end{equation}
 that implies (again by Lemma \ref{lemma_R0} which is still true for $ R< R_0$)
 \begin{equation}\label{sign_cond_m_id}
 \mu-2\mathbb H[\varphi_{m}^{(0)}]_x\ge\delta\,,\quad\mbox{in}\,\,\mathbb T,\quad \mbox{for} \,\, m=1,2,\dots.
 \end{equation}
 For each $m$, let us consider the initial value problem for the equation \eqref{onde_integro_diff}, with initial data $\left(\varphi^{(0)}_m, \varphi^{(1)}_m\right)$, i.e.
 \begin{equation}\label{Pm}
 \begin{cases}
 \mbox{Eqt. \eqref{onde_integro_diff}}\\
 \varphi(0)=\varphi^{(0)}_m\,,\\
 \varphi_t(0)=\varphi^{(1)}_m\,.
 \end{cases}
 \end{equation}
 In particular $\varphi^{(0)}_m\in H^{11}(\mathbb T)$, $\varphi^{(1)}_m\in H^{10}(\mathbb T)$ and \eqref{sign-cond-id} holds true (cf. \eqref{sign_cond_m_id}); hence, in view of Theorem \ref{nonlin_th}, problem \eqref{Pm} admits a unique solution $\varphi_m$ on a time interval $[0,T_m]$, where $T_m=T(\mu,\delta, R_m)>0$ depends on $\mu$, $\delta$ and $R_m$ such that
 \begin{equation}\label{Rm}
 \Vert\varphi^{(0)}_m \Vert^2_{H^{11}(\mathbb T)} + \Vert\varphi^{(1)}_m \Vert^2_{H^{10}(\mathbb T)}\leq R^2_m
 \end{equation}
 (namely $R_m$ is the radius of a ball centered in the origin in the space  $H^{11}(\mathbb T)\times H^{10}(\mathbb T)$), and
 \begin{equation}\label{sign-cond-m}
 \mu-2\mathbb H[\varphi_{m}]_x\ge\frac{\delta}{2}\,,\quad\mbox{in}\,\,[0,T_m]\times\mathbb T\,.
 \end{equation}
 Because of the $H^\infty$-smoothness of the initial data $\left(\varphi^{(0)}_m, \varphi^{(1)}_m\right)$, we get that $\left(\varphi^{(0)}_m, \varphi^{(1)}_m\right)\in H^{\nu+1}(\mathbb T)\times H^\nu(\mathbb T)$ for all $\nu$, hence from Theorem \ref{nonlin_th} Statement $(2)$, the solution $\varphi_m$ satisfies
\begin{equation}\label{reg_sol}
\varphi_m\in H^2(0,T_m;H^\infty(\mathbb T))\,.
\end{equation}
Notice also that in view of Sobolev Imbedding Theorem
\begin{equation}\label{reg_sol1}
\varphi_m\in C([0,T_m];H^{12}(\mathbb T))\cap C^1([0,T_m]; H^{11}(\mathbb T))
\end{equation}
and from Remark \ref{rmk:Hs1}
\begin{equation}\label{reg_sol2}
\varphi_m\in C^2([0,T_m];H^{10}(\mathbb T))\,.
\end{equation}
Since the sequences $\left(\varphi^{(0)}_m, \varphi^{(1)}_m\right)$ can not be uniformly bounded in $H^{11}(\mathbb T)\times H^{10}(\mathbb T)$ with respect to m (that is $R_m \nearrow +\infty$, as $m\rightarrow +\infty$) and  $T_m$ depends decreasingly on $R_m$, in principle it could happen that $T_m\searrow 0$, as $m\rightarrow +\infty$.
Our next goal, is to prove that each solution $\varphi_m$ can be actually extended up to the fixed time $T$ defined by \eqref{T0_bis} (independent of $m$).

For $m$ arbitrarily fixed, let us assume that
\begin{equation}\label{ip_Tm}
1-C_\delta T_m\rho\ge\frac12\,.
\end{equation}
If the contrary were true, then we would infer that $T_m>\frac1{2C_\delta\rho}$, which should mean that for the considered index $m$ the solution $\varphi_m$ already exists up to the time $T$ (which is our goal).

If \eqref{ip_Tm} is true, then we are in the position to apply to $\varphi_m$ on $[0,T_m]$ the result of Theorem \ref{mainHs} statement (1). Indeed all the assumptions \eqref{ip:1} (with $T_m$ instead of $T$, cf. \eqref{ip_Tm}), \eqref{regH2} (cf. \eqref{reg_sol1}, \eqref{reg_sol2}), \eqref{sign-cond-sol} (cf. \eqref{sign-cond-m}) and \eqref{ip:3} (cf. \eqref{ipotesi_dati_appr}) are satisfied. Then from Theorem \ref{mainHs} we find that $\varphi_m$ fulfils
\begin{equation*}\label{stima_H2_m}
\Vert\varphi_{m,x}(t)\Vert_{H^2}^2+\Vert\varphi_{m,t}(t)\Vert_{H^2}^2\le\frac{4\rho^2}{\min\left\{1,\frac{\delta}{2}\right\}}\,,\quad\forall\,t\in[0,T_m]\,.
\end{equation*}
In particular this implies (recalling the definition of $\rho$ (cf. \eqref{rho0})
\begin{equation*}\label{stima_H1_m}
\Vert\varphi_{m,x}(t)\Vert_{H^1}^2\le\frac{4(1+\mu+C) R^2}{\min\left\{1,\frac{\delta}{2}\right\}}\,,\quad\forall\,t\in[0,T_m]\,,
\end{equation*}
hence, by Lemma \ref{lemma_R0},
\begin{equation}\label{sign-cond-delta-m}
\mu-2\mathbb H[\varphi_m]_x\ge\delta\,,\quad\mbox{in}\,\,[0,T_m]\times\mathbb T\,.
\end{equation}
Under the same assumptions \eqref{ip:1}--\eqref{ip:3}, Theorem \ref{mainHs} statement (2) yields that $\varphi_m$ fulfils estimate \eqref{stima_Hs} with $r=10$, that is
\begin{equation}\label{stima_Hr_10}
    \Vert\varphi_{m,x}(t)\Vert^2_{H^{10}}+\Vert\varphi_{m,t}(t)\Vert^2_{H^{10}}\le\frac{1+\mu+C}{\min\left\{1,\frac{\delta}{2}\right\}}e^{\widetilde C_{10}}\left\{\Vert\varphi_{m,x}^{(0)}\Vert^2_{H^{10}}+\Vert\varphi_m^{(1)}\Vert^2_{H^{10}}\right\}\le\mathcal C_{\mu,\delta}^2 R_m^2\,,\quad\forall\,t\in[0,T_m]\,,
    \end{equation}
where $R_m$ is defined in \eqref{Rm} and
\begin{equation}\label{Cmudelta}
\mathcal C_{\mu,\delta}^2:=\frac{1+\mu+C}{\min\left\{1,\frac{\delta}{2}\right\}}e^{\widetilde C_{10}}>1\,,
\end{equation}
is independent of $m$.

In view of \eqref{ip_Tm}, $T_m$ is supposed to be less than $T$. So we need to continue the solution beyond $T_m$.
Our next goal is to prove that, in fact, for each $m$ the solution $\varphi_m$ can be extended up to $T$.

{\it Step (2): Time extension of the solution.} In order to perform such an extension, we first extend the solution $\varphi_m$ beyond $T_m$; to do so, we consider the problem
\begin{equation}\label{P'm}
 \begin{cases}
 \mbox{Eqt. \eqref{onde_integro_diff}}\\
 \varphi(T_m)=\varphi_m(T_m)\,,\\
 \varphi_t(T_m)=\varphi_{m, t}(T_m)
 \end{cases}
 \end{equation}
to which we would like to apply Theorem \ref{nonlin_th}.

To begin with we need to check that the values $\varphi_m(T_m)$, $\varphi_{m, t}(T_m)$ involved in the initial conditions in \eqref{P'm} are well-defined respectively in $H^{11}(\mathbb T)$ and $H^{10}(\mathbb T)$.
To do so it is sufficient to observe that from \eqref{reg_sol1} we derive
\begin{equation}\label{reg_sol3}
\varphi_m(T_m)\in H^{12}(\mathbb T)\hookrightarrow H^{11}(\mathbb T)\,,\qquad \varphi_{m, t}(T_m)\in H^{11}(\mathbb T)\hookrightarrow H^{10}(\mathbb T)\,.
\end{equation}
From \eqref{sign-cond-delta-m} computed at $t= T_m$ we also deduce that the initial data $\varphi_m(T_m)$ satisfy the sign condition
 \begin{equation}\label{sign_cond_m_nuova}
 \mu-2\mathbb H[\varphi_{m}(T_m)]_x\geq\delta\,, \quad \mbox{on}\,\,\mathbb T\,.
\end{equation}
In the statement of Theorem \ref{nonlin_th} the time existence depends on the initial data through the radius of a ball centered at the origin in the space of the initial data $H^{11}(\mathbb T)\times H^{10}(\mathbb T)$. This radius is determined from the a priori estimate \eqref{stima_Hr_10} at $t=T_m$ which gives
\begin{equation}\label{stima_Tm}
\begin{split}
\Vert\varphi_{m,x}(T_m)\Vert^2_{H^{10}(\mathbb T)}+ \Vert\varphi_{m,t}(T_m)\Vert^2_{H^{10}(\mathbb T)}
&\leq \mathcal C^2_{\mu,\delta} R_m^2\,,
\end{split}
\end{equation}
where $R_m$ is defined in \eqref{Rm} and $\mathcal C_{\mu,\delta}$, defined in \eqref{Cmudelta}, is independent of $m$.
We have found that all the assumptions in Theorem \ref{nonlin_th} are verified by the initial data of problem \eqref{P'm}. Hence applying Theorem \ref{nonlin_th}, we find that \eqref{P'm} admits a solution $\widetilde{\varphi}_m\in H^2(T_m, T_m+T^\prime_m; H^\infty(\mathbb T))$, defined on a time interval $[T_m, T_m+T^\prime_m]$ where $T^\prime_m=T^\prime_m(\mu, \delta, \mathcal C_{\mu,\delta}R_m)$, such that
\begin{equation*}
\mu-2\mathbb H[\widetilde{\varphi}_m]_x\ge\frac{\delta}{2}\,,\quad\mbox{in}\,\,[T_m,T_m+T^\prime_m]\times\mathbb T\,.
\end{equation*}
By construction the function $\widetilde{\varphi}_m$ provides an extension of the solution $\varphi_m$ of problem \eqref{Pm} to the interval $[0, T_m+T^\prime_m]$; let us continue to denote by $\varphi_m$ this extension, so that now
\begin{equation}\label{reg_sol4}
\varphi_m\in H^2(0,T_m+T^\prime_m; H^\infty(\mathbb T))
\end{equation}
and
\begin{equation*}
\mu-2\mathbb H[\varphi_m]_x\ge\frac{\delta}{2}\,,\quad\mbox{in}\,\,[0,T_m+T^\prime_m]\times\mathbb T\,.
\end{equation*}

Let us assume now that
\begin{equation*}
1-C_\delta(T_m+T^{\prime}_m)\rho\ge\frac12
\end{equation*}
(if this would not be the case, as before we should have $T_m+T^{\prime}_m> T$ then $\varphi_m$ would be defined up to $T$ and we would be done).

Since as before the regularity in \eqref{reg_sol4} implies
\begin{equation*}
\varphi_m\in C([0,T_m+T^{\prime}_m]; H^{12}(\mathbb T))\cap C^1([0,T_m+T^{\prime}_m]; H^{11}(\mathbb T))\cap C^2([0,T_m+T^{\prime}_m]; H^{10}(\mathbb T))
\end{equation*}
and recalling that $\varphi^{(0)}_m$, $\varphi^{(1)}_m$ satisfy \eqref{ipotesi_dati_appr}, we find again that the solution $\varphi_m$ of problem \eqref{Pm} verifies all the assumptions \eqref{ip:1}--\eqref{ip:3} of Theorem \ref{mainHs} on the time interval $[0,T_m+T^{\prime}_m]$. Thus we obtain that
\begin{equation*}
\Vert\varphi_{m,x}(t)\Vert_{H^2}^2+\Vert\varphi_{m,t}(t)\Vert_{H^2}^2\le\frac{4\rho^2}{\min\left\{1,\frac{\delta}{2}\right\}}\,,\quad\forall\,t\in[0,T_m+T^{\prime}_m]\,.
\end{equation*}
In particular this implies (recalling the definition of $\rho$ in \eqref{rho0})
\begin{equation*}
\Vert\varphi_{m,x}(t)\Vert_{H^1}^2\le\frac{4(1+\mu+C)R^2}{\min\left\{1,\frac{\delta}{2}\right\}}\,,\quad\forall\,t\in[0,T_m+T^{\prime}_m]\,,
\end{equation*}
hence, by Lemma \ref{lemma_R0},
\begin{equation*}
\mu-2\mathbb H[\varphi_m]_x\ge\delta\,,\quad\mbox{in}\,\,[0,T_m+T^{\prime}_m]\times\mathbb T\,.
\end{equation*}
From the additional regularity of $\varphi_m$ again by Theorem \ref{mainHs} we also find the estimate
\begin{equation}\label{stima_Hr_10_nuova}
    \Vert\varphi_{m,x}(t)\Vert^2_{H^{10}}+\Vert\varphi_{m,t}(t)\Vert^2_{H^{10}}\le\frac{1+\mu+C}{\min\left\{1,\frac{\delta}{2}\right\}}e^{\widetilde C_{10}}\left\{\Vert\varphi_{m,x}^{(0)}\Vert^2_{H^{10}}+\Vert\varphi_m^{(1)}\Vert^2_{H^{10}}\right\}\le\mathcal C_{\mu,\delta}^2 R_m^2\,,\quad\forall\,t\in[0,T_m+T^{\prime}_m]\,,
\end{equation}
where $R_m$ is defined in \eqref{Rm} and $\mathcal C_{\mu,\delta}$ in \eqref{Cmudelta}.

If $T_m+T^\prime_m<T$, we would like to extend again the solution $\varphi_m$ beyond $T_m+T^\prime_m$. Hence we consider again the ``updated'' problem
\begin{equation}\label{P''m}
\begin{cases}
\mbox{Eqt. \eqref{onde_integro_diff}}\\
\varphi(T_m+T^\prime_m)=\varphi_m(T_m+T^\prime_m)\,,\\
\varphi_t(T_m+T^\prime_m)=\varphi_{m, t}(T_m+T^\prime_m)\,,
\end{cases}
\end{equation}
where we observe that the new initial data $\varphi_m(T_m+T^\prime_m)$, $\varphi_{m, t}(T_m+T^\prime_m)$ satisfy the same assumptions \eqref{reg_sol3}, \eqref{sign_cond_m_nuova} as the older ones $\varphi_m(T_m)$, $\varphi_{m, t}(T_m)$.

So we may apply once again Theorem \ref{nonlin_th} to derive that problem \eqref{P''m} admits a solution defined on a time interval $[T_m+T^\prime_m, T_m+T^\prime_m+T^{\prime\prime}_m]$, where again $T^{\prime\prime}_m$ depends on $\mu$, $\delta$ and the radius of a ball centered at the origin in the space $H^{11}(\mathbb T)\times H^{10}(\mathbb T)$ containing the initial data $(\varphi_m(T_m+T^\prime_m), \varphi_{m,t}(T_m+T^\prime_m)$. Such a solution can be used to further extend the solution $\varphi_m$ of problem \eqref{Pm} (see \eqref{reg_sol4}) to a function (again denoted $\varphi_m$) defined on $[0, T_m+T^\prime_m+T^{\prime\prime}_m]$.

The key point is that the radius of the ball in $H^{11}(\mathbb T)\times H^{10}(\mathbb T)$ containing the initial data $(\varphi_m(T_m+T^\prime_m),\varphi_{m, t}(T_m+T^\prime_m))$ can be chosen to be exactly the same as for the older data $(\varphi_m(T_m), \varphi_{m, t}(T_m))$, which was $\mathcal C_{\mu,\delta} R_m$ (see \eqref{stima_Tm}), this implies that $T^\prime_m=T^{\prime\prime}_m$. This claim comes from \eqref{stima_Hr_10_nuova} at $t=T_m+T^\prime_m$.
Thus the solution $\varphi_m$ to problem \eqref{Pm} exists up to $T_m+2T^\prime_m$. Now, it is clear that the last argument above can be iterated several times; at each new iteration, the older solution $\varphi_m$ is prolonged on a time interval of the same amplitude as the previous one. Then after a finite number of steps we cover the interval $[0, T]$. To be more precise, for each $m$ the number of needed steps to reach the time $T$ is computed by
\begin{equation*}
k_m:=2+\left[\frac{T-T_m}{T^\prime_m}\right],
\end{equation*}
where the square brackets denote the integer part.

\begin{remark}
Let us observe that, for each $m$ fixed, in order to perform the iterative procedure based on the application of theorems \ref{nonlin_th} and \ref{mainHs}, which yields that the solution $\varphi_m$ is extendable up to $T$, we need only to require a finite regularity for the approximating initial data $\varphi^{(0)}_m, \varphi^{(1)}_m$. Notice that this finite regularity must be computed by taking into account  that, at each step of the iterative procedure described before,  a finite loss of regularity from the updated data appears, and one needs to conclude the procedure with initial data in $H^{11}(\mathbb T)\times H^{10}(\mathbb T)$.

Notice that however, for $m\rightarrow +\infty$  we can have $k_m\rightarrow +\infty$. Therefore, in order to deal with the case of $m$ arbitrarily large,  the $H^\infty$-regularity is required.
\end{remark}

To conclude, we have proved that for each integer $m\ge 1$ the solution $\varphi_m$ of problem \eqref{Pm} exists at least up to the time $T$. By construction, for every $m$ the solution $\varphi_m\in H^2(0,T; H^\infty(\mathbb T))$ enjoys the following additional regularity
\begin{equation}\label{reg_phim}
\varphi_m\in C([0,T]; H^{r+2}(\mathbb T))\cap C^1([0, T]; H^{r+1}(\mathbb T))\cap C^2([0, T]; H^r(\mathbb T))\,,\quad\forall\,r\ge 2
\end{equation}
and satisfies the assumptions \eqref{sign-cond-sol} and \eqref{ip:3} on $[0, T]$. Hence it fulfils the a priori estimate \eqref{stima_H2}, namely
\begin{equation}\label{stima_2m}
\begin{split}
\Vert\varphi_{m,\,x}(t)\Vert^2_{H^2(\mathbb T)}+ \Vert\varphi_{m,\,t}(t)\Vert^2_{H^2(\mathbb T)}
&\leq \frac{4\rho^2}{\min\left\{1,\frac{\delta}{2}\right\}}\,,\quad\forall\,t\in[0, T]\,,\quad m=1,2,\dots\,,
\end{split}
\end{equation}
from which, in view of Poincar\'e's inequality, we also derive
\begin{equation}\label{stima_3m}
\begin{split}
\Vert\varphi_{m}(t)\Vert^2_{H^3(\mathbb T)}+ \Vert\varphi_{m,\,t}(t)\Vert^2_{H^2(\mathbb T)}
&\leq  M\,,\quad\forall\,t\in[0,  T]\,,\,\,m=1,2,\dots\,,
\end{split}
\end{equation}
where $ M$ is independent of $m$. Moreover, by Lemma \ref{lemma_R0} we have
\begin{equation}\label{131}
\mu-2\mathbb H[\varphi_m]_x\ge\delta\,,\quad\mbox{in}\,\,[0,T]\times\mathbb T\,.
\end{equation}
The next step is the passage to the limit as $m\to +\infty$ in the sequence $\{\varphi_m\}$.


\vspace{.2cm}
\noindent
{\it Step (3): Passage to the limit}. We prove the passage to the limit in several steps.

\vspace{.15cm}
{\it Step (3.1): Weak $\ast$ limit}. Estimate \eqref{stima_3m} tells us that the sequence $\{\varphi_m\}$ is bounded in $L^\infty(0, T; H^3(\mathbb T))\cap W^{1,\infty}(0, T; H^2(\mathbb T))$; hence, we can extract from $\{\varphi_m\}$ a subsequence (still denoted with the same symbol) such that
\begin{equation}\label{conv_star}
\begin{split}
&\varphi_m\rightharpoonup^\ast\varphi\,,\quad\mbox{in}\,\,L^\infty(0, T; H^3(\mathbb T))\,,\\
&\varphi_{m\,, t}\rightharpoonup^\ast\varphi_t\,,\quad\mbox{in}\,\,L^\infty(0, T; H^2(\mathbb T))
\end{split}
\end{equation}
for a suitable $\varphi\in L^\infty(0, T; H^3(\mathbb T))\cap W^{1,\infty}(0, T; H^2(\mathbb T))$.

\vspace{.15cm}
{\it Step (3.2): Strong limit}. We are going to show that, up to subsequences, $\{\varphi_m\}$ strongly converges to $\varphi$ in $C([0, T]; H^{3-\varepsilon}(\mathbb T))$ whenever $0<\varepsilon<1$.

For a fixed $\varepsilon\in]0,1[$ let $p>1$ be chosen such that $0<\varepsilon-\frac1{p}=:\alpha<1$. Since $\{\varphi_m\}$ is bounded in $C([0, T]; H^3(\mathbb T))\cap C^1([0, T]; H^2(\mathbb T))$, it is also bounded in $L^p(0, T; H^3(\mathbb T))\cap W^{1,p}(0, T; H^2(\mathbb T))$ and, by interpolation, in $W^{\varepsilon, p}(0, T; H^{3-\varepsilon}(\mathbb T))$. On the other hand by Sobolev Imbedding Theorem we have
\begin{equation*}
W^{\varepsilon, p}(0, T; H^{3-\varepsilon}(\mathbb T))\hookrightarrow C^\alpha([0, T]; H^{3-\varepsilon}(\mathbb T))\,.
\end{equation*}
Hence we derive that $\{\varphi_m\}$ is bounded in $C([0, T]; H^3(\mathbb T))\cap C^\alpha([0, T]; H^{3-\varepsilon}(\mathbb T))$. Since moreover $H^3(\mathbb T)\hookrightarrow H^{3-\varepsilon}(\mathbb T)$ with compact imbedding, we find that the sequence $\{\varphi_m\}$ is equicontinuous in $C([0, T]; H^{3-\varepsilon}(\mathbb T))$. From Ascoli-Arzel\`a Theorem, we obtain that $\{\varphi_m\}$ is relatively compact in $C([0, T]; H^{3-\varepsilon}(\mathbb T))$, hence from $\{\varphi_m\}$ one can extract a subsequence (still denoted with $\{\varphi_m\}$) strongly convergent in $C([0, T]; H^{3-\varepsilon}(\mathbb T))$. By uniqueness of the limit, the strong limit in $C([0, T]; H^{3-\varepsilon}(\mathbb T))$ is the same function $\varphi$ obtained as the weak $\ast$ limit of $\{\varphi_m\}$ in the previous step (3.1).

\vspace{.15cm}
{\it Step (3.3): Additional regularity of $\varphi$}. We are going to show that $\varphi\in C_w([0, T]; H^3(\mathbb T))$, namely the space of continuous functions on $[0, T]$ with values in $H^3(\mathbb T)$ endowed with the weak topology. To this end, we use the results of Appendix \ref{Lions-Magenes_sect}. In fact, from \eqref{conv_star} we have (using the notation of that section)
\begin{equation*}
\varphi\in W(H^3(\mathbb T),H^2(\mathbb T)):=\left\{u\in L^2(0, T; H^3(\mathbb T))\,:\,\,u_t\in L^2(0, T; H^2(\mathbb T))\right\}
\end{equation*}
(recall that $L^\infty(0, T; X)\hookrightarrow L^2(0, T; X)$ for every Banach space $X$). Hence by Lemma \ref{lemma2} in Appendix \ref{Lions-Magenes_sect} we have
\begin{equation*}
\varphi\in W(H^3(\mathbb T),H^2(\mathbb T))\hookrightarrow C([0, T]; \left[H^3(\mathbb T),H^2(\mathbb T)\right]_{1/2})\equiv C([0, T]; H^{2.5}(\mathbb T))\hookrightarrow C([0, T]; H^2(\mathbb T))\,.
\end{equation*}
Hence from \eqref{conv_star} and the above inclusion we get that
\begin{equation*}
\varphi\in L^\infty(0, T; H^3(\mathbb T))\cap C([0, T]; H^2(\mathbb T))\hookrightarrow C_w([0, T]; H^3(\mathbb T))\,,
\end{equation*}
in view of Lemma \ref{lemma1} of Appendix \ref{Lions-Magenes_sect}. Let us observe that for all $m$ also the function $\varphi_m$ belongs to $C_w([0, T]; H^3(\mathbb T))$ as a straightforward consequence of the regularity in \eqref{reg_phim}.

\vspace{.15cm}
{\it Step (3.4): Weak limit}. We are going to show that $\{\varphi_m\}$ converges to $\varphi$ in $C_w([0, T]; H^{3}(\mathbb T))$, which means that
\begin{equation}\label{conv_debole}
\varphi_m(t)\rightharpoonup\varphi(t)\,,\quad\mbox{in}\,\,H^3(\mathbb T)\,\,\,\mbox{uniformly in}\,\,t\in[0, T]\,.
\end{equation}
To prove the above convergence, we need to show that for all $\phi\in H^{-3}(\mathbb T)$
\begin{equation}\label{conv_debole1}
\langle \varphi_m(t)-\varphi(t)\,,\,\phi\rangle_{H^{3}, H^{-3}}\rightarrow 0\,,\quad\mbox{uniformly in}\,\,t\in[0, T]\,.
\end{equation}

Let $\phi\in H^{-3}(\mathbb T)$ arbitrarily fixed; since $H^{\varepsilon-3}(\mathbb T)$ is dense in $H^{-3}(\mathbb T)$, there exists a sequence $\{\phi_k\}$ in $H^{\varepsilon-3}(\mathbb T)$ such that $\phi_k\to\phi$ in $H^{-3}(\mathbb T)$. For  arbitrary indices $m,k$, we compute
\begin{equation}\label{stima_conv_deb1}
\begin{split}
\vert \langle \varphi_m(t)&-\varphi(t)\,,\,\phi\rangle_{H^{3}, H^{-3}}\vert \le \vert \langle \varphi_m(t)-\varphi(t)\,,\,\phi_k\rangle_{H^{3}, H^{-3}}\vert +\vert\langle \varphi_m(t)-\varphi(t)\,,\,\phi-\phi_k\rangle_{H^{3}, H^{-3}}\vert\\
&\le \Vert \varphi_m(t)-\varphi(t)\Vert_{H^{3-\varepsilon}(\mathbb T)}\Vert\phi_k\Vert_{H^{\varepsilon-3}(\mathbb T)}+(\Vert\varphi_m(t)\Vert_{H^3(\mathbb T)}+\Vert\varphi(t)\Vert_{H^3(\mathbb T)})\Vert\phi-\phi_k\Vert_{H^{-3}(\mathbb T)}\\
&\le\Vert \varphi_m(t)-\varphi(t)\Vert_{H^{3-\varepsilon}(\mathbb T)}\Vert\phi_k\Vert_{H^{\varepsilon-3}(\mathbb T)}+C_1\Vert\phi-\phi_k\Vert_{H^{-3}(\mathbb T)}\,,
\end{split}
\end{equation}
since $\varphi_m, \varphi$ are uniformly bounded in $L^\infty(0, T; H^{3}(\mathbb T))$.

Since $\phi_k\to\phi$ in $H^{-3}(\mathbb T)$,
\begin{equation}\label{conv_H-3}
\forall \eta>0 \,\,\,\,\ \exists\,\, k_0\,\, :\, \forall k\geq k_0\,\,\,\, \Vert\phi-\phi_k\Vert_{H^{-3}(\mathbb T)}\leq \frac{\eta}{2C_1}.
\end{equation}
For arbitrarily fixed $\eta>0$, estimate \eqref{stima_conv_deb1} for $k=k_0$, together with \eqref{conv_H-3}, becomes
\begin{equation}\label{stima_conv_deb2}
\vert \langle \varphi_m(t)-\varphi(t)\,,\,\phi\rangle_{H^{3}, H^{-3}}\vert
\le\Vert \varphi_m(t)-\varphi(t)\Vert_{H^{3-\varepsilon}(\mathbb T)}\Vert\phi_{k_0}\Vert_{H^{\varepsilon-3}(\mathbb T)}+\frac{\eta}{2}\,.
\end{equation}
Since $\varphi_m$  is strongly convergent to $\varphi$ in $C([0, T];H^{3-\varepsilon}(\mathbb T))$, there exists $m_0$ (depending only on $\eta$ and $k_0$ through $\Vert \phi_{k_0}\Vert_{H^{\varepsilon-3}(\mathbb T)}$) such that
\begin{equation}\label{stima_conv_deb3}
\Vert \varphi_m(t)-\varphi(t)\Vert_{H^{3-\varepsilon}(\mathbb T)}\leq \frac{\eta}{2\Vert \phi_{k_0} \Vert_{H^{\varepsilon-3}(\mathbb T)}}\,\, \forall\, m\geq m_0, \,\,\forall\, t\in [0, T].
\end{equation}
Gathering estimates \eqref{stima_conv_deb2} and \eqref{stima_conv_deb3} gives
\begin{equation}\label{stima_conv_deb4}
\vert \langle \varphi_m(t)-\varphi(t)\,,\,\phi\rangle_{H^{3}, H^{-3}}\vert
\le \eta\,\,\,\,\,\quad
\forall t\in[0, T],
\end{equation}
i.e. the convergence in \eqref{conv_debole}.

\vspace{.15cm}
{\it Step (3.5): Additional regularity for $\varphi_t$}. We are going to prove that $\varphi_t\in C_w([0, T]; H^2(\mathbb T))$. In order to apply the same arguments used above for $\varphi$ (cf. Steps (3.3)), in view of the second convergence in \eqref{conv_star} it is sufficient to show that
 \begin{equation}\label{conv-derivate-tt}
 \varphi_{m,tt}\rightharpoonup^* \varphi_{tt} \,\quad  \mbox{in}\,\,\,\, L^\infty(0, T;H^1)\,
 \end{equation}
 and work with $(\varphi_t,\varphi_{tt})$ in $H^2(\mathbb T)\times H^1(\mathbb T)$ instead of $(\varphi, \varphi_t)\in H^3(\mathbb T)\times H^2(\mathbb T)$.

Let us start from $\varphi_m$, solution of \eqref{onde_integro_diff} on $[0, T]$, which gives (see also \eqref{equ1ter})
\begin{equation*}
\varphi_{m,tt}=(\mu-2\phi_{m,x}) \varphi_{m,xx}-\mathcal Q[\varphi_m]\,.
\end{equation*}
Since $\varphi_m\in L^\infty(0, T;H^3(\mathbb T))$, the Sobolev continuity of $\mathbb H$ and Lemma \ref{lemma_prod_alg} imply
\begin{equation*}
(\mu-2\phi_{m,x})\in L^\infty(0, T;H^2(\mathbb T))\,, \quad \varphi_{m,xx}\in L^\infty(0, T;H^1(\mathbb T))\,\,
\Rightarrow (\mu-2\phi_{m,x}) \varphi_{m,xx} \in L^\infty(0, T;H^1(\mathbb T))\,,
\end{equation*}
with norm uniformly bounded with respect to $m$, because of \eqref{stima_3m}.
By using Lemma \ref{lemma_stima_quadr} with $r=1$
\begin{equation*}
\Vert\mathcal Q[\varphi_m] \Vert_{H^1(\mathbb T)}\leq C\Vert\varphi_{m,x}\Vert_{H^2(\mathbb T)} \Vert\varphi_{m,x}\Vert_{H^1(\mathbb T)}\leq C\Vert\varphi_{m,x}\Vert^2_{H^2(\mathbb T)}\leq C\,,\quad \forall t\in[0, T]\,
\end{equation*}
in view of \eqref{stima_3m}, with $C$ independent on $m$. Hence we get that $\{\varphi_{m,tt}\}$ is bounded in $L^\infty(0, T;H^1(\mathbb T))$. Up to a subsequence we find that
\begin{equation}\label{conv_tt}
\varphi_{m,tt}\rightharpoonup^* \,\tilde{\varphi}\quad \mbox{in} \,\, L^\infty(0, T;H^1(\mathbb T))\,.
\end{equation}
It remains to prove that $\tilde{\varphi}=\varphi_{tt}$. This comes from the convergence
\begin{equation}
\varphi_{m}\rightharpoonup^*\varphi \,\quad  \mbox{in}\,\,\,\, L^\infty(0, T;H^3(\mathbb T))\quad \Rightarrow\quad \varphi_{m} \rightarrow \varphi \,\quad  \mbox{in}\,\,\,\,\mathcal {D}^\prime((0, T)\times \mathbb T) \quad \Rightarrow\quad \varphi_{m,tt} \rightarrow \varphi_{tt} \,\quad  \mbox{in}\,\,\,\,\mathcal {D}^\prime((0, T)\times \mathbb T)
\end{equation}
which gives, from the uniqueness of the limit, $\varphi_{tt}=\tilde{\varphi}\in L^\infty(0, T;H^1(\mathbb T))$.

\vspace{.15cm}
{\it Step (3.6): Weak limit of $\{\varphi_{m,t}\}$.} In order to prove that $\varphi_{m, t}\rightarrow \varphi_t$ in $C_w([0, T]; H^2(\mathbb T))$ (up to subsequences), we repeat the same arguments as in steps (3.3) (3.4), with $\varphi_{m,t}$, $\varphi_t$ instead of $\varphi_m$, $\varphi$ (and $H^2(\mathbb T)$ instead of $H^3(\mathbb T)$).

\vspace{.15cm}
{\it Step (3.7): Sign condition for the limit $\varphi$.} From the strong convergence in $C([0,T];H^{3-\epsilon}(\mathbb T))$, the Sobolev continuity of $\mathbb H$,
and the Sobolev imbedding $H^{2-\epsilon}(\mathbb T)\hookrightarrow L^\infty(\mathbb T)$ we find that
\begin{equation}\label{conv_uniforme}
\phi_{m,x}(t,x)\rightarrow\phi_{x}(t,x)\quad\mbox{uniformly in}\,\,[0, T]\times\mathbb T\,.
\end{equation}
Then, passing to the point-wise limit as $m\to +\infty$ in the sign condition \eqref{131} for $\varphi_m$,
we recover the same condition for $\varphi$ that is
\begin{equation}\label{sign_cond_limit}
\mu-2\phi_{x}(t,x)\ge\delta\,,\quad (t,x)\in[0, T]\times\mathbb T\,.
\end{equation}

\vspace{.15cm}
{\it Step (3.8): Strong time continuity of the limit $\varphi$.} We are going to prove that $\varphi\in C([0, T];H^3(\mathbb T))\cap C^1([0, T]; H^2(\mathbb T))$, namely that the $\varphi$ and $\varphi_t$ are in fact {\it strongly continuous in time} as $H^3$ and $H^2-$valued functions respectively. In order to get the strong continuity we follow here a classical argument developed by Majda \cite{majda}. Nevertheless for the reader convenience, we develop the argument explicitly here below. We limit ourselves to prove the strong right-continuity in $t=0$, the continuity in the rest of the interval $[0, T]$ being obtained by the same arguments.
For each $m$ let us set for short
\begin{equation}\label{ym}
\mathcal E(\varphi_m(t)):=\Vert\varphi_{m,t}(t)\Vert^2_{H^2(\mathbb T)}+\int_{\mathbb T}(\mu-2\phi_{m,x}(t))\vert\langle\partial_x\rangle^2\varphi_{m,x}(t)\vert^2dx\,.
\end{equation}
Following the lines of the proof of Theorem \ref{mainHs} we find that
\begin{equation}\label{dis_ym}
\frac{d}{dt}\sqrt{\mathcal E(\varphi_m(t))}\le C_\delta\mathcal E(\varphi_m(t))\,,\qquad\forall\,t\in[0, T]\,,
\end{equation}
see \eqref{dis_y}. Integrating \eqref{dis_ym} over $[0,t]$ for arbitrary $0<t\le  T$ and using \eqref{stima_3m} we obtain
\begin{equation}\label{dis_int_ym}
\sqrt{\mathcal E(\varphi_m(t))}\le \sqrt{\mathcal E(\varphi_m(0))}+ C_\delta\int_0^t \mathcal E(\varphi_m(\tau))\,d\tau\le \sqrt{\mathcal E(\varphi_m(0))}+C^\prime_{\delta}t\,,\qquad\forall\,t\in[0, T]\,,
\end{equation}
where $C^\prime_\delta$ is a suitable positive constant independent of $m$.

Now we proceed with the following steps.

{\it Step $(3.8)_1$}: {\it We pass to the limit as $m\to +\infty$ in the inequality \eqref{dis_int_ym}} in order to find an analogous relation for the weak limit $\varphi(t)$, that is
\begin{equation}\label{dis_int_y}
\sqrt{\mathcal E(\varphi(t))}\le \sqrt{\mathcal E(\varphi(0))}+C^\prime_{\delta}t\,,\qquad\forall\,t\in[0, T]\,.
\end{equation}

In view of Corollary \ref{cor_ioffe}, we know that $\mathcal E(\cdot)$ is lower semi-continuous with respect to the weak convergences in $H^3(\mathbb T)$ and $H^2(\mathbb T)$ in the sense that
\begin{equation}\label{conv_m}
\varphi_m(t)\rightharpoonup\varphi(t)\,\,\,\mbox{w-}H^3(\mathbb T)\quad\mbox{and}\quad\varphi_{m,t}(t)\rightharpoonup\varphi_t(t)\,\,\,\mbox{w-}H^2(\mathbb T)
\end{equation}
yields
\begin{equation*}
\mathcal E(\varphi(t))\le\liminf\limits_{m\to +\infty}\mathcal E(\varphi_m(t))\,.
\end{equation*}
Hence passing to the limit as $m\to +\infty$ in \eqref{dis_int_ym} we get
\begin{equation*}
\sqrt{\mathcal E(\varphi(t))}\le\sqrt{\liminf\limits_{m\to +\infty}\mathcal E(\varphi_m(t))}=\liminf\limits_{m\to +\infty}\sqrt{\mathcal E(\varphi_m(t))}\le\liminf_{m\to +\infty}\sqrt{\mathcal E(\varphi_m(0))}+C^\prime_{\delta}t\,,\qquad\forall\,t\in[0, T]\,,
\end{equation*}
where
\begin{equation*}
\liminf_{m\to +\infty}\sqrt{\mathcal E(\varphi_m(0))}=\lim_{m\to +\infty}\sqrt{\mathcal E(\varphi_m(0))}=\sqrt{\mathcal E(\varphi(0))}\,,
\end{equation*}
since from
\begin{equation*}
\varphi_m(0)=\varphi_m^{(0)}\rightarrow\varphi(0)=\varphi^{(0)}\,\,\,\mbox{s-}H^3(\mathbb T)\quad\mbox{and}\quad\varphi_{m,t}(0)=\varphi_m^{(1)}\rightarrow\varphi_t(0)=\varphi^{(1)}\,\,\,\mbox{s-}H^2(\mathbb T)
\end{equation*}
(see \eqref{convergenze}) it follows that
\begin{equation*}
\mu-2\phi_{m,x}(0)\rightarrow\mu-2\phi_{x}(0)\,\,\,\mbox{in}\,\,L^\infty(\mathbb T)\quad\mbox{and}\quad\vert\langle\partial_x\rangle^2\varphi_{m,x}(0)\vert^2\rightarrow \vert\langle\partial_x\rangle^2\varphi_{x}(0)\vert^2\,\,\,\mbox{s-}L^1(\mathbb T)\,.
\end{equation*}
This proves the inequality \eqref{dis_int_y}.

{\it Step $(3.8)_2$:} {\it We pass to the limit as $t\to 0^+$ in \eqref{dis_int_y}}. Directly from the latter inequality we find
\begin{equation}\label{limsup_y}
\limsup\limits_{t\to 0^+}\sqrt{\mathcal E(\varphi(t))}\le \sqrt{\mathcal E(\varphi(0))}\,.
\end{equation}
Because $\varphi\in C_w([0, T];H^{3}(\mathbb T))$ and $\varphi_t\in C_w([0, T];H^{2}(\mathbb T))$ we have that
\begin{equation*}
\varphi(t)\rightharpoonup\varphi(0)\,\,\,\mbox{w-}H^3(\mathbb T)\quad\mbox{and}\quad\varphi_{t}(t)\rightharpoonup\varphi_t(0)\,\,\,\mbox{w-}H^2(\mathbb T)\,.
\end{equation*}
Hence, from the lower semi-continuity of $\mathcal E(\cdot)$ proved in Corollary \ref{cor_ioffe} we find
\begin{equation*}
\mathcal E(\varphi(0))\le\liminf\limits_{t\to 0^+}\mathcal E(\varphi(t))\,,
\end{equation*}
from which
\begin{equation}\label{liminf_y}
\sqrt{\mathcal E(\varphi(0))}\le\sqrt{\liminf\limits_{t\to 0^+}\mathcal E(\varphi(t))}=\liminf\limits_{t\to 0^+}\sqrt{\mathcal E(\varphi(t))}\,.
\end{equation}
From \eqref{limsup_y} and \eqref{liminf_y} we immediately find
\begin{equation*}
\sqrt{\mathcal E(\varphi(0))}=\lim\limits_{t\to 0^+}\sqrt{\mathcal E(\varphi(t))}\,.
\end{equation*}
hence
\begin{equation}\label{lim_y}
\mathcal E(\varphi(0))=\lim\limits_{t\to 0^+}\mathcal E(\varphi(t))\,.
\end{equation}

{\it Step $(3.8)_3$: We derive the strong convergences $\varphi_t(t) \rightarrow \varphi_t(0)$ and $\varphi_x(t) \rightarrow \varphi_x(0)$ in $H^2(\mathbb T)$}. This comes from the the weak convergences  $\varphi(t) \rightharpoonup \varphi(0)$ $w$-$H^3(\mathbb T)$, $\varphi_t(t) \rightharpoonup \varphi_t(0)$  $w$-$H^2(\mathbb T)$ and the energy convergence \eqref{lim_y}, see Section \ref{Ioffe}. By using the result of Section \ref{Ioffe} we find that
$$\varphi_t(t) \rightarrow \varphi_t(0) \quad s-H^2(\mathbb T)\qquad\mbox{and}\qquad \varphi_x(t) \rightarrow \varphi_x(0) \quad  s-H^2(\mathbb T).
$$
This concludes the proof of the strong continuity of $\varphi$ in $H^3(\mathbb T)$ (recall that $\varphi$ has zero spatial mean).

{\it Step (3.9): The strong limit $\varphi$ satisfies the problem \eqref{onde_integro_diff} \eqref{id}}. For every $m$ we know that $\varphi_m$ satisfies the equation
\begin{equation}\label{eq_m}
\varphi_{m,tt}- (\mu-2\phi_{m,x})\varphi_{m,xx}+\mathcal Q[\varphi_m]=0\,\quad \mbox{in}\,\,\,[0, T]\times \mathbb T\,.
\end{equation}
Our goal here is passing to the limit as $m\to +\infty$ in each term of the above equation to show that $\varphi$ is still a solution to the same equation. From \eqref{conv_debole} and since $H^3(\mathbb T)$ is compactly imbedded in $H^2(\mathbb T)$, we find that
\begin{equation}\label{pass_lim_1}
\begin{split}
&\varphi_m(t)\rightharpoonup\varphi(t)\,,\quad{w-}H^3(\mathbb T)\quad\Rightarrow\quad \varphi_m(t)\rightarrow\varphi(t)\,,\quad\mbox{s-}H^2(\mathbb T)\\
&\Rightarrow\quad \varphi_{m,x}(t)\rightarrow\varphi_x(t)\quad\mbox{s-}H^1(\mathbb T)\quad \Rightarrow\quad \phi_{m,x}(t)\rightarrow\phi_x(t)\quad\mbox{s-}H^1(\mathbb T)\\
&\Rightarrow\quad \mu-2\phi_{m,x}(t)\rightarrow\mu-2\phi_x(t)\quad\mbox{s-}H^1(\mathbb T)\,,\quad\mbox{uniformly in}\,\,t\in[0, T]
\end{split}
\end{equation}
and
\begin{equation}\label{pass_lim_2}
\begin{split}
&\varphi_{m,x}(t)\rightarrow\varphi_x(t)\quad\mbox{s-}H^1(\mathbb T)\quad \Rightarrow\quad \varphi_{m,xx}(t)\rightarrow\varphi_{xx}(t)\quad\mbox{s-}L^2(\mathbb T)\,,\quad\mbox{uniformly in}\,\,t\in[0, T]\,.
\end{split}
\end{equation}
From \eqref{pass_lim_1} and \eqref{pass_lim_2} and using Lemma \ref{lemma_prod_alg} we find
\begin{equation}\label{pass_lim_3}
(\mu-2\phi_{m,x}(t))\varphi_{m,xx}(t)\rightarrow (\mu-2\phi_{x}(t))\varphi_{xx}(t)\quad\mbox{s-}L^2(\mathbb T)\,,\quad\mbox{uniformly in}\,\,t\in[0, T]\,.
\end{equation}
We consider now the nonlinear term
\begin{equation*}
\mathcal Q[\varphi_m]=-3\left[\mathbb H\,;\,\phi_{m,x}\right]\phi_{m,xx}-\left[\mathbb H\,;\,\phi_{m}\right]\phi_{m,xxx}\,.
\end{equation*}
Expanding the first commutator in the right-hand side above as
\begin{equation*}
\left[\mathbb H\,;\,\phi_{m,x}\right]\phi_{m,xx}=\mathbb H[\phi_{m,x}\phi_{m,xx}]-\phi_{m,x}\mathbb H[\phi_{m,xx}]
\end{equation*}
we note that the latter appears to be the sum of products of two sequences, one $\{\phi_{m,xx}\}$ strongly convergent in $L^2(\mathbb T)$ and the other $\{\phi_{m,x}\}$ strongly convergent in $H^1(\mathbb T)$ (as a straightforward consequence of \eqref{pass_lim_1} and \eqref{pass_lim_2} and the Sobolev continuity of $\mathbb H$). Then in view of Lemma \ref{lemma_prod_alg} the product of these terms is strongly convergent in $L^2(\mathbb T)$ to the product of the limits, hence
\begin{equation}\label{pass_lim_4}
\left[\mathbb H\,;\,\phi_{m,x}(t)\right]\phi_{m,xx}(t)\rightarrow \left[\mathbb H\,;\,\phi_{x}(t)\right]\phi_{xx}(t)\,,\quad\mbox{s-}L^2(\mathbb T)\,,\quad\mbox{uniformly in}\,\,t\in[0, T]\,.
\end{equation}
It remains to prove that the commutator $\left[\mathbb H\,;\,\phi_{m}\right]\phi_{m,xxx}$ strongly converges to $\left[\mathbb H\,;\,\phi\right]\phi_{xxx}$ in $L^2(\mathbb T)$, uniformly in $t\in[0, T]$. To prove it we first compute for each $t\in[0, T]$:
\begin{equation*}
\begin{split}
\Vert &\left[\mathbb H\,;\,\phi_{m}(t)\right]\phi_{m,xxx}(t)-\left[\mathbb H\,;\,\phi(t)\right]\phi_{xxx}(t)\Vert_{L^2(\mathbb T)}\\
&= \Vert \left[\mathbb H\,;\,\phi_{m}(t)\right]\phi_{m,xxx}(t)-\left[\mathbb H\,;\,\phi(t)\right]\phi_{m,xxx}(t)+\left[\mathbb H\,;\,\phi(t)\right]\phi_{m,xxx}(t)-\left[\mathbb H\,;\,\phi(t)\right]\phi_{xxx}(t)\Vert_{L^2(\mathbb T)}\\
&\le\Vert \left[\mathbb H\,;\,\phi_{m}(t)-\phi(t)\right]\phi_{m,xxx}(t)\Vert_{L^2(\mathbb T)}+\Vert \left[\mathbb H\,;\,\phi(t)\right](\phi_{m,xxx}(t)-\phi_{xxx}(t))\Vert_{L^2(\mathbb T)}\,.
\end{split}
\end{equation*}
Now we use estimate \eqref{stima_comm_p} of Lemma \ref{lemma_comm_ale_2} with $\sigma=0$, $p=2$, $v=\phi_{m}(t)-\phi(t)$ and $f=\phi_{m,x}(t)$ to get
\begin{equation}\label{stima_pass_lim1}
\Vert \left[\mathbb H\,;\,\phi_{m}(t)-\phi(t)\right]\phi_{m,xxx}(t)\Vert_{L^2(\mathbb T)}\le C\Vert\phi_{m,xx}(t)-\phi_{xx}(t)\Vert_{L^2(\mathbb T)}\Vert \phi_{m,x}(t)\Vert_{H^1(\mathbb T)}\,;
\end{equation}
similarly, estimate \eqref{stima_comm_p} of Lemma \ref{lemma_comm_ale_2} with $\sigma=0$, $p=2$, $v=\phi(t)$ and $f=\phi_{m,x}(t)-\phi_x(t)$ gives
\begin{equation}\label{stima_pass_lim2}
\Vert \left[\mathbb H\,;\,\phi(t)\right](\phi_{m,xxx}(t)-\phi_{xxx}(t))\Vert_{L^2(\mathbb T)}\le C\Vert\phi_{xx}(t)\Vert_{L^2(\mathbb T)}\Vert \phi_{m,x}(t)-\phi_{x}(t)\Vert_{H^1(\mathbb T)}\,.
\end{equation}
From \eqref{stima_pass_lim1} and \eqref{stima_pass_lim2} we then get
\begin{equation*}
\begin{split}
\Vert &\left[\mathbb H\,;\,\phi_{m}(t)\right]\phi_{m,xxx}(t)-\left[\mathbb H\,;\,\phi(t)\right]\phi_{xxx}(t)\Vert_{L^2(\mathbb T)}\\
&\le C\Vert\phi_{m,xx}(t)-\phi_{xx}(t)\Vert_{L^2(\mathbb T)}\Vert \phi_{m,x}(t)\Vert_{H^1(\mathbb T)}+C\Vert\phi_{xx}(t)\Vert_{L^2(\mathbb T)}\Vert \phi_{m,x}(t)-\phi_{x}(t)\Vert_{H^1(\mathbb T)}\,,
\end{split}
\end{equation*}
from  which we derive that
\begin{equation}\label{pass_lim_5}
\left[\mathbb H\,;\,\phi_{m}(t)\right]\phi_{m,xxx}(t)\rightarrow\left[\mathbb H\,;\,\phi(t)\right]\phi_{xxx}(t)\,,\quad\mbox{s-}L^2(\mathbb T)\,,\quad\mbox{uniformly in}\,\,t\in[0, T]\,,
\end{equation}
as a consequence of the convergence in \eqref{pass_lim_1}, \eqref{pass_lim_2} and the Sobolev continuity of $\mathbb H$. From \eqref{pass_lim_4} and \eqref{pass_lim_5} we derive that
\begin{equation}\label{pass_lim_6}
\mathcal Q[\varphi_m]\rightarrow\mathcal Q[\varphi]\,,\quad\mbox{s-}L^2(\mathbb T)\,,\quad\mbox{uniformly in}\,\,t\in[0, T]\,.
\end{equation}
If we solve the equation \eqref{eq_m} with respect to $\varphi_{m,tt}$, the results collected before show that
\begin{equation*}
\varphi_{m,tt}\rightarrow (\mu-2\phi_x)\varphi_{xx}-\mathcal Q[\varphi]\,,\quad\mbox{s-}L^2(\mathbb T)\,,\quad\mbox{uniformly in}\,\,t\in[0, T]\,.
\end{equation*}
On the other hand, from
\begin{equation}\label{debole_star_H3}
\varphi_m\rightharpoonup^\ast\varphi\,,\quad\mbox{in}\,\,L^\infty(0, T; H^3(\mathbb T))
\end{equation}
it follows that
\begin{equation*}
\varphi_m\rightarrow\varphi\,,\quad\mbox{in}\,\,\mathcal D^\prime((0, T)\times\mathbb T)\quad\Rightarrow\quad \varphi_{m,tt}\rightarrow\varphi_{tt}\,,\quad\mbox{in}\,\,\mathcal D^\prime((0, T)\times\mathbb T)\,.
\end{equation*}
Hence, from uniqueness of the limit, we get
\begin{equation*}
\varphi_{tt}=(\mu-2\phi_x)\varphi_{xx}-\mathcal Q[\varphi]\,,\quad\mbox{in}\,\,\mathcal D^\prime((0,  T)\times\mathbb T)\,,
\end{equation*}
which ends the proof.

\vspace{.15cm}
In the end, let us show that $\varphi$ satisfies the initial conditions \eqref{id}. From \eqref{pass_lim_1} we have that
\begin{equation*}
\varphi_m(0)=\varphi_m^{(0)}\rightarrow\varphi(0)\,,\quad\mbox{s-}H^2(\mathbb T)\,.
\end{equation*}
On the other hand,
\begin{equation*}
\varphi_m^{(0)}\rightarrow\varphi^{(0)}\,,\quad\mbox{s-}H^3(\mathbb T)\,,
\end{equation*}
by construction, see \eqref{convergenze}. Thus $\varphi(0)=\varphi^{(0)}$ follows from uniqueness of the limit in $H^2(\mathbb T)$.

Arguing as in the Step (3.4) above, one can prove that
\begin{equation*}
\varphi_{m,t}(t)\rightharpoonup\varphi_t(t)\,,\quad\mbox{w-}H^2(\mathbb T)\,,\quad\mbox{uniformly in}\,\,t\in[0, T]\,,
\end{equation*}
hence, by compactness of the imbedding $H^2(\mathbb T)\hookrightarrow H^1(\mathbb T)$,
\begin{equation*}
\varphi_{m,t}(0)=\varphi^{(1)}_m\rightarrow\varphi_t(0)\,,\quad\mbox{s-}H^1(\mathbb T)\,.
\end{equation*}
On the other hand
\begin{equation*}
\varphi^{(1)}_m\rightarrow\varphi^{(1)}\,,\quad\mbox{s-}H^2(\mathbb T)
\end{equation*}
by construction, see \eqref{convergenze} once again. Thus $\varphi_t(0)=\varphi^{(1)}$ follows from uniqueness of the limit in $H^1(\mathbb T)$.
We have proved that $\varphi$ is a solution of problem \eqref{onde_integro_diff} \eqref{id} on $[0, T]$.

\vspace{.5cm}
{\it Step (4)}: Passing to the limit as $m\to +\infty$ into the estimate \eqref{stima_2m} for an arbitrary $t\in[0, T]$ and in view of the convergence
\begin{equation*}
\varphi_{m,x}(t)\rightharpoonup\varphi_x(t)\,,\quad \varphi_{m,t}(t)\rightharpoonup\varphi_t(t)\,,\quad\mbox{in}\,\,H^2(\mathbb T)\,,
\end{equation*}
(cf. Steps (3.4), (3.5)) and the lower semi-continuity of the $H^2-$norm with respect to the weak convergence, we get that $\varphi$ satisfies
\begin{equation}\label{stima_2_phi}
\begin{split}
\Vert\varphi_{x}(t)\Vert^2_{H^2(\mathbb T)}+ \Vert\varphi_{t}(t)\Vert^2_{H^2(\mathbb T)}
&\leq \frac{4\rho^2}{\min\left\{1,\frac{\delta}{2}\right\}}\,,\quad\forall\,t\in[0, T]\,.
\end{split}
\end{equation}

Because of the arbitrariness of $ R$ satisfying \eqref{Rtilde}, let us now argue as before on a decreasing sequence $\{ R_\nu\}$ such that
\begin{equation*}
\Vert\varphi^{(0)}_x\Vert^2_{H^2(\mathbb T)}+\Vert\varphi^{(1)}\Vert^2_{H^2(\mathbb T)}< R_\nu^2<R_0^2\,,\qquad  R_\nu\searrow (\Vert\varphi^{(0)}_x\Vert^2_{H^2(\mathbb T)}+\Vert\varphi^{(1)}\Vert^2_{H^2(\mathbb T)})^{1/2}\,,
\end{equation*}
(thus $ T_\nu:=\frac1{2C_\delta\rho_\nu}\nearrow T_0$). Because of the uniqueness of the solution, this gives that the solution $\varphi$ of problem \eqref{onde_integro_diff} \eqref{id} exists on $[0, T_0)$ such that
\begin{equation*}
\varphi\in C([0,T_0); H^3(\mathbb T))\cap C^1([0,T_0); H^2(\mathbb T))\,.
\end{equation*}
Moreover, passing to the limit as $\nu\to +\infty$ into the estimate \eqref{stima_2_phi} written on $[0, T_\nu]$ (with $\rho_\nu$ instead of $\rho$) and using Poincar\'e's inequality we find that $\varphi$ satisfies on $[0,T_0)$ the energy estimate
\begin{equation}\label{stima_2_phi1}
\begin{split}
\Vert\varphi(t)\Vert^2_{H^3(\mathbb T)}+ \Vert\varphi_{t}(t)\Vert^2_{H^2(\mathbb T)}
&\leq \frac{4(1+\mu+C)}{\min\left\{1,\frac{\delta}{2}\right\}}\left\{\Vert\varphi^{(0)}\Vert^2_{H^3(\mathbb T)}+ \Vert\varphi^{(1)}\Vert^2_{H^2(\mathbb T)}\right\}\,,\quad\forall\,t\in[0,T_0)\,,
\end{split}
\end{equation}
which is estimate \eqref{stima_energia_2} with $C_2=\frac{4(1+\mu+C)}{\min\left\{1,\frac{\delta}{2}\right\}}$.

\begin{remark}\label{}
Using the a priori estimate \eqref{stima_2_phi1} and applying some of the previous arguments, we can prove that $\varphi$ can be extended up to $t=T_0$ as a $H^3$ function and that $\varphi\in C_w([0, T_0]; H^3(\mathbb T))$ with $\varphi_t\in C_w([0, T_0]; H^2(\mathbb T))$.
\end{remark}


\subsection{Proof of Statement (2)}
Let us consider now the case of initial data $\varphi^{(0)}\in H^s(\mathbb T)$, $\varphi^{(1)}\in H^{s-1}(\mathbb T)$, with $s>3$, satisfying \eqref{ip_dati_iniziali} with the $R_0>0$ constructed in the proof of Statement (1). As before, we take sequences $\{\varphi^{(0)}_m\}$, $\{\varphi^{(1)}_m\}$ in $H^{\infty}(\mathbb T)$, with zero spatial mean, such that
 \begin{equation}\label{convergenze_s}
 \begin{split}
 &\varphi^{(0)}_m \rightarrow \varphi^{(0)}\,, \quad \mbox{in}\,\, H^{s}(\mathbb T)\,,\\
 & \varphi^{(1)}_m \rightarrow \varphi^{(1)}\,\quad \mbox{in}\,\, H^{s-1}(\mathbb T)\,;
 \end{split}
 \end{equation}
 satisfying \eqref{ipotesi_dati_appr} and \eqref{sign_cond_m_id}. For each $m$, let $\varphi_m$ the solution of problem \eqref{Pm}. From the proof of statement (1), we know that $\varphi_m$ satisfies the regularity assumptions \eqref{reg_phim} on $[0,T]$, for all positive $T<T_0$, with $T_0$ defined in \eqref{T0_nuova} and condition \eqref{sign_cond_ms}. From Theorem \ref{mainHs}, statement (2), with $r=s-1$, Poincar\'e's inequality and the boundedness of the sequences $\{\varphi^{(0)}_m\}$ and $\{\varphi^{(1)}_m\}$ respectively in $H^s(\mathbb T)$ and $H^{s-1}(\mathbb T)$, we find that
\begin{equation}\label{norme_limitate}
\Vert\varphi_{m}(t)\Vert^2_{H^s(\mathbb T)}+\Vert\varphi_{m,t}(t)\Vert^2_{H^{s-1}(\mathbb T)}\le M\,,\quad\forall\,t\in[0,T]\,,
\end{equation}
from which
\begin{equation*}
\varphi_{m}\rightharpoonup^\ast \widetilde\varphi\,,\quad\mbox{in}\,\,L^\infty(0,T; H^s(\mathbb T))
\end{equation*}
and
\begin{equation*}
\varphi_{m,t}\rightharpoonup^\ast \widetilde\varphi_t\,,\quad\mbox{in}\,\,L^\infty(0,T; H^{s-1}(\mathbb T))\,,
\end{equation*}
for a suitable function $\widetilde\varphi\in L^\infty(0,T; H^s(\mathbb T))\cap W^{1,\infty}(0,T; H^{s-1}(\mathbb T))$. Since from statement (1) we already know that
\begin{equation*}
\varphi_{m}\rightharpoonup^\ast \varphi\,,\quad\mbox{in}\,\,L^\infty(0,T; H^3(\mathbb T))\,,
\end{equation*}
being $\varphi$ the solution to original problem \eqref{onde_integro_diff}, \eqref{id} on $[0,T_0)$, see \eqref{debole_star_H3}, by uniqueness of the limit we find that $\varphi=\widetilde\varphi\in L^\infty(0,T; H^s(\mathbb T))\cap W^{1,\infty}(0,T; H^{s-1}(\mathbb T))$ and
\begin{equation}\label{debole_star_Hs}
\begin{split}
&\varphi_{m}\rightharpoonup^\ast \varphi\,,\quad\mbox{in}\,\, L^\infty(0,T; H^s(\mathbb T))\qquad\mbox{and}\qquad\varphi_{m,t}\rightharpoonup^\ast\varphi_t\quad\mbox{in}\,\,L^\infty(0,T; H^{s-1}(\mathbb T))\,.
\end{split}
\end{equation}
It remains to prove the continuity in time of $\varphi$ and $\varphi_t$. To first get the weak continuity in time, we proceed as in Steps (3.3)--(3.6) above (with $s$ instead of $3$) to find
\begin{equation}\label{reg_debole}
\varphi\in C_w([0,T]; H^s(\mathbb T))\cap C^1_w([0,T]; H^{s-1}(\mathbb T))
\end{equation}
and
\begin{equation}\label{debole_Hs}
\begin{split}
&\varphi_{m}(t)\rightharpoonup \varphi(t)\,,\quad\mbox{in}\,\, H^s(\mathbb T)\qquad\mbox{and}\qquad\varphi_{m,t}(t)\rightharpoonup\varphi_t(t)\quad\mbox{in}\,\,H^{s-1}(\mathbb T)\,,\quad\mbox{uniformly in}\,\,t\in[0,T]\,.
\end{split}
\end{equation}
To recover the strong continuity in time, we follow the same idea of Step (3.8): we prove the convergence of the energy
\begin{equation}\label{convE}
\mathcal E(\varphi(0))=\lim\limits_{t\to 0^+}\mathcal E(\varphi(t))\,,
\end{equation}
where now the energy is
\begin{equation}\label{ys}
\mathcal E(\varphi(t)):=\Vert\varphi_{t}(t)\Vert^2_{H^{s-1}(\mathbb T)}+\int_{\mathbb T}(\mu-2\phi_{x}(t))\vert\langle\partial_x\rangle^{s-1}\varphi_{x}(t)\vert^2dx\,.
\end{equation}
The convergence \eqref{convE} together with the weak convergence
\begin{equation}\label{conv_s}
\varphi(t)\rightharpoonup \varphi(0)\,,\quad\mbox{in}\,\, H^s(\mathbb T)\qquad\mbox{and}\qquad\varphi_{t}(t)\rightharpoonup\varphi_t(0)\quad\mbox{in}\,\,H^{s-1}(\mathbb T)\,,\quad\mbox{as}\,\,t\to 0^+
\end{equation}
(see \eqref{reg_debole}) implies the strong convergence
\begin{equation*}
\varphi(t)\rightarrow \varphi(0)\,,\quad\mbox{in}\,\, H^s(\mathbb T)\qquad\mbox{and}\qquad\varphi_{t}(t)\rightarrow\varphi_t(0)\quad\mbox{in}\,\,H^{s-1}(\mathbb T)\,,\quad\mbox{as}\,\,t\to 0^+\,,
\end{equation*}
see Section C.1 in Appendix \ref{Ioffe}.

To get \eqref{convE}, we first observe that Ioffe Theorem \ref{teorema_ioffe} and the weak convergence \eqref{conv_s} we derive
\begin{equation}\label{liminfE}
\mathcal E(\varphi(0))\le\liminf\limits_{t\to 0^+}\mathcal E(\varphi(t))\,.
\end{equation}
We prove that
\begin{equation}\label{limsupE}
\limsup\limits_{t\to 0^+}\mathcal E(\varphi(t))\le\mathcal E(\varphi(0))
\end{equation}
by following again the arguments developed in Step (3.8). From estimate \eqref{stima_s2} written for $\varphi_m$ and \eqref{norme_limitate} we get
\begin{equation*}
\frac{d}{dt} \mathcal E(\varphi_m(t)) \leq C\,, \quad \forall\, t\in[0,T]\,,
\end{equation*}
which gives, integrating on $[0,t]$ for $0<t<T$,
\begin{equation}\label{stima_Em}
\mathcal E(\varphi_m(t)) \leq \mathcal E(\varphi_m(0))+Ct\,, \quad \forall\, t\in[0,T]\,,
\end{equation}
with a positive constant $C$ independent of $m$. We pass to the limit as $m\to +\infty$ in \eqref{stima_Em} in order to find an analogue estimate for $\varphi(t)$. First by applying again Ioffe Theorem \ref{teorema_ioffe}, from \eqref{debole_Hs} we derive
\begin{equation*}
\mathcal E(\varphi(t)) \leq \liminf\limits_{m\to +\infty}\mathcal E(\varphi_m(t))\,, \quad \forall\, t\in[0,T]\,.
\end{equation*}
Passing now to the limit in \eqref{stima_Em} as $m\to +\infty$, in view of \eqref{convergenze_s} we find
\begin{equation}\label{limsup0}
\mathcal E(\varphi(t)) \leq \liminf\limits_{m\to +\infty}\mathcal E(\varphi_m(t))\le \liminf\limits_{m\to +\infty}\mathcal E(\varphi_m(0))+Ct=\lim\limits_{m\to +\infty}\mathcal E(\varphi_m(0))+Ct=\mathcal E(\varphi(0))+Ct\,, \quad \forall\, t\in[0,T]\,.
\end{equation}
Then passing to $\limsup$ as $t\to 0^+$ in \eqref{limsup0} we get \eqref{limsupE}.

In the end we find that the solution $\varphi$ satisfies the a priori estimate \eqref{stima_energia_s} on $[0,T]$; indeed passing to the limit as $m\to +\infty$ into the energy estimate \eqref{stima_Hs} written for $\varphi_m$ with $r=s-1$, where we exploit the weak convergence \eqref{debole_Hs}, the strong convergence \eqref{convergenze_s} and the lower semi-continuity of the Sobolev norms, we get

\begin{equation*}
\begin{split}
\Vert &\varphi_{x}(t)\Vert_{H^{s-1}(\mathbb T)}^2+\Vert\varphi_t(t)\Vert_{H^{s-1}(\mathbb T)}^2\le\liminf\limits_{m\to +\infty}\left\{\Vert\varphi_{m, x}(t)\Vert_{H^{s-1}(\mathbb T)}^2+\Vert\varphi_{m, t}(t)\Vert_{H^{s-1}(\mathbb T)}^2\right\}\\
&\le  C_s\liminf_{m\to +\infty}\left\{\Vert\varphi^{(0)}_{m,x}\Vert^2_{H^{s-1}(\mathbb T)}+\Vert\varphi^{(1)}_m\Vert^2_{H^{s-1}(\mathbb T)}\right\}=C_s\lim_{m\to +\infty}\left\{\Vert\varphi^{(0)}_{m,x}\Vert^2_{H^{s-1}(\mathbb T)}+\Vert\varphi^{(1)}_m\Vert^2_{H^{s-1}(\mathbb T)}\right\}\\
&=C_s\left\{\Vert\varphi^{(0)}_{x}\Vert^2_{H^{s-1}(\mathbb T)}+\Vert\varphi^{(1)}\Vert^2_{H^{s-1}(\mathbb T)}\right\}\le C_s\left\{\Vert\varphi^{(0)}\Vert^2_{H^{s}(\mathbb T)}+\Vert\varphi^{(1)}\Vert^2_{H^{s-1}(\mathbb T)}\right\}\,,\quad\forall\,t\in[0,T]\,.
\end{split}
\end{equation*}
From Poincar\'e inequality the above estimate gives \eqref{stima_energia_s}.

\appendix

\section{Some commutator and product estimates}\label{stima_commutatore}
In this Section, we collect some commutator estimates that will be used throughout the paper.                                                                                                    \begin{lemma}\label{lemma_comm}
 For $\sigma>1/2$ there exists a constant $C_\sigma>0$ such that
\begin{eqnarray*}
\Vert \left[\mathbb H\,;\,v\right]f\Vert_{L^2(\mathbb T)}\le C_\sigma\Vert v\Vert_{H^\sigma(\mathbb T)}\Vert f\Vert_{L^2(\mathbb T)}\,,\quad\forall\,v\in H^\sigma(\mathbb T)\,,\,\,\forall\,f\in L^2(\mathbb T)\,;\\
\Vert \left[\mathbb H\,;\,v\right]f_x\Vert_{L^2(\mathbb T)}\le C_\sigma\Vert v_x\Vert_{H^\sigma(\mathbb T)}\Vert f\Vert_{L^2(\mathbb T)}\,,\quad\forall\,v\in H^{\sigma+1}(\mathbb T)\,,\,\,\forall\,f\in L^2(\mathbb T)\,,
\end{eqnarray*}
where $\left[\mathbb H\,;\,v\right]$ is the commutator between the Hilbert transform $\mathbb H$ and the multiplication by $v$.
\end{lemma}
\begin{proof}
See \cite{M-S-T:ONDE1}.
\end{proof}

\begin{lemma}\label{lemma_comm_ale_2}
There exists a positive constant $C$ such that for every real $\sigma\ge 0$ and integer $p\ge 0$ the following estimates hold true:
 \begin{itemize}
\item [1.] for all functions $v\in H^{\sigma+p}(\mathbb T)$ and $f\in H^1(\mathbb T)$
\begin{equation}\label{stima_comm_p}
\Vert \left[\mathbb H\,;\,v\right]\partial^p_x f\Vert_{H^\sigma(\mathbb T)}\le C_{\sigma,p}\Vert \partial^{p}_x v\Vert_{H^\sigma(\mathbb T)}\Vert f\Vert_{H^1(\mathbb T)}\,.
\end{equation}
\item [2.] for all functions $v\in H^{\sigma+1+p}(\mathbb T)$ and $f\in L^2(\mathbb T)$
\begin{equation}\label{stima_comm_p2}
\Vert \left[\mathbb H\,;\,v\right]\partial^p_x f\Vert_{H^\sigma(\mathbb T)}\le C_{\sigma,p}\Vert \partial^{p}_x v\Vert_{H^{\sigma+1}(\mathbb T)}\Vert f\Vert_{L^2(\mathbb T)}\,.
\end{equation}
\end{itemize}
\end{lemma}
\begin{proof}
Using \eqref{convoluzione} and \eqref{hilbert1}, for all sufficiently smooth functions $u$, $v$  one computes
\begin{equation}\label{fourier1}
\begin{split}
\widehat{[\mathbb H\,;\,v]u}(k)&=\widehat{\mathbb H[vu]}(k)-\widehat{v\mathbb H[u]}(k)=-i{\rm sgn}\,k\sum\limits_{\ell}\widehat{v}(k-\ell)\widehat{u}(\ell)-\sum\limits_{\ell}\widehat{v}(k-\ell)(-i{\rm sgn}\,\ell)\widehat{u}(\ell)\\
&=-\sum\limits_{\ell}i({\rm sgn}\,k-{\rm sgn}\,\ell)\widehat{v}(k-\ell)\widehat{u}(\ell).
\end{split}
\end{equation}
Applying \eqref{fourier1} with $u=\partial^p_x f$ and noticing that ${\rm sgn}\,k-{\rm sgn}\,\ell=0$ as long as $k\ell>0$, we get
\begin{equation}\label{casi_2}
\widehat{\left[\mathbb H\,;\,v\right]\partial^p_x f}(k)=\begin{cases}\displaystyle -2i\sum\limits_{\ell<0}\widehat{v}(k-\ell)(i\ell)^p\widehat{f}(\ell)\,,\quad\mbox{if}\,\,k>0\,,\\
\displaystyle i\sum\limits_{\ell}{\rm sgn}\,\ell\,\widehat{v}(-\ell)(i\ell)^p\widehat{f}(\ell)\,,\quad\mbox{if}\,\,k=0\,,\\
\displaystyle 2i\sum\limits_{\ell>0}\widehat{v}(k-\ell)(i\ell)^p\widehat{f}(\ell)\,,\quad\mbox{if}\,\,k<0\,.
\end{cases}
\end{equation}
On the other hand, in view of \eqref{normaHs}
\begin{equation}\label{stima_ale_1}
\Vert\left[\mathbb H\,;\,v\right]\partial^p_x f\Vert^2_{H^\sigma(\mathbb T)}=2\pi\sum\limits_{k\in \mathbb Z}\langle k\rangle^{2\sigma}\vert\widehat{\left[\mathbb H\,;\,v\right]\partial^p_x f}(k)\vert^2\,.
\end{equation}
From \eqref{casi_2} and using the trivial inequality $\langle -\ell\rangle\ge 1$, we get for $k=0$:
\begin{equation}\label{coeff_0_ale}
\begin{split}
\vert \widehat{\left[\mathbb H\,;\,v\right]\partial^p_x f}(0)\vert &\le\sum\limits_{\ell}\vert\widehat{v}(-\ell)\vert\,\vert\ell\vert^p\,\vert\widehat{f}(\ell)\vert\le\sum\limits_{\ell}\langle -\ell\rangle^\sigma \vert -i\ell\vert^{p} \vert\widehat{v}(-\ell)\vert\vert\widehat{f}(\ell)\vert\\
&=\left(\vert\langle\cdot\rangle^\sigma\widehat{\partial_x^{p} v}\vert\ast\vert \widehat f\vert\right)(0)=\left(\vert\widehat{\langle\partial_x\rangle^\sigma\partial_x^{p}v}\vert\ast\vert \widehat f\vert\right)(0)\,.
\end{split}
\end{equation}
For $k>0$ we obtain
\begin{equation}\label{coeff_k+_ale}
\begin{split}
\langle k\rangle^\sigma\,\vert & \widehat{\left[\mathbb H\,;\,v\right]\partial^p_x f}(k)\vert\le 2\sum\limits_{\ell<0}\langle k\rangle^\sigma\,\vert\widehat{v}(k-\ell)\vert\,\vert\ell\vert^p\,\vert\widehat{f}(\ell)\vert\\
&\le 2\sum\limits_{\ell<0}\langle k-\ell\rangle^\sigma\vert i(k-\ell)\vert^{p}\,\vert\widehat{v}(k-\ell)\vert\,\vert\widehat{f}(\ell)\vert\\
&\le 2\sum\limits_{\ell< 0} \langle k-\ell\rangle^\sigma\vert\widehat{\partial^{p}_x v}(k-\ell)\vert\,\vert\widehat{f}(\ell)\vert=2(\vert\widehat{\langle\partial_x\rangle^\sigma\partial^{p}_x v}\vert\ast\vert\widehat{f}\vert)(k)\,,
\end{split}
\end{equation}
where we used that $\vert k\vert<\vert k-\ell\vert$, hence $\langle k\rangle<\langle k-\ell\rangle$, and $\vert\ell\vert<\vert k-\ell\vert$ for $k>0$ and $\ell<0$. The same estimate as above can be extended to the coefficients $\vert\widehat{\left[\mathbb H\,;\,v\right]\partial^p_x f}(k)\vert$, for $k<0$, by repeating the same arguments and since the inequalities $\vert k\vert<\vert k-\ell\vert$ and $\vert\ell\vert<\vert k-\ell\vert$ are still true for $k<0$ and $\ell>0$.

Using \eqref{coeff_0_ale}, \eqref{coeff_k+_ale} to estimate the right-hand side of \eqref{stima_ale_1}, by Young's inequality and \eqref{imm_sobolev} we obtain
\begin{equation}\label{stima_ale_2}
\begin{split}
\Vert &\left[\mathbb H\,;\,v\right]\partial^p_x f\Vert^2_{H^\sigma(\mathbb T)}\le 8\pi\sum\limits_{k\in \mathbb Z}\left(\vert\widehat{\langle\partial_x\rangle^\sigma\partial^{p}_x v}(k)\vert\ast\vert\widehat{f}(k)\vert\right)^2\le 8\pi\left(\sum\limits_{k\in \mathbb Z}\vert\widehat{\langle\partial_x\rangle^\sigma\partial^{p}_x v}(k)\vert^2\right)\left(\sum\limits_{k\in \mathbb Z}\vert\widehat{f}(k)\vert\right)^2\\
&\le C\Vert\partial^{p}_x v\Vert^2_{H^\sigma(\mathbb T)}\Vert f\Vert^2_{H^1(\mathbb T)}\,,
\end{split}
\end{equation}
with some numerical constant $C$ independent of $\sigma$ and $p$. This ends the proof of the first part of Lemma. As for the second part it is sufficient to interchange the role of $v$ and $f$ in the above application of the Young inequality, by taking the $\ell^1$-norm of  $\{\widehat{\langle\partial_x\rangle^\sigma\partial^{p}_x v}(k)\}$ and the $\ell^2$-norm of $\{\widehat{f}(k)\}$.
\end{proof}


\begin{lemma}\label{commutatore_ds}
For every real $r\ge 1$ there exists a constant $C_r>0$ such that for all $f\in H^{r-1}(\mathbb T)\cap H^1(\mathbb T)$ and $v\in H^r(\mathbb T)\cap H^2(\mathbb T)$ the following estimate holds true
\begin{equation}\label{stima_comm_ds}
\Vert \left[\langle\partial_x\rangle^r\,;\,v\right]f\Vert_{L^2(\mathbb T)}\le C_r\left\{\Vert v\Vert_{H^r(\mathbb T)}\Vert f\Vert_{H^1(\mathbb T)}+\Vert v_x\Vert_{H^1(\mathbb T)}\Vert f\Vert_{H^{r-1}(\mathbb T)}\right\}\,.
\end{equation}
\end{lemma}
\begin{proof}
For all $k\in\mathbb Z$ we compute
\begin{equation}\label{coeff_fourier_comm}
\begin{split}
\widehat{\left[\langle\partial_x\rangle^r\,;\,v\right]f}(k)&=\langle k\rangle^r\widehat{vf}(k)-\widehat{v\langle\partial_x\rangle^r f}(k)\\
&=\langle k\rangle^r\sum\limits_{\ell}\widehat{v}(k-\ell)\widehat f(\ell)-\sum\limits_{\ell}\widehat{v}(k-\ell)\langle\ell\rangle^r\widehat{f}(\ell)\\
&=\sum\limits_{\ell}\left(\langle k\rangle^r-\langle\ell\rangle^r\right)\widehat{v}(k-\ell)\widehat{f}(\ell)\,.
\end{split}
\end{equation}
On the other hand we have
\begin{equation*}
\langle k\rangle^r-\langle\ell\rangle^r=\int_0^1\frac{d}{d\theta}\left(\langle \ell+\theta(k-\ell)\rangle^r\right)\,d\theta
=(k-\ell)\int_0^1 D\left(\langle \cdot\rangle^r\right)(\ell+\theta(k-\ell))d\theta\,,
\end{equation*}
where $D$ denotes the derivative of the function $\langle \cdot\rangle^r$. Combining the preceding with the estimate
\begin{equation}\label{derivate_ds}
\left\vert\frac{d}{d\xi}\langle \xi\rangle^r\right\vert\le C_r\langle\xi\rangle^{r-1}\,,\quad\forall\,\xi\in\mathbb R\,,
\end{equation}
then gives
\begin{equation}\label{stima_diff_ds}
\vert\langle k\rangle^r-\langle\ell\rangle^r\vert\le\vert k-\ell\vert\int_0^1 \vert D\left(\langle \cdot\rangle^r\right)(\ell+\theta(k-\ell))\vert d\theta\le C_r\vert k-\ell\vert\int_0^1\langle \ell+\theta(k-\ell)\rangle^{r-1}\,d\theta\,.
\end{equation}
Using \eqref{stima_diff_ds}, from \eqref{coeff_fourier_comm} we get
\begin{equation}\label{stima_coeff_comm}
\begin{split}
\vert\widehat{\left[\langle\partial_x\rangle^r\,;\,v\right]f}(k)\vert&\le\sum\limits_{\ell}\left\vert\langle k\rangle^r-\langle\ell\rangle^r\right\vert\vert\widehat{v}(k-\ell)\vert\vert\widehat{f}(\ell)\vert\\
&\le C_r\sum\limits_{\ell}\int_0^1\vert k-\ell\vert\langle \ell+\theta(k-\ell)\rangle^{r-1}\vert\widehat{v}(k-\ell)\vert\vert\widehat{f}(\ell)\vert\,d\theta\,.
\end{split}
\end{equation}
Since the function $\langle\zeta\rangle^{r-1}$ is sub-additive and $0\le\theta\le 1$, we have
\begin{equation}\label{sub-add}
\langle\ell+\theta(k-\ell)\rangle^{r-1}\le C_r\left\{\langle\theta(k-\ell)\rangle^{r-1}+\langle\ell\rangle^{r-1}\right\}\le C_r\left\{\langle k-\ell\rangle^{r-1}+\langle\ell\rangle^{r-1}\right\}\,,
\end{equation}
with positive constant $C_r$ depending only on $r$. Using \eqref{sub-add} to estimate the right-hand side of \eqref{stima_coeff_comm} then gives
\begin{equation}\label{stima_coeff_comm1}
\begin{split}
\vert\widehat{\left[\langle\partial_x\rangle^r\,;\,v\right]f}(k)\vert&
\le C_r\sum\limits_{\ell}\left\{\int_0^1\langle k-\ell\rangle^{r}\vert\widehat{v}(k-\ell)\vert\vert\widehat{f}(\ell)\vert\,d\theta+\int_0^1\vert k-\ell\vert \vert\widehat{v}(k-\ell)\langle\ell\rangle^{r-1}\vert\vert\widehat{f}(\ell)\vert\,d\theta\right\}\\
&\le C^\prime_r\left\{\left(\vert\widehat{\langle\partial_x\rangle^r v}\vert\ast\vert\widehat{f}\vert\right) (k)+\left(\vert\widehat{v_x}\vert\ast\vert\widehat{\langle\partial_x\rangle^{r-1}f}\vert\right)(k)\right\}\,.
\end{split}
\end{equation}
Using Parseval's identity, Young's inequality with $\left\{\vert\widehat{\langle\partial_x\rangle^r v}(k)\vert\right\}\in\ell^2$, $\left\{\vert\widehat{f}(k)\vert\right\}\in\ell^1$, $\left\{\vert\widehat{v_x}(k)\vert\right\}\in\ell^1$,  $\left\{\vert\widehat{\langle\partial_x\rangle^{r-1} f}(k)\vert\right\}\in\ell^2$, \eqref{normaHs} and \eqref{imm_sobolev} with $r=1$ to estimate the $\ell^1-$norms, from \eqref{stima_coeff_comm1} we derive \eqref{stima_comm_ds}.
\end{proof}



As an application of Lemma \ref{commutatore_ds}, let us prove the following result.
\begin{corollary}\label{add_regularity}
Let $\varphi\in C([0,T]; H^{r+2}(\mathbb T))\cap C^1([0,T]; H^{r+1}(\mathbb T))$, for $r>1/2$ and $T>0$, be a solution of the equation \eqref{onde_integro_diff} on $[0,T]$. Then $\varphi\in C^2([0,T]; H^r(\mathbb T))$.
\end{corollary}
\begin{proof}
Solving the equation \eqref{onde_integro_diff}, written in the equivalent form \eqref{equ1ter}, with respect to $\varphi_{tt}$, we get
\begin{equation}\label{phitt}
\varphi_{tt}=(\mu-2\phi_x)\varphi_{xx}-\mathcal Q[\varphi]\,,
\end{equation}
where
\begin{equation*}
\mathcal Q[\varphi]=-3\left[\mathbb H\,;\,\phi_x\right]\phi_{xx}-\left[\mathbb H\,;\,\phi\right]\phi_{xxx}\,.
\end{equation*}
From the algebra property for $H^r(\mathbb T)$ with $r>1/2$ (see Lemma \ref{lemma_prod_alg}), the Sobolev continuity of $\mathbb H$ and $\varphi\in C([0,T]; H^{r+2}(\mathbb T))$  we derive that
\begin{equation*}
\phi_x\in C([0,T]; H^{r+1}(\mathbb T))\,,\qquad \varphi_{xx}\in C([0,T]; H^r(\mathbb T))\quad\Rightarrow\quad (\mu-2\phi_x)\varphi_{xx}\in C([0,T]; H^r(\mathbb T))\,.
\end{equation*}
Concerning the first commutator in the right-hand side of the expression for $\mathcal Q[\varphi]$, we write it explicitly as $-3\left[\mathbb H\,;\,\phi_x\right]\phi_{xx}=-3\mathbb H[\phi_x\phi_{xx}]+3\phi_x\mathbb H[\phi_{xx}]$ and use the same arguments as to obtain the regularity of $(\mu-2\phi_x)\varphi_{xx}$. As for the second commutator, we cannot argue as for the  previous one; instead we need to exploit the commutator structure by applying Lemma \ref{lemma_comm_ale_2} with $p=2$, $v=\phi$ and $f=\phi_x$; then we get that for each $t\in [0,T]$, $\left[\mathbb H\,;\,\phi(t)\right]\phi_{xxx}(t)\in H^r(\mathbb T)$ and satisfies the estimate
\begin{equation*}
\Vert \left[\mathbb H\,;\,\phi (t)\right]\phi_{xxx}(t)\Vert_{H^r(\mathbb T)}\le C\Vert\phi_{xx}(t)\Vert_{H^r(\mathbb T)}\Vert\phi_x(t)\Vert_{H^1(\mathbb T)}\,,\quad\forall\,t\in[0,T]\,,
\end{equation*}
with constant $C$ depending only on $r$. The time continuity of $\left[\mathbb H\,;\,\phi\right]\phi_{xxx}$, as an $H^r(\mathbb T)-$valued function, comes from the estimate
\begin{equation*}
\begin{split}
\Vert &\left[\mathbb H\,;\,\phi (t_1)\right]\phi_{xxx}(t_1)-\left[\mathbb H\,;\,\phi (t_2)\right]\phi_{xxx}(t_2)\Vert_{H^r(\mathbb T)}\\
&=\Vert \left[\mathbb H\,;\,\phi (t_1)-\phi(t_2)\right]\phi_{xxx}(t_1)+\left[\mathbb H\,;\,\phi (t_2)\right](\phi_{xxx}(t_1)-\phi_{xxx}(t_2)\Vert_{H^r(\mathbb T)}\\
&\le C\left\{\Vert\phi_{xx}(t_1)-\phi_{xx}(t_2)\Vert_{H^r(\mathbb T)}\Vert\phi_x(t_1)\Vert_{H^1(\mathbb T)}+\Vert\phi_{xx}(t_2)\Vert_{H^r(\mathbb T)}\Vert\phi_x(t_1)-\phi_x(t_2)\Vert_{H^1(\mathbb T)}\right\}\,,\quad\forall\,t_1, t_2\in[0,T]\,,
\end{split}
\end{equation*}
which is again a consequence of Lemma \ref{lemma_comm_ale_2}, and that $\phi_{xx}\in C([0,T]; H^r(\mathbb T))$, $\phi_x\in C([0,T]; H^{r+1}(\mathbb T))\hookrightarrow C([0,T]; H^{1}(\mathbb T))$. In view of \eqref{phitt} we then find that $\varphi_{tt}\in C([0,T]; H^r(\mathbb T))$ which completes the proof.
\end{proof}

In the end of this section, we recall the following algebra property of Sobolev spaces.
\begin{lemma}\label{lemma_prod_alg}
For all real $\sigma>1/2$ and $m\ge 0$, the set inclusion $H^\sigma(\mathbb T)\cdot H^m(\mathbb T)\subset H^m(\mathbb T)$ holds with continuous imbedding. In particular, for $m=\sigma>1/2$ the space $H^\sigma(\mathbb T)$ is an algebra for the point-wise product of functions.
\end{lemma}

\section{The analysis of the quadratic term $\mathcal Q[\varphi]$}\label{sec_quadratic}
This section is devoted to prove the following result.
\begin{lemma}\label{lemma_stima_quadr}
There exists a positive constant $C$ such that for every real $r\ge 0$ and for all $\varphi\in H^3(\mathbb T)\cap H^{r+1}(\mathbb T)$
\begin{equation}\label{stima_quadr}
\Vert\mathcal Q[\varphi]\Vert_{H^r(\mathbb T)}\le C\Vert\varphi_x\Vert_{H^2(\mathbb T)}\Vert\varphi_x\Vert_{H^r(\mathbb T)}\,.
\end{equation}
\end{lemma}
In order to make the estimate of Lemma \ref{lemma_stima_quadr}, it is first convenient to find out the explicit form of the Fourier coefficients of $\mathcal Q[\varphi]$. Let us recall that (cf. \eqref{termine_nl})
\begin{equation*}
\mathcal Q\left[\varphi\right]:=-3\left[\mathbb H\,;\,\phi_x\right]\phi_{xx}-\left[\mathbb H\,;\,\phi\right]\phi_{xxx}\,;
\end{equation*}
hence for all $k\in\mathbb Z$, we have
\begin{equation}\label{fourier_quadr1}
\widehat{\mathcal Q\left[\varphi\right]}(k)=-3\left(\left[\mathbb H\,;\,\phi_x\right]\phi_{xx}\right)^{\wedge}(k)-\left(\left[\mathbb H\,;\,\phi\right]\phi_{xxx}\right)^{\wedge}(k)\,.
\end{equation}
Applying \eqref{fourier1} in the right-hand side of \eqref{fourier_quadr1} with $v=\phi_x$, $u=\phi_{xx}$ in the first commutator and $v=\phi$ and $u=\phi_{xxx}$ in the second commutator gives
\begin{equation}\label{fourier_quadr2}
\begin{split}
\widehat{\mathcal Q\left[\varphi\right]}&(k)=3\sum\limits_{\ell}i({\rm sgn}\,k-{\rm sgn}\,\ell)\widehat{\phi_x}(k-\ell)\widehat{\phi_{xx}}(\ell)
+\sum\limits_{\ell}i({\rm sgn}\,k-{\rm sgn}\,\ell)\widehat{\phi}(k-\ell)\widehat{\phi_{xxx}}(\ell)\\
&=3\sum\limits_{\ell}i({\rm sgn}\,k-{\rm sgn}\,\ell)i(k-\ell)\widehat{\phi}(k-\ell)(i\ell)^2\widehat{\phi}(\ell)+\sum\limits_{\ell}i({\rm sgn}\,k-{\rm sgn}\,\ell)\widehat{\phi}(k-\ell)(i\ell)^3\widehat{\phi}(\ell)\\
&=\sum\limits_{\ell}\left\{3({\rm sgn}\,k-{\rm sgn}\,\ell)(k-\ell)\ell^2+({\rm sgn}\,k-{\rm sgn}\,\ell)\ell^3\right\}\widehat{\phi}(k-\ell)\widehat{\phi}(\ell)\\
&=-\sum\limits_{\ell}({\rm sgn}\,k-{\rm sgn}\,\ell)\,\ell^2\,(3k-2\ell)\,{\rm sgn}(k-\ell)\,{\rm sgn}\ell\,\widehat{\varphi}(k-\ell)\widehat{\varphi}(\ell)\\
&=\sum\limits_{\ell}\Lambda(k-\ell,\ell)\widehat{\varphi}(k-\ell)\widehat{\varphi}(\ell)\,,
\end{split}
\end{equation}
where
\begin{equation}\label{nucleo}
\Lambda(m,\ell):=-({\rm sgn}\,(m+\ell)-{\rm sgn}\,\ell)\,\ell^2\,(3m+\ell)\,{\rm sgn}m\,{\rm sgn}\ell\,,\quad\forall\,(m,\ell)\in\mathbb Z\times\mathbb Z\,.
\end{equation}
We notice that the ``kernel'' $\Lambda$ involved in the last expression above is not symmetric with respect to its arguments $m$ and $\ell$. In order to exploit some ``cancelation effects'' it is convenient to pass from $\Lambda$ to its symmetric counterpart
\begin{equation}\label{nucleo_simm}
\begin{split}
\widetilde{\Lambda}(m,\ell)&:=\frac12\left\{\Lambda(m,\ell)+\Lambda(\ell,m)\right\}\\
&=-\frac1{2}\left\{({\rm sgn}(m+\ell)-{\rm sgn}\ell)\,\ell^2\,(3m+\ell)+({\rm sgn}(m+\ell)-{\rm sgn}m)\,m^2\,(3\ell+m)\right\}{\rm sgn}m\,{\rm sgn}\ell\,.
\end{split}
\end{equation}
One can easily check that the $\Lambda$ can be replaced by $\widetilde{\Lambda}$ in the representation formula of the Fourier coefficients of $\mathcal Q[\varphi]$, so that
\begin{equation}\label{fourier_quadr3}
\widehat{\mathcal Q\left[\varphi\right]}(k)=\sum\limits_{\ell}\widetilde{\Lambda}(k-\ell,\ell)\widehat{\varphi}(k-\ell)\widehat{\varphi}(\ell)\,.
\end{equation}

Besides the symmetry some additional properties of the kernel $\widetilde{\Lambda}$ will be useful in the sequel of the analysis:
\begin{itemize}
\item[(i)] $\widetilde{\Lambda}(m,\ell)=0$ as long as $m\ell\ge 0$;
\item[(ii)] the real-valued kernel $\widetilde{\Lambda}$ satisfies the {\it reality condition} (see \cite{hunter}) that is
\begin{equation}\label{reality}
\widetilde{\Lambda}(-m,-\ell)=\widetilde{\Lambda}(m,\ell)\,,\quad\forall\,(m,\ell)\in\mathbb Z\times\mathbb Z\,.
\end{equation}
\end{itemize}
In view of property (i) above, for every given $k\in\mathbb Z$ the expression \eqref{fourier_quadr3} of the $k-$th Fourier coefficient of $\mathcal Q\left[\varphi\right]$ reduces to
\begin{equation}\label{fourier_quadr4}
\widehat{\mathcal Q\left[\varphi\right]}(k)=\sum\limits_{\ell\in\mathcal A_k}\widetilde{\Lambda}(k-\ell,\ell)\widehat{\varphi}(k-\ell)\widehat{\varphi}(\ell)\,,
\end{equation}
where
\begin{equation}\label{Ak}
\mathcal A_k:=\{\ell\in\mathbb Z\,:\,\,(k-\ell)\ell<0\}\,.
\end{equation}
In order to exploit the properties of $\widetilde{\Lambda}$ above, it is convenient to decompose the set of frequencies $\{(m,\ell)\in\mathbb Z\times\mathbb Z:\,\,m\ell<0\}$, where $\widetilde{\Lambda}$ is not zero, into four pairwise disjoint sub-regions below:
\begin{eqnarray*}
\mathbb F_I:=\{(m,\ell)\,:\,\,m+\ell>0\,,\,\,m<0\}\,;\label{F1}\\
\mathbb F_{II}:=\{(m,\ell)\,:\,\,m+\ell\le 0\,,\,\,m<0\}\,;\label{F2}\\
\mathbb F_{III}:=\{(m,\ell)\,:\,\,m+\ell\le 0\,,\,\,m>0\}\,;\label{F3}\\
\mathbb F_{IV}:=\{(m,\ell)\,:\,\,m+\ell>0\,,\,\,m>0\}\,.\label{F4}
\end{eqnarray*}

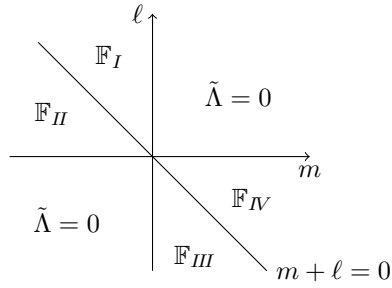
\begin{figure}[htbp]
{\centering
\begin{tikzpicture}[scale=.38]
\draw[->] (-5,0) -- (5.5,0) node[below=.0cm]{$m$};
\draw[->] (0,-4) -- (0,5) node[left]{$\ell$};
\draw[-] (-4,4) -- (4,-4) node[right]{$m+\ell=0$};
\draw (3,2.2) node{$\tilde\Lambda=0$};
\draw (-3,-2.2) node{$\tilde\Lambda=0$};
\draw (-1.5,3.5) node{$\mathbb F_{I}$};
\draw (-3.5,1.5) node{$\mathbb F_{I\!I}$};
\draw (1.5,-3.5) node{$\mathbb F_{I\!I\!I}$};
\draw (3.5,-1.5) node{$\mathbb F_{I\!V}$};
\end{tikzpicture}
}
\end{figure}


According to the above decomposition, for every fixed $k\in\mathbb Z$ one can split the index set $\mathcal A_k$ involved in the sum in the right-hand side of \eqref{fourier_quadr4} into the four pairwise disjoint subsets
\begin{eqnarray*}
\mathcal A^I_k:=\{\ell\,:\,\,(k-\ell,\ell)\in\mathbb F_I\}\,;\label{AI}\\
\mathcal A^{II}_k:=\{\ell\,:\,\,(k-\ell,\ell)\in\mathbb F_{II}\}\,;\label{AII}\\
\mathcal A^{III}_k:=\{\ell\,:\,\,(k-\ell,\ell)\in\mathbb F_{III}\}\,;\label{AIII}\\
\mathcal A^{IV}_k:=\{\ell\,:\,\,(k-\ell,\ell)\in\mathbb F_{IV}\}\;\label{AIV}
\end{eqnarray*}
correspondingly the sum in the right-hand side of \eqref{fourier_quadr4} can be split into four contributions:
\begin{equation}\label{fourier_quadr5}
\begin{split}
\widehat{\mathcal Q\left[\varphi\right]}(k)&=\sum\limits_{\ell\in\mathcal A^{I}_k}\widetilde{\Lambda}(k-\ell,\ell)\widehat{\varphi}(k-\ell)\widehat{\varphi}(\ell)+\sum\limits_{\ell\in\mathcal A^{II}_k}\widetilde{\Lambda}(k-\ell,\ell)\widehat{\varphi}(k-\ell)\widehat{\varphi}(\ell)\\
&+\sum\limits_{\ell\in\mathcal A^{III}_k}\widetilde{\Lambda}(k-\ell,\ell)\widehat{\varphi}(k-\ell)\widehat{\varphi}(\ell)+\sum\limits_{\ell\in\mathcal A^{IV}_k}\widetilde{\Lambda}(k-\ell,\ell)\widehat{\varphi}(k-\ell)\widehat{\varphi}(\ell)\\
&=I_1(k)+I_2(k)+I_3(k)+I_4(k)\,.
\end{split}
\end{equation}
It is worth noticing that for any given $k\in\mathbb Z$ some of the sets $\mathcal A^{I}_k,\dots,\mathcal A^{IV}_k$ above may become empty. For instance from \eqref{F1} one derives
\begin{equation}\label{A1_esplicito}
\ell\in\mathcal A^I_{k}\quad\Leftrightarrow\quad(k-\ell,\ell)\in\mathbb F_I\quad\Leftrightarrow\quad \ell>k>0\,.
\end{equation}
In particular the preceding implies that $\mathcal A^{I}_k=\emptyset$ for $k\le 0$. Arguing similarly on the other sets in \eqref{AII}--\eqref{AIV}, one has
\begin{eqnarray*}
\mathcal A^{I}_k=\mathcal A^{IV}_k=\emptyset\quad\mbox{if}\,\,k\le 0;\label{identita_A1}\\
\mathcal A^{II}_k=\mathcal A^{III}_k=\emptyset\quad\mbox{if}\,\,k>0\,.\label{identita_A2}
\end{eqnarray*}
Of course, when an index set $\mathcal A^{I}_k,\dots,\mathcal A^{IV}_k$ is empty it is understood that the corresponding term $I_1(k),\dots, I_4(k)$ in the right-hand side of \eqref{fourier_quadr5} is zero; then one has
\begin{itemize}
\item[1.] If $k\le 0$ then $I_1(k)=I_4(k)=0$, thus
\begin{equation}\label{decomp1}
\widehat{\mathcal Q\left[\varphi\right]}(k)=\sum\limits_{\ell\in\mathcal A^{II}_k}\widetilde{\Lambda}(k-\ell,\ell)\widehat{\varphi}(k-\ell)\widehat{\varphi}(\ell)+\sum\limits_{\ell\in\mathcal A^{III}_k}\widetilde{\Lambda}(k-\ell,\ell)\widehat{\varphi}(k-\ell)\widehat{\varphi}(\ell)=I_2(k)+I_3(k)\,.
\end{equation}
\item[2.] If $k>0$ then $I_2(k)=I_3(k)=0$ thus
\begin{equation}\label{decomp2}
\widehat{\mathcal Q\left[\varphi\right]}(k)=\sum\limits_{\ell\in\mathcal A^{I}_k}\widetilde{\Lambda}(k-\ell,\ell)\widehat{\varphi}(k-\ell)\widehat{\varphi}(\ell)+\sum\limits_{\ell\in\mathcal A^{IV}_k}\widetilde{\Lambda}(k-\ell,\ell)\widehat{\varphi}(k-\ell)\widehat{\varphi}(\ell)=I_1(k)+I_4(k)\,.
\end{equation}
\end{itemize}
In view of the symmetry of $\widetilde{\Lambda}$, that is
\begin{equation}\label{simmetria}
\widetilde{\Lambda}(m,\ell)=\widetilde{\Lambda}(\ell,m)\,,\quad\forall\,(m,\ell)\in\mathbb Z\times\mathbb Z\,,
\end{equation}
and the reality condition \eqref{reality}, we may further reduce the expression of the Fourier coefficients $\widetilde{\mathcal Q[\varphi]}(k)$ in \eqref{decomp1}, \eqref{decomp2}. More precisely, we may prove the following.
\begin{lemma}\label{lemma_simmetrie}
\begin{itemize}
\item[(1)] From the symmetry condition \eqref{simmetria} we deduce that
\begin{equation}\label{identita_simm}
I_1(k)=I_4(k)\,,\quad I_2(k)=I_3(k)\,,\quad\forall\,k\in\mathbb Z\,.
\end{equation}
\item[(2)] From \eqref{reality} we deduce that
\begin{equation}\label{identita_realta}
I_2(k)=\overline{I_1(-k)}\,,\quad\forall\,k\neq 0\,.
\end{equation}
\end{itemize}
\end{lemma}
\begin{proof}
{\it Statement (1).} Let us first observe that the frequency regions $\mathbb F_I$ and $\mathbb F_{IV}$ are symmetric with respect to the diagonal $\Delta:=\{(m,\ell)\,:\,\,m=\ell\}$, that is
\begin{equation*}
(m,\ell)\in\mathbb F_I\quad\Leftrightarrow\quad (\ell,m)\in\mathbb F_{IV}\,.
\end{equation*}
Using the above observation, together with \eqref{simmetria}, and performing a change of index $\ell^\prime=k-\ell$ for any $k\in\mathbb Z$ one computes:
\begin{equation*}
\begin{split}
I_4(k)&=\sum\limits_{\ell\,:\,\,(k-\ell,\ell)\in\mathbb F_{IV}}\widetilde{\Lambda}(k-\ell,\ell)\widehat{\varphi}(k-\ell)\widehat{\varphi}(\ell)=\sum\limits_{\ell\,:\,\,(\ell, k-\ell)\in\mathbb F_{I}}\widetilde{\Lambda}(k-\ell,\ell)\widehat{\varphi}(k-\ell)\widehat{\varphi}(\ell)\\
&=\sum\limits_{\ell\,:\,\,(\ell, k-\ell)\in\mathbb F_{I}}\widetilde{\Lambda}(\ell, k-\ell)\widehat{\varphi}(k-\ell)\widehat{\varphi}(\ell)=\sum\limits_{\ell^\prime\,:\,\,(k-\ell^\prime, \ell^\prime)\in\mathbb F_{I}}\widetilde{\Lambda}(k-\ell^\prime,\ell^\prime)\widehat{\varphi}(\ell^\prime)\widehat{\varphi}(k-\ell^\prime)=I_1(k)\,.
\end{split}
\end{equation*}
This establishes the first equality in \eqref{identita_simm}; the second equality in \eqref{identita_simm} follows from the same arguments above, after observing that the frequency sets $\mathbb F_{II}$ and $\mathbb F_{III}$ are symmetric with respect to $\Delta$.

{\it Statement (2).} Let us first observe that the identity \eqref{identita_realta} becomes trivial for $k>0$, since in this case one has $I_2(k)=0$ and $I_1(-k)=0$ (see \eqref{identita_A1}).

For $k<0$, we first get
\begin{equation*}
I_1(-k)=\sum\limits_{\ell\,:(-k-\ell,\ell)\in\mathbb F_{I}}\widetilde{\Lambda}(-k-\ell,\ell)\widehat{\varphi}(-k-\ell)\widehat{\varphi}(\ell)\,,
\end{equation*}
hence using that $\widetilde{\Lambda}$ satisfies \eqref{reality} and is real-valued and that $\varphi$ is also real-valued so that $\overline{\widehat{\varphi}(k)}=\widehat{\varphi}(-k)$ we compute
\begin{equation*}
\begin{split}
\overline{I_1(-k)}&=\sum\limits_{\ell\,:(-k-\ell,\ell)\in\mathbb F_{I}}\overline{\widetilde{\Lambda}(-k-\ell,\ell)}\,\,\overline{\widehat{\varphi}(-k-\ell)}\,\,\overline{\widehat{\varphi}(\ell)}\\
&=\sum\limits_{\ell\,:(-k-\ell,\ell)\in\mathbb F_{I}}\widetilde{\Lambda}(k+\ell,-\ell)\widehat{\varphi}(k+\ell)\widehat{\varphi}(-\ell)\,.
\end{split}
\end{equation*}
Now we notice that the frequency sets $\mathbb F_I$ and $\mathbb F_{III}$ are symmetric with respect to the diagonal $\Delta$, up to the point of the line $m+\ell=0$, that is
\begin{equation*}
(m,\ell)\in\mathbb F_I\quad\Leftrightarrow\quad (-m,-\ell)\in\mathbb F_{III}\setminus\{m+\ell=0\}\,.
\end{equation*}
This yields that for $k<0$ and $\ell\in\mathbb Z$
\begin{equation}\label{simmetria_I_III}
(-k-\ell,\ell)\in\mathbb F_I\quad\Leftrightarrow\quad (k+\ell,-\ell)\in\mathbb F_{III}\,;
\end{equation}
we notice that for $k<0$ the point $(k+\ell,-\ell)$ never belongs to the set $\{(m,\ell):\,\,m+\ell=0\}$. Coming back to the expression of $\overline{I_1(-k)}$ and performing the change of index $\ell^\prime=-\ell$ we may further write
\begin{equation*}
\begin{split}
\overline{I_1(-k)}&=\sum\limits_{\ell\,:(-k-\ell,\ell)\in\mathbb F_{I}}\widetilde{\Lambda}(k+\ell,-\ell)\widehat{\varphi}(k+\ell)\widehat{\varphi}(-\ell)=\sum\limits_{\ell\,:(k+\ell,-\ell)\in\mathbb F_{III}}\widetilde{\Lambda}(k+\ell,-\ell)\widehat{\varphi}(k+\ell)\widehat{\varphi}(-\ell)\\
&=\sum\limits_{\ell\,:(k-\ell^\prime,\ell^\prime)\in\mathbb F_{III}}\widetilde{\Lambda}(k-\ell^\prime,\ell^\prime)\widehat{\varphi}(k-\ell^\prime)\widehat{\varphi}(\ell^\prime)=I_3(k)\,.
\end{split}
\end{equation*}
Since we know that $I_3(k)=I_2(k)$ for all $k$ (see the second identity in \eqref{identita_simm}), we get
\begin{equation*}
\overline{I_1(-k)}=I_2(k)\,,\quad\mbox{for}\,\,k<0\,,
\end{equation*}
which completes the proof of \eqref{identita_realta}.
\end{proof}
As an immediate consequence of Lemma \ref{lemma_simmetrie} we obtain our final form of the Fourier coefficients of $\mathcal Q[\varphi]$.
\begin{corollary}\label{cor_Qphi}
For all $k\neq 0$
\begin{equation}\label{coeff_Qphi1}
\widehat{\mathcal Q[\varphi]}(k)=2I_1(k)+2\overline{I_1(-k)}\,.
\end{equation}
In particular, according to \eqref{decomp1}, \eqref{decomp2} we have
\begin{equation}\label{coeff_Qphi2}
\widehat{\mathcal Q[\varphi]}(k)=
\begin{cases}
2I_1(k)\,,\quad\mbox{if}\,\,k>0\\
\\
2\overline{I_1(-k)}\,,\quad\mbox{if}\,\,k<0\,.
\end{cases}
\end{equation}
\end{corollary}
Let us observe in the end that from \eqref{nucleo_simm} we see that in $\mathbb F_I=\{(m,\ell)\,:\,\,m+\ell>0\,,\,\,m<0\}$ the kernel $\widetilde{\Lambda}$ reduces to
\begin{equation}\label{nucleo_simm_F1}
\widetilde{\Lambda}(m,\ell)=m^2(3\ell+m)\,,
\end{equation}
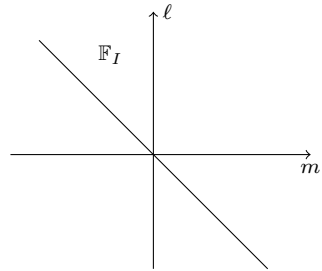
\begin{figure}[htbp]
{\centering
\begin{tikzpicture}[scale=.38]
\draw[->] (-5,0) -- (5.5,0) node[below=.0cm]{\footnotesize$ m$};
\draw[->] (0,-4) -- (0,5) node[right]{\footnotesize$\ell $};
\draw[-] (-4,4) -- (4,-4) node[right]{$ $};
\draw (-1.5,3.5) node{ \footnotesize{$\mathbb{F}_{I}$}};
\draw (15,0) node{\footnotesize$\tilde\Lambda(m,\ell)=m^2(3\ell+m)  \quad{\rm if}\;\; (m,\ell)\in \mathbb F_{I}$};
\end{tikzpicture}
}
\end{figure}


Consequently, from \eqref{fourier_quadr5}, \eqref{AI}, \eqref{A1_esplicito} we deduce that
\begin{equation}\label{I1_esplicito}
I_1(k)=\sum\limits_{\ell\,:\,\,\ell>k}(k-\ell)^2(k+2\ell)\widehat{\varphi}(k-\ell)\widehat{\varphi}(\ell)\,,\quad\mbox{for}\,\,k>0\,,
\end{equation}
hence from \eqref{coeff_Qphi2}
\begin{equation}\label{coeff_Qphi2_esplicito}
\widehat{\mathcal Q[\varphi]}(k)=
\begin{cases} \displaystyle 2\sum\limits_{\ell\,:\,\,\ell>k}(k-\ell)^2(k+2\ell)\widehat{\varphi}(k-\ell)\widehat{\varphi}(\ell)\,,\quad\mbox{if}\,\,k>0\\
\\
\begin{split}& 2\sum\limits_{\ell\,:\,\,\ell>-k}(-k-\ell)^2(-k+2\ell)\overline{ \widehat{\varphi}(-k-\ell)}\,\,\overline{\widehat{\varphi}(\ell)}\\
&=2\sum\limits_{\ell\,:\,\,\ell>-k}(k+\ell)^2(-k+2\ell) \widehat{\varphi}(k+\ell)\,\,\widehat{\varphi}(-\ell)\quad\mbox{if}\,\,k<0\,.
\end{split}
\end{cases}
\end{equation}

\begin{remark}\label{rmk:5}
Notice that the formulas \eqref{coeff_Qphi1}, \eqref{coeff_Qphi2} are in agreement with the fact that the quadratic term $\mathcal Q[\varphi]$ is a real-valued function of $x$ (since $\varphi=\varphi(x)$ is too). Indeed \eqref{coeff_Qphi1} yields at once
\begin{equation*}
\overline{\widehat{\mathcal Q[\varphi]}(k)}=2\overline{I_1(k)}+2I_1(-k)=\widehat{\mathcal Q[\varphi]}(-k)\,,\quad\mbox{for}\,\,k\neq 0\,.
\end{equation*}
Let us even observe that formula \eqref{coeff_Qphi1} does not hold for $k=0$; indeed if such formula would be true one should have that
\begin{equation*}
\widehat{\mathcal Q[\varphi]}(0)=2I_1(0)+2\overline{I_1(0)}=0\,,
\end{equation*}
since $I_1(0)=0$; an explicit computation gives for the $0-$th Fourier coefficient of $\mathcal Q[\varphi]$ the expression
\begin{equation}\label{coeff_Q0}
\widehat{\mathcal Q[\varphi]}(0)=\sum\limits_{\ell}\widetilde{\Lambda}(-\ell,\ell)\widehat{\varphi}(-\ell)\widehat{\varphi}(\ell)
=2\sum\limits_{\ell}\vert\ell\vert^3\widehat{\varphi}(-\ell)\widehat{\varphi}(\ell)\,.
\end{equation}
The reason why formula \eqref{coeff_Qphi1} cannot be extended to $k=0$ is that such formula was obtained using the identity \eqref{identita_realta}. The latter is not true for $k=0$ since in this case the equivalence (needed to prove \eqref{identita_realta})
\begin{equation*}
(-k-\ell,\ell)=(-\ell,\ell)\in\mathbb F_I\quad\Leftrightarrow\quad (k+\ell,-\ell)=(\ell,-\ell)\in\mathbb F_{III}
\end{equation*}
does not hold; indeed, as long as $\ell>0$, the point $(\ell,-\ell)$ belongs to $\mathbb F_{III}$ whereas $(-\ell,\ell)$ does not belong to $\mathbb F_I$, see \eqref{F1}, \eqref{F3}.
\end{remark}

\subsection{Proof of Lemma \ref{lemma_stima_quadr}}\label{proof_stima_quadr}
Along the proof, we restore the notation $\psi=\langle\partial_x\rangle^r\varphi$ (cf. \eqref{derivatas}) for short. By \eqref{normaHs}
\begin{equation}\label{parseval}
\Vert\mathcal Q[\varphi]\Vert^2_{H^r(\mathbb T)}=2\pi\sum\limits_{k}\langle k\rangle^{2r}\,\vert\widehat{\mathcal Q[\varphi]}(k)\vert^2\,.
\end{equation}
Then we need an estimate of $\vert\langle k\rangle^r\,\widehat{\mathcal Q[\varphi]}(k)\vert$ for each $k$.

For $k>0$, from \eqref{coeff_Qphi2_esplicito} and \eqref{derivatas} we get
\begin{equation}\label{stima1_Qphi2}
\begin{split}
\langle k\rangle^r\,\vert\widehat{\mathcal Q[\varphi]}(k)\vert &\le 2\sum\limits_{\ell\,:\,\,\ell>k}\langle k\rangle^r(k-\ell)^2(k+2\ell)\vert\widehat{\varphi}(k-\ell)\vert\,\vert\widehat{\varphi}(\ell)\vert\\
&=2\sum\limits_{\ell\,:\,\,\ell>k}\left(\frac{\langle k\rangle}{\langle \ell\rangle}\right)^r(k-\ell)^2(k+2\ell)\vert\widehat{\varphi}(k-\ell)\vert\,\vert\widehat{\psi}(\ell)\vert\,.
\end{split}
\end{equation}
Now $\ell>k>0$ trivially implies $\langle\ell\rangle>\langle k\rangle$, hence
\begin{equation}\label{stime_nucleo+}
\left(\frac{\langle k\rangle}{\langle \ell\rangle}\right)^r\le 1\,\,(\mbox{since}\,\,r\ge 0)\qquad\mbox{and}\qquad k+2\ell<3\ell=3\vert\ell\vert\,.
\end{equation}
Then using \eqref{stime_nucleo+} in \eqref{stima1_Qphi2} we get
\begin{equation}\label{stima2_Qphi2}
\begin{split}
\langle k\rangle^r\,\vert\widehat{\mathcal Q[\varphi]}(k)\vert&\le 6\sum\limits_{\ell\,:\,\,\ell>k}(k-\ell)^2\vert\ell\vert\vert\widehat{\varphi}(k-\ell)\vert\,\vert\widehat{\psi}(\ell)\vert\\
&=6\sum\limits_{\ell\,:\,\,\ell>k}\vert\widehat{\varphi_{xx}}(k-\ell)\vert\,\vert\widehat{\psi_x}(\ell)\vert
\\
&\le 6\sum\limits_{\ell}\vert\widehat{\varphi_{xx}}(k-\ell)\vert\,\vert\widehat{\psi_x}(\ell)\vert=6\left(\vert\widehat{\varphi_{xx}}\vert\ast\vert\widehat{\psi_x}\vert\right)(k)\,.
\end{split}
\end{equation}
For $k<0$ we argue from \eqref{coeff_Qphi2_esplicito} similarly as for positive $k$ to get
\begin{equation}\label{stima3_Qphi2}
\begin{split}
\langle k\rangle^r\,\vert\widehat{\mathcal Q[\varphi]}(k)\vert &\le 2\sum\limits_{\ell\,:\,\,\ell>-k}\langle k\rangle^r(k+\ell)^2(-k+2\ell) \vert\widehat{\varphi}(k+\ell)\vert\,\vert\widehat{\varphi}(-\ell)\vert\\
&=2\sum\limits_{\ell^\prime\,:\,\,\ell^\prime<k}\langle k\rangle^r(k-\ell^\prime)^2(-k-2\ell^\prime) \vert\widehat{\varphi}(k-\ell^\prime)\vert\,\vert\widehat{\varphi}(\ell^\prime)\vert\\
&=2\sum\limits_{\ell^\prime\,:\,\,\ell^\prime<k}\left(\frac{\langle k\rangle}{\langle \ell^\prime\rangle}\right)^r(k-\ell^\prime)^2(-k-2\ell^\prime) \vert\widehat{\varphi}(k-\ell^\prime)\vert\,\vert\widehat{\psi}(\ell^\prime)\vert\,,
\end{split}
\end{equation}
where we have performed the change of index $\ell^\prime=-\ell$. Since now $\ell^\prime<k<0$ implies $\vert\ell^\prime\vert>\vert k\vert$ as before we get
\begin{equation}\label{stime_nucleo-}
\left(\frac{\langle k\rangle}{\langle \ell^\prime\rangle}\right)^r\le 1\,\,(\mbox{since}\,\,r\ge 0)\qquad\mbox{and}\qquad -k-2\ell^\prime<-3\ell^\prime=3\vert\ell^\prime\vert\,,
\end{equation}
hence, similarly as for $k$ positive, we find again
\begin{equation}\label{stima4_Qphi2}
\begin{split}
\langle k\rangle^r\,\vert\widehat{\mathcal Q[\varphi]}(k)\vert\le 6\left(\vert\widehat{\varphi_{xx}}\vert\ast\vert\widehat{\psi_x}\vert\right)(k)\,.
\end{split}
\end{equation}
Concerning the case $k=0$, from \eqref{coeff_Q0} and the trivial inequality $\langle\ell\rangle\ge 1$ we immediately get
\begin{equation}\label{stima5_Qphi2}
\begin{split}
\langle 0\rangle^r\vert\widehat{\mathcal{Q}[\varphi]}(0)\vert &=\vert\widehat{\mathcal{Q}[\varphi]}(0)\vert\\
&\le 2\sum\limits_{\ell}\vert\ell\vert^3\vert\widehat{\varphi}(-\ell)\vert\,\vert\widehat{\varphi}(\ell)\vert
=2\sum\limits_{\ell}\vert-\ell\vert^2\vert\widehat{\varphi}(-\ell)\vert\,\vert\ell\vert\frac{\vert\widehat{\psi}(\ell)\vert}{\langle\ell\rangle^r}\\
&\le 2\sum\limits_{\ell}\vert\widehat{\varphi_{xx}}(-\ell)\vert\,\vert\widehat{\psi_x}(\ell)\vert=2\left(\vert\widehat{\varphi_{xx}}\vert\ast\vert\widehat{\psi_x}\vert\right)(0)\,.
\end{split}
\end{equation}
In view of the previous calculations, we see that
\begin{equation}\label{stima6_Qphi2}
\langle k\rangle^r\vert\widehat{\mathcal{Q}[\varphi]}(k)\vert\le 6\left(\vert\widehat{\varphi_{xx}}\vert\ast\vert\widehat{\psi_x}\vert\right)(k)\,,\quad\forall\,k\in\mathbb Z\,.
\end{equation}
From \eqref{parseval} and \eqref{stima6_Qphi2} and applying Young's inequality to $\{\vert\widehat{\varphi_{xx}}\vert\}\in\ell^1$ and $\{\vert\widehat{\psi_x}\vert\}\in\ell^2$ and \eqref{imm_sobolev} for $H^1(\mathbb T)$, we finally obtain
\begin{equation}\label{stimafin_Qphi2}
\begin{split}
\Vert\langle\partial_x\rangle^r &\mathcal Q[\varphi]\Vert^2_{L^2(\mathbb T)}\le 72\pi\sum\limits_{k}\left\vert\left(\vert\widehat{\varphi_{xx}}\vert\ast\vert\widehat{\psi_x}\vert\right)(k)\right\vert^2
=72\pi\Vert\{\vert\widehat{\varphi_{xx}}\vert\ast\vert\widehat{\psi_x}\vert\}\Vert^2_{\ell^2}\\
&\le 72\pi\Vert\{\vert\widehat{\varphi_{xx}}\vert\}\Vert^2_{\ell^1}\Vert\{\vert\widehat{\psi_x}\vert\}\Vert^2_{\ell^2}\le C\Vert\varphi_{xx}\Vert^2_{H^1(\mathbb T)}\Vert\psi_x\Vert^2_{L^2(\mathbb T)}\le C\Vert\varphi_{x}\Vert^2_{H^2(\mathbb T)}\Vert\varphi_x\Vert^2_{H^r(\mathbb T)}\,,
\end{split}
\end{equation}
with some positive numerical constant $C$ independent of $r$. This completes the proof of Lemma \ref{lemma_stima_quadr}.

\section{Lower semi-continuity of the energy}\label{Ioffe}
Here we recall the statement of a version of Ioffe  \cite[Theorem 5.8, pag 267]{Ambrosio} (see also \cite[Theorem 21, pg 171]{Ekeland-et-all})
\begin{theorem}\label{teorema_ioffe}
Let $\Omega\subset \mathbb R^N$ be an open bounded set. Let $f:\Omega\times\mathbb R^{n+k}\rightarrow [0,+\infty]$ be a normal function (i.e $f(\cdot,s,z)$ is measurable, $f(x,\cdot,\cdot)$ is lower semicontinuous for a.e. $x$) and assume that $z\rightarrow f(x,s,z)$ is convex in $\mathbb R^k$ for any $x\in \Omega$ and any $s\in \mathbb R^n$. Then
\begin{equation*}
\int_{\Omega}f(x,u,v)\,dx\leq \liminf\limits_{m \to +\infty} \int_{\Omega}f(x,u_m,v_m)\,dx\,,
\end{equation*}
whenever $\{u_m\}\subset [L^1(\Omega)]^n$ strongly converges to $u$ and $\{v_m\}\subset [L^1(\Omega)]^k$ weakly converges to $v$.
\end{theorem}
We apply this theorem to the study of the lower semi-continuity of the following functional
\begin{equation}\label{funzionale_int}
\mathcal G(u,v)=\int_{\mathbb T}\max\left\{{\delta},\,\mu-2u(x)\right\}\vert v(x)\vert^2\,dx
\end{equation}
which comes from the energy \eqref{ym} involved in the estimates \eqref{dis_ym}, \eqref{dis_int_ym} (see Section \ref{sec_existence})\footnote{The reason why in the definition of $\mathcal G$ we consider $\max\left\{ \delta,\,\mu-2u(x)\right\}$ (instead of $\mu-2u(x)$ alone, in agreement with \eqref{ym}) is technical: it guarantees the non negativity of the function under the integral sign in \eqref{funzionale_int} as required in Theorem \ref{teorema_ioffe}. One should avoid such technical requirement by considering a more sophisticated version of Ioffe's Theorem (see \cite[Theorem 21]{Ekeland-et-all}).}. With respect to the general form of the functional considered in the statement of Theorem \ref{teorema_ioffe}, here $\Omega=\mathbb T$ and $N=n=k=1$ and the function $f(x,s,z)$ is defined by
\begin{equation*}
f(x,s,z)=F(s,z)=\max\left\{\delta,\,\mu-2s\right\}\vert z\vert^2\,.
\end{equation*}
Let us notice that all the assumptions of Theorem \ref{teorema_ioffe} about the function $f(x,s,z)$ are satisfied by $F(s,z)$ defined above; in particular, $F$ is independent of $x$ and continuous with respect to $(s,z)$; moreover the function $z\mapsto F(s,z)$ is convex on $\mathbb R$ for every $s\in\mathbb R$, and
\begin{equation*}
F(s,z)\ge \delta \vert z\vert^2
\end{equation*}
yields that $F$ is valued in $[0,+\infty[$.

In terms of the functional $\mathcal G$, the energy \eqref{ym} can be restated as
\begin{equation}\label{Gym}
\begin{split}
\mathcal E(\varphi_m(t))&=\Vert\varphi_{m,t}(t)\Vert^2_{H^2(\mathbb T)}+\int_{\mathbb T}(\mu-2\phi_{m,x}(t))\vert \langle\partial_x\rangle^2\varphi_{m,x}(t)\vert^2\,dx\\
&=\Vert\varphi_{m,t}(t)\Vert^2_{H^2(\mathbb T)}+\mathcal G\left(\phi_{m,x}(t), \langle\partial_x\rangle^2\varphi_{m,x}(t)\right)\,,
\end{split}
\end{equation}
where we recall that along the sequence $\{\varphi_m\}$ the sign condition \eqref{131} is satisfied on $[0,T]$ (hence $\mu-2\phi_{m,x}(t)=\max\left\{ \delta,\,\mu-2\phi_{m,x}(t)\right\}$).

Using now that the sequences $\{\varphi_{m}(t)\}$ and $\{\varphi_{m,t}(t)\}$ are weakly convergent respectively in $H^3(\mathbb T)$ and $H^2(\mathbb T)$, for all $t\in[0,T]$ (see \eqref{conv_m}), and applying Theorem \ref{teorema_ioffe}, we are able to prove the following result.

\begin{corollary}\label{cor_ioffe}
Assume that for all $t\in[0,T]$
\begin{equation*}
\varphi_m(t)\rightharpoonup\varphi(t)\,\,\,\mbox{w-}H^3(\mathbb T)\quad\mbox{and}\quad\varphi_{m,t}(t)\rightharpoonup\varphi_t(t)\,\,\,\mbox{w-}H^2(\mathbb T)\,.
\end{equation*}
Then
\begin{equation}\label{lscE}
\mathcal E(\varphi(t))\le\liminf\limits_{m\to +\infty}\mathcal E(\varphi_m(t))\,.
\end{equation}
\end{corollary}
\begin{proof}
From the weak convergence of $\{\varphi_{m,t}(t)\}$ to $\varphi_t(t)$ in $H^2(\mathbb T)$ and the lower semi-continuity of the norm in $H^2(\mathbb T)$ we deduce that
\begin{equation}\label{lsc1}
\Vert\varphi_t(t)\Vert^2_{H^2(\mathbb T)}\le\liminf\limits_{m\to +\infty}\Vert\varphi_{m,t}(t)\Vert^2_{H^2(\mathbb T)}\,.
\end{equation}
From the weak convergence of $\{\varphi_{m}(t)\}$ to $\varphi(t)$ in $H^3(\mathbb T)$, the compactness of the imbedding $H^2(\mathbb T)\hookrightarrow L^2(\mathbb T)$ and the $L^2-$continuity of $\mathbb H$ we also deduce that
\begin{equation*}
\varphi_{m,x}(t)\rightarrow\varphi_{x}(t)\,\,\,\mbox{s-}L^2(\mathbb T)\quad\Rightarrow\quad \phi_{m,x}(t)\rightarrow\phi_{x}(t)\,\,\,\mbox{s-}L^2(\mathbb T)
\end{equation*}
and
\begin{equation*}
\langle\partial_x\rangle^2\varphi_{m,x}(t)\rightharpoonup\langle\partial_x\rangle^2\varphi_x(t)\,\,\,\mbox{w-}L^2(\mathbb T)\,.
\end{equation*}
We are now in the position to apply to $\mathcal G$ the result of Ioffe Theorem \ref{teorema_ioffe} with $u_m=\mu-2\phi_{m,x}(t)$, $v_m=\langle\partial_x\rangle^2\varphi_{m,x}(t)$, $u=\mu-2\phi_{x}(t)$, and $v=\langle\partial_x\rangle^2\varphi_{x}(t)$. From the previous arguments we have that
\begin{equation*}
\begin{split}
u_m\rightarrow u\,\,\,\mbox{s-}L^2(\mathbb T)\quad\Rightarrow\quad u_m\rightarrow u\,\,\,\mbox{s-}L^1(\mathbb T)\,,\\
v_m\rightharpoonup v \,\,\,\mbox{w-}L^2(\mathbb T)\quad\Rightarrow\quad v_m\rightharpoonup v\,\,\,\mbox{w-}L^1(\mathbb T)\,,
\end{split}
\end{equation*}
because $L^2(\mathbb T)\subset L^1(\mathbb T)$. Then from Theorem \ref{teorema_ioffe} we get
\begin{equation}\label{lsc2}
\begin{split}
\int_{\mathbb T}(\mu-2\phi_{x}(t))&\vert \langle\partial_x\rangle^2\varphi_{x}(t)\vert^2\,dx=\mathcal G(u,v)\\
&\le\liminf\limits_{m\to +\infty}\mathcal G(u_m,v_m)=\liminf\limits_{m\to +\infty}\int_{\mathbb T}(\mu-2\phi_{m,x}(t))\vert \langle\partial_x\rangle^2\varphi_{m,x}(t)\vert^2\,dx\,.
\end{split}
\end{equation}
Adding \eqref{lsc1}, \eqref{lsc2} we get \eqref{lscE},
and the proof is complete.
\end{proof}

\subsection{Convergence in energy and strong convergence}
Let $\varphi\in C_{w}([0,T]; H^3(\mathbb T))\cap C^1_w([0,T]; H^2(\mathbb T))$ be the weak limit of the sequence $\{\varphi_m\}$ of approximating solutions of problem \eqref{onde_integro_diff}, \eqref{id} constructed in the proof of Theorem \ref{th_esistenza}. This section is devoted to the proof that the weak convergence
\begin{equation}\label{_conv_debole_nuova}
\varphi(t)\rightharpoonup\varphi(0)\,\,\,\mbox{w-}H^3(\mathbb T)\quad\mbox{and}\quad\varphi_{t}(t)\rightharpoonup\varphi_t(0)\,\,\,\mbox{w-}H^2(\mathbb T)
\end{equation}
and the convergence in energy
\begin{equation}\label{conv_energia}
\lim\limits_{t\to 0^+}\mathcal E(\varphi(t))=\mathcal E(\varphi(0))
\end{equation}
imply the strong convergence
\begin{equation*}
\varphi(t)\rightarrow\varphi(0)\,\,\,\mbox{s-}H^3(\mathbb T)\quad\mbox{and}\quad\varphi_{t}(t)\rightarrow\varphi_t(0)\,\,\,\mbox{s-}H^2(\mathbb T)\,.
\end{equation*}
The following elementary result on real sequences will be used.
\begin{lemma}\label{successioni}
Let $\{a_m\}$, $\{b_m\}$ be two real-valued sequences such that
\begin{equation*}
\liminf\limits_{m\to +\infty}a_m\ge a\quad\mbox{and}\quad \liminf\limits_{m\to +\infty}b_m\ge b\,,
\end{equation*}
and
\begin{equation*}
\lim\limits_{m\to +\infty}(a_m+b_m)=a+b\,,
\end{equation*}
for $a, b\in\mathbb R$. Then
\begin{equation*}
\lim\limits_{m\to +\infty}a_m=a\quad\mbox{and}\quad \lim\limits_{m\to +\infty}b_m=b\,.
\end{equation*}
\end{lemma}
\begin{proof}
From the assumptions it easily follows that
\begin{equation*}
\liminf\limits_{m\to +\infty}(a_m+b_m)\ge\liminf\limits_{m\to +\infty}a_m+\liminf\limits_{m\to +\infty}b_m\ge a+b=\lim\limits_{m\to +\infty}(a_m+b_m)=\limsup\limits_{m\to +\infty}(a_m+b_m)\,,
\end{equation*}
which implies
\begin{equation}\label{limite1}
a+b=\liminf\limits_{m\to +\infty}a_m+\liminf\limits_{m\to +\infty}b_m\,.
\end{equation}
This gives
\begin{equation}\label{limite2}
a=\liminf\limits_{m\to +\infty}a_m\quad\mbox{and}\quad b=\liminf\limits_{m\to +\infty}b_m\,.
\end{equation}
By contradiction, let us suppose that $\{a_m\}$ does not converge to $a$; then there exists a subsequence $\{a_{m_k}\}$ such that
\begin{equation*}
a_{m_k}\rightarrow\widetilde a\neq a\,.
\end{equation*}
Because of \eqref{limite2}, this implies that $\widetilde{a}>a$. Considering also the subsequence $\{b_{m_k}\}$ of $\{b_m\}$ (with the same indices $m_k$ as in $\{a_{m_k}\}$) and repeating on $\{a_{m_k}\}$ and $\{b_{m_k}\}$ the arguments above we find that
\begin{equation*}
a+b=\liminf\limits_{k\to +\infty}a_{m_k}+\liminf\limits_{k\to +\infty}b_{m_k}\,,
\end{equation*}
hence
\begin{equation*}
a=\liminf\limits_{k\to +\infty}a_{m_k}\quad\mbox{and}\quad b=\liminf\limits_{k\to +\infty}b_{m_k}\,.
\end{equation*}
This leads to a contradiction, since $\liminf\limits_{k\to +\infty}a_{m_k}\equiv\lim\limits_{k\to +\infty}a_{m_k}=\widetilde a>a$. Thus every subsequence of $\{a_m\}$ is convergent to $a$, which gives that $\lim\limits_{m\to +\infty}a_m=a$. The same argument applies to prove that $\lim\limits_{m\to +\infty}b_m=b$.
\end{proof}
We will also use the following result on weak and strong $L^2-$convergence.
\begin{lemma}\label{weak/strong}
Let $\{u_m\}$ and $\{v_m\}$ be two sequences in $L^2(\Omega)$ where $\Omega$ is an open bounded subset of $\mathbb R^n$. Assume that there exist $u, v\in L^2(\Omega)$ and positive constants $\delta$ and $C$ such that
\begin{itemize}
\item[a.] $v_m\rightarrow v\,,\quad\mbox{a.e. in}\,\,\Omega\,;$
\item[b.] $\delta\le v_m\le C$ in $\Omega$ for all $m$;
\item[c.] $u_m\rightharpoonup u\,,\quad\mbox{w-}L^2(\Omega)\,$;
\item[d.] $\lim\limits_{m\to +\infty}\int_{\Omega}v_m u_m^2\,dx=\int_{\Omega}v u^2\,dx$.
\end{itemize}
Then
\begin{equation*}
u_m\rightarrow u\,,\quad\mbox{s-}L^2(\Omega)\,.
\end{equation*}
\end{lemma}
\begin{proof}
Using b. we get
\begin{equation*}
\begin{split}
\limsup\limits_{m\to +\infty}&\,\delta\,\int_{\Omega}(u_m-u)^2dx\le \limsup\limits_{m\to +\infty}\int_{\Omega}v_m(u_m-u)^2dx\\
&=\limsup\limits_{m\to +\infty}\left(\int_{\Omega}v_mu_m^2dx+\int_{\Omega}v_mu^2dx-2\int_{\Omega}v_muu_m\,dx\right)\,.
\end{split}
\end{equation*}
The assumption d. gives that
\begin{equation}\label{conv1}
\limsup\limits_{m\to +\infty}\int_{\Omega}v_mu_m^2dx=\lim\limits_{m\to +\infty}\int_{\Omega}v_mu_m^2dx=\int_{\Omega}vu^2dx\,.
\end{equation}
From assumptions a., b. and using that $u\in L^2(\Omega)$ we have
\begin{equation*}
\begin{split}
& v_mu^2\rightarrow vu^2\quad\mbox{a.e. in}\,\,\Omega\,; \\
&\vert v_mu^2\vert\le C u^2\in L^1(\Omega)\,,
\end{split}
\end{equation*}
then from the dominated convergence theorem we get
\begin{equation}\label{conv2}
\limsup\limits_{m\to +\infty}\int_{\Omega}v_mu^2dx=\lim\limits_{m\to +\infty}\int_{\Omega}v_mu^2dx=\int_{\Omega}vu^2dx\,.
\end{equation}
Again from a. we get
\begin{equation*}
\begin{split}
& v_mu\rightarrow vu\quad\mbox{a.e. in}\,\,\Omega\,; \\
&\vert v_mu\vert\le C \vert u\vert\in L^2(\Omega)\,,
\end{split}
\end{equation*}
hence again by the dominated convergence theorem we get
\begin{equation*}
v_m u\rightarrow vu\,,\quad\mbox{s-}L^2(\Omega)\,.
\end{equation*}
Then from c. we derive
\begin{equation}\label{conv3}
\limsup\limits_{m\to +\infty}-2\int_{\Omega}v_mu u_mdx=\lim\limits_{m\to +\infty}-2\int_{\Omega}v_mu u_mdx=-2\int_{\Omega}vu^2dx\,.
\end{equation}
From \eqref{conv1}-\eqref{conv3} we obtain
\begin{equation*}
\limsup\limits_{m\to +\infty}\delta\int_{\Omega}(u_m-u)^2dx\le 0\,,
\end{equation*}
which gives, since $\delta>0$,
\begin{equation*}
\limsup\limits_{m\to +\infty}\int_{\Omega}(u_m-u)^2dx\le 0\,.
\end{equation*}
On the other hand,
\begin{equation*}
\liminf\limits_{m\to +\infty}\int_{\Omega}(u_m-u)^2dx\ge 0\,.
\end{equation*}
Then
\begin{equation*}
\lim\limits_{m\to +\infty}\int_{\Omega}(u_m-u)^2dx=0\,,
\end{equation*}
which proves the result.
\end{proof}
Now we are in the position to prove the result announced at the beginning of this section.

First, we notice that from the weak convergence in \eqref{_conv_debole_nuova} and the lower semi-continuity of the $H^2-$norm and the functional $\mathcal G$ defined by \eqref{funzionale_int} we obtain that
\begin{equation}\label{liminf_phi}
\begin{split}
&\liminf\limits_{t\to 0^+}\Vert\varphi_t(t)\Vert^2_{H^2(\mathbb T)}\ge\Vert\varphi_t(0)\Vert^2_{H^2(\mathbb T)}\,\\
&\liminf\limits_{t\to 0^+}\int_{\mathbb T}(\mu-2\phi_{x}(t))\vert\langle\partial_x\rangle^2\varphi_x(t)\vert^2\,dx\ge \int_{\mathbb T}(\mu-2\phi_{x}(0))\vert\langle\partial_x\rangle^2\varphi_x(0)\vert^2\,dx\,.
\end{split}
\end{equation}
In view of the convergence \eqref{conv_energia}, we can apply the result of Lemma \ref{successioni} to find that
\begin{equation}\label{lim_phi}
\begin{split}
&\lim\limits_{t\to 0^+}\Vert\varphi_t(t)\Vert^2_{H^2(\mathbb T)}=\Vert\varphi_t(0)\Vert^2_{H^2(\mathbb T)}\,\\
&\lim\limits_{t\to 0^+}\int_{\mathbb T}(\mu-2\phi_{x}(t))\vert\langle\partial_x\rangle^2\varphi_x(t)\vert^2\,dx= \int_{\mathbb T}(\mu-2\phi_{x}(0))\vert\langle\partial_x\rangle^2\varphi_x(0)\vert^2\,dx\,.
\end{split}
\end{equation}

The weak convergence $\varphi_t(t) \rightharpoonup\varphi_t(0)$ $w-H^2(\mathbb T)$ and the convergence of the norms $\lim\limits_{t\to 0^+}\Vert\varphi_t(t)\Vert^2_{H^2(\mathbb T)}=\Vert\varphi_t(0)\Vert^2_{H^2(\mathbb T)}$ (see \eqref{lim_phi}) imply immediately that $\varphi_t(t) \rightarrow\varphi_t(0)$ $s-H^2(\mathbb T)$ (see for instance \cite[Proposition III.30]{brezis}).  Since $\varphi$ has zero mean, in order to prove that $\varphi(t) \rightarrow\varphi(0)$ $s-H^3(\mathbb T)$, it remains only to prove that $\varphi_x(t) \rightarrow\varphi_x(0)$ $s-H^2(\mathbb T)$. This comes from the convergence
\begin{equation*}
\lim\limits_{t\to 0^+}\int_{\mathbb T}(\mu-2\phi_{x}(t))\vert\langle\partial_x\rangle^2\varphi_x(t)\vert^2\,dx= \int_{\mathbb T}(\mu-2\phi_{x}(0))\vert\langle\partial_x\rangle^2\varphi_x(0)\vert^2\,dx\,
\end{equation*}
as a simple applications of Lemma \ref{weak/strong}, where the role of $v_m$ and $u_m$ is here played by $\mu-2\phi_x(t)$ and $\langle\partial_x\rangle^2\varphi_x(t)$, respectively.
All the assumptions in Lemma \ref{weak/strong} are verified up to a subsequence.

\section{Some results from abstract functional analysis}\label{Lions-Magenes_sect}
In this section, we assume that $X$ and $Y$ are reflexive Banach spaces such that
\begin{equation*}
X\hookrightarrow Y\,,\quad\mbox{with continuous embedding}\,.
\end{equation*}
The space $W(X,Y)$ is defined as
\begin{equation}\label{W}
W(X,Y):=\left\{ u\in L^2(0,T; X)\,:\,\,u_t\in L^2(0,T; Y)\right\}\,.
\end{equation}
For every $\theta\in[0,1]$, let the {\it intermediate space} $\left[X,Y\right]_{\theta}$ be defined as in \cite{lionsmagenes1}.
\begin{lemma}\label{lemma1}(\cite{lionsmagenes3})
The following inclusion holds with continuous embedding:
\begin{equation*}
L^\infty(0,T; X)\cap C([0,T]; Y)\hookrightarrow C_w([0,T]; X)\,.
\end{equation*}
\end{lemma}
\begin{lemma}\label{lemma2}(\cite[Theorem 3.1]{lionsmagenes1})
For $Z=\left[X,Y\right]_{1/2}$ we have
\begin{equation*}
W(X,Y)\hookrightarrow C([0,T]; Z)\hookrightarrow C([0,T]; Y)\,\quad\mbox{with continuous embedding}\,.
\end{equation*}
\end{lemma}

\vspace{.20cm}
\noindent
\paragraph{\bf Acknowledgements}
The authors should like to express their thanks to Professors Alessandro Giacomini and Riccarda Rossi for the helpful discussions and their valuable comments.

\end{document}